\documentclass[aos]{imsart}

\RequirePackage{amsthm,amsmath,amsfonts,amssymb}
\RequirePackage{natbib}
\RequirePackage[colorlinks,citecolor=blue,urlcolor=blue]{hyperref}
\RequirePackage{graphicx}
\usepackage{xcolor}
\usepackage{bm}
\usepackage[justification=centering]{caption}
\usepackage{subfigure}

\usepackage{xr}
\startlocaldefs
\theoremstyle{plain} 
\newtheorem{thm}{Theorem}[section]
\newtheorem{lem}[thm]{Lemma}
\newtheorem{prop}[thm]{Proposition}
\newtheorem{cor}[thm]{Corollary}

\theoremstyle{remark}
\newtheorem{defn}[thm]{Definition}
\newtheorem{conj}[thm]{Conjecture}
\newtheorem{ex}[thm]{Example}

\newtheorem{rem}[thm]{Remark}

\newcommand{\cA}{\mathcal{A}}
\newcommand{\cC}{\mathcal{C}}

\newcommand{\E}{\mathbb{E}}

\newcommand{\cG}{\mathcal{G}}

\newcommand{\cI}{\mathcal{I}}

\newcommand{\R}{\mathbb{R}}

\newcommand{\cU}{\mathcal{U}}

\newcommand{\Z}{\mathbb{Z}}

\newcommand{\cV}{\mathcal{V}}

\newcommand{\bs}{\boldsymbol}

\newcommand{\<}{\langle}
\renewcommand{\>}{\rangle}

\endlocaldefs

\begin{document}

\begin{frontmatter}
\title{Non-Independent Components Analysis}
%\title{A sample article title with some additional note\thanksref{t1}}
\runtitle{Non-Independent Components Analysis}
%\thankstext{T1}{A sample additional note to the title.}

\begin{aug}
%%%%%%%%%%%%%%%%%%%%%%%%%%%%%%%%%%%%%%%%%%%%%%%
%% Only one address is permitted per author. %%
%% Only division, organization and e-mail is %%
%% included in the address.                  %%
%% Additional information can be included in %%
\author[A]{\fnms{Geert}~\snm{Mesters}\ead[label=e1]{geert.mesters@upf.edu}\orcid{0000-0001-6996-9520}},
\author[B]{\fnms{Piotr}~\snm{Zwiernik}\ead[label=e2]{piotr.zwiernik@utoronto.ca}\orcid{0000-0003-3431-131X}}
%%%%%%%%%%%%%%%%%%%%%%%%%%%%%%%%%%%%%%%%%%%%%%
%% Addresses                                %%
%%%%%%%%%%%%%%%%%%%%%%%%%%%%%%%%%%%%%%%%%%%%%%
\address[A]{Department of Economics and Business,
Universitat Pompeu Fabra\printead[presep={,\ }]{e1}}

\address[B]{Department of Statistical Sciences,
University or Toronto\printead[presep={,\ }]{e2}}
\end{aug}

\begin{abstract}
A seminal result in the ICA literature states that for $AY =  \varepsilon$, if the components of $\varepsilon$ are independent and at most one is Gaussian, then $A$ is identified up to sign and permutation of its rows \citep{C94}. In this paper we study to which extent the independence assumption can be relaxed by replacing it with restrictions on higher order moment or cumulant tensors of $\varepsilon$. We document new conditions that establish identification for several non-independent component models, e.g. common variance models, and propose efficient estimation methods based on the identification results. We show that in situations where independence cannot be assumed the efficiency gains can be significant relative to methods that rely on independence. 
\end{abstract}

\begin{keyword}[class=MSC]
\kwd[Primary ]{15A69}
\kwd{62H99}
\kwd[; secondary ]{}
\end{keyword}

\begin{keyword}
\kwd{independent component analysis}
\kwd{identifiability}
\kwd{cumulants}
\kwd{tensors}
\end{keyword}

\end{frontmatter}
%%%%%%%%%%%%%%%%%%%%%%%%%%%%%%%%%%%%%%%%%%%%%%
%% Please use \tableofcontents for articles %%
%% with 50 pages and more                   %%
%%%%%%%%%%%%%%%%%%%%%%%%%%%%%%%%%%%%%%%%%%%%%%
%\tableofcontents

\section{Introduction}

\noindent Consider the linear system 
\begin{equation}\label{nica}
	AY=\varepsilon~, 
\end{equation}
where  $Y\in \R^d$ is observed, $A\in \R^{d\times d}$ is invertible, and $\varepsilon$ is a mean-zero hidden random vector with uncorrelated components. If $\varepsilon$ is standard Gaussian, or more generally spherical, then the distribution of $Y$ can identify $A$ only up to orthogonal transformations. In contrast, if the components of $\varepsilon$ are mutually independent and at least $d-1$ are non-Gaussian, then $A$ can be identified up to permutation and sign transformations of its rows \citep{C94}. This result follows from the Darmois-Skitovich theorem \citep{Darmois.53,Skitivic.53} and forms the building block of the vast literature on independent components analysis (ICA) \citep[e.g.][]{KHO01,ComonJutten2010handbook,H13}.   

As implied by its name, the working assumption in the ICA literature is that the components of $\varepsilon$ are independent. For some applications this is an important starting principle as the interest is explicitly in recovering the independent components, see for instance the cocktail party problem described in \citet[][p. 148]{KHO01}. However, in other applications, where the interest is solely in recovering $A$, the independence assumption is not a crucial starting point and can in fact be restrictive as the distribution of $Y$ may not admit a linear transformation that leads to independent components \citep[e.g.][]{hyvarinen2001topographic,MT17}. 

To this extent, in this paper we study assumptions that (i) relax the independence assumption yet (ii) assure the identifiability of the matrix $A$ from observations of $Y$. We generally normalize ${\rm var}(\varepsilon)=I_d$ which implies that ${\rm var}(Y)=(A'A)^{-1}$ and narrows down the identification problem to the compact set $\Omega\;=\;\{QA: Q\in {\rm O}(d)\}$, where ${\rm O}(d)$ is the set of $d$-dimensional orthogonal matrices. This refinement allows to formally state our research question: Which higher order restrictions on $\varepsilon$ allow to identify a finite, possibly structured, subset of $\Omega$? 

We systematically study our question by considering different restrictions on the higher order moments or cumulants of $\varepsilon$. We focus on cases where a subset of entries of a given $r$th order moment/cumulant tensor are set to zero, for some $r > 2$. Although there are alternative types of restrictions that can be considered, zero restrictions are attractive as they can often be justified by generative models for $\varepsilon$ such as the common variance, scale-elliptical and mean-independent component models introduced in Section~\ref{sec:motivation}. Additionally, zero restrictions often arise naturally from subject specific knowledge, see \cite{Bekaert2021} for examples from economics.

We provide two classes of higher order restrictions that (a) identify the set of signed permutation matrices and (b) strictly relax the identification assumptions of \cite{C94}. 

First, we consider the class where the off-diagonal elements of a given moment or cumulant tensor are all zero. Such off-diagonal restrictions are often adopted for estimation in the ICA literature under the independence assumption. For instance the JADE algorithm of \cite{CardosoSouloumiac1993} is based on diagonalizing the fourth order cumulant tensor.  
We show that, without imposing the independence assumption, if we set the off-diagonal elements of \emph{any} $r$th order moment or cumulant tensor to zero we obtain sufficient identifying restrictions to pin down $Q$ up to sign and permutation. We point out that for $r=3,4$ similar results are shown for moment restrictions in \cite{Guay.21} and \cite{velasco2020identification} using a different proof strategy, which does not generalize to higher $r$.

Second, while off-diagonal zero restrictions are commonly adopted, they cannot always be used when the components of $\varepsilon$ are not independent. For instance, if $\varepsilon$ follows a symmetric distribution the odd order tensors are all zero and provide no restrictions, but the even order tensors may not be diagonal as is the case, for instance, when the errors have common stochastic variance \citep[e.g.][]{hyvarinen2001topographic,OleaPlagborg2022}. This motivates our second class of tensor restrictions, which we refer to as reflectionally invariant restrictions, where the only non-zero tensor entries are those where each index appears even number of times. This provides a strict relaxation of the diagonal tensor assumption and we show that this assumption remains sufficient to identify $Q$ up to sign and permutation.  

Overall, diagonal and reflectionally invariant restrictions are most relevant for practical purposes, as efficient estimation methods can be easily implemented based on such identifying assumptions. Moreover, these restrictions allow to identify several specific non-independent components models, such as those with common variance components, scale-elliptical errors, and mean independent errors, for which previously no identification results existed. 

With identification established we turn to estimation. For moment restrictions we note that generalized moment estimators \citep{Hansen1982} are attractive as they are (i) easy to implement and (ii) semi-parametrically efficient in settings where the only known features of the model are the moment restrictions \citep[e.g.][]{Chamberlain1987}. We extend this class by also allowing for cumulant restrictions. The resulting class of higher order based minimum distance estimators is large and includes existing tensor based estimators for model \eqref{nica}, such as JADE \citep{CardosoSouloumiac1993}, as special cases, but also introduces new estimators. We show that estimators in this class are consistent and asymptotically normal under standard regularity conditions. 

%The ICA literature is large and for a comprehensive review we refer to \cite{KHO01}, \cite{ComonJutten2010handbook} and \cite{H13}. Overall, our main contribution is to study the identification problem for model \eqref{nica} when the components of $\varepsilon$ are not independent. 

Our starting observation --- independent components may not exists --- is not new. In fact, such concerns were common in the early literature on Blind Source Separation, see \citet[][Chapter 1]{ComonJutten2010handbook} for an illuminating discussion, and they motivated explicit tests for the existence of independent components \cite[e.g.][]{MT17,DN22}. In addition, the possible absence of independent components motivated the usage of alternative identifying restrictions. For instance, a large literature has explored the usage of time/frequency characteristics of non-stationary components for identification \citep[e.g.][Chapter 11]{ComonJutten2010handbook}. In the current paper we do not exploit non-stationarity for identification.  

There exists numerous methods for estimation and inference in independent components models: e.g. cumulant and moment based methods \citep{Cardoso1989,CardosoSouloumiac1993,Cardoso1999,H99,LanneLuoto2021,DW21}, kernel methods \cite{BJ02}, maximum likelihood methods \cite{chenbickel,SamworthYuan_2012,LM21} and rank based methods \cite{IlmonenPaindaveine2011,HallinMehta2015}.   Based on our new identification results these methods could be modified to relax the independence assumption. We perform this task for moment and cumulant based estimation methods, but clearly other methods could be modified as well. For moment estimators a well developed general inference theory exists, see \cite{Hall2005} for a textbook treatment. For cumulant based estimators less work has been done. A notable exception is found for measurement error models where cumulant based estimators have been developed in \cite{Geary1941} and \cite{Erickson2014}. The difference in their setting is that the parameters of interest can be written as a linear function of the higher order cumulants of the observables. For model \eqref{nica} this is not possible. 

The remainder of this paper is organized as follows. Section~\ref{sec:motivation} provides motivating examples where independent components do not exist. Section~\ref{sec:basiccum} defines some tensor notation and reviews relevant existing results. The general problem that we study is introduced in Section~\ref{sec:simplecontraints}. The new identification results are discussed in Section~\ref{sec:mainres}. Inference is discussed in Section~\ref{sec:inference} followed by some numerical results in Section~\ref{sec:sims}. Any references to sections, equations, lemmas etc. which start with ``S'' refer to the supplementary material.

\section{Examples of non-independent component models}\label{sec:motivation}

Independent component analysis assumes that the components of the latent vector $\varepsilon$ are completely independent. In this section we introduce a few examples of popular generative models for which the independence assumption is violated. Below we revisit these examples to show that these models do satisfy weaker higher order tensor restrictions based on which we can establish the identification of $A$ up to permutation and sign.

\subsection{Common variance components models}\label{sec:commonv} Consider 
\begin{equation}\label{eq:commonvariance} 
	AY \;=\; \varepsilon~, \qquad \text{with} \qquad \varepsilon\;=\;\tau \eta~, 
\end{equation}
where $\tau$ is some positive random variable wth finite second moment and $\eta$ is a random vector that is independent of $\tau$ and such that $\E(\eta)=0$ and ${\rm var}(\eta)=I_d$. In this situation
$$
{\rm var}(\varepsilon)\;=\;{\rm var}(\E(\varepsilon|\tau))+\E({\rm var}(\varepsilon|\tau))\;=\;\E(\tau^2)I_d
$$
and so the entries of $\varepsilon$ remain uncorrelated. However, even if the components of $\eta$ are independent, the components of $\varepsilon$ are generally not. Indeed, assuming $\eta$ has independent components, we have
$$
\E(\varepsilon_i^2\varepsilon_j^2)\;=\;\E(\tau^4)\E(\eta_i^2)\E(\eta_j^2)\;=\;\E(\tau^4)
$$
and
$$
\E(\varepsilon_i^2)\E(\varepsilon_j^2)=\E(\tau^2)^2\E(\eta_i^2)\E(\eta_j^2).$$
Thus $\E(\varepsilon_i^2\varepsilon_j^2)\neq \E(\varepsilon_i^2)\E(\varepsilon_j^2)$, unless ${\rm var}(\tau)=0$, and $A$ cannot be identified using the standard ICA assumptions.

In the ICA literature common variance models are one of the motivating examples for topographic ICA (TICA) \cite{hyvarinen2001topographic}, which can be used in image analysis \cite{meyer2003topographic,meyer2004model}, among others. Further, in finance the variances of stock returns and other financial assets often depend on common components, see \cite{AsaiMcAleerYu.06} for a review of the literature. And while ICA has been applied in this context \citep[e.g.][]{BackWaigend.97} the presence of common volatility limits its credibility. Finally, in macroeconomics there is also strong empirical evidence for common volatility structures \citep[e.g.][]{LudvigsonMaNg.21}.  

The non-independence for the baseline common variance model \eqref{eq:commonvariance} carries over to more general models, where $\tau$ becomes a random vector. For instance, let $K \in \R^{d \times m}$ be a fixed loading matrix, then a more general common variance model reads 
\begin{equation}\label{eq:generalcommonvariance}
	AY \;=\; \varepsilon~, \quad \text{with} \quad \varepsilon\;=\;\tau \odot \eta \quad \text{and} \quad \tau\;=\;\phi(KZ)~, 
\end{equation}
where $\eta\in \R^d $ and $Z\in \R^m$ are  independent random vectors with independent components, the function $\phi:\R\to  \R_{>0} = \{x \in \R : x > 0 \}$, is applied coordinatewise and $\odot$ denotes the Hadamard product. In this model $\tau\in \R^m$ is independent of $\eta$, and  the components of $\varepsilon$ share a common variance if they load on the same underlying components, or factors, $Z$. By exactly the same argument as above $\varepsilon$ has uncorrelated components and, generally,  $\E(\varepsilon^2_i\varepsilon_j^2)\neq \E(\varepsilon_i^2)\E(\varepsilon_j^2)$ and the ICA identification result does not apply.

Formal theoretical identifiability results for $A$ along the lines of \cite{C94} and \cite{ErikssonKoivunen.03} have not been developed for common variance models. Section~\ref{sec:mainres} provides such results.

\subsection{Scale elliptical components models}  

Suppose that in the common variance model \eqref{eq:commonvariance} the error $\varepsilon = \tau \eta$ satisfies $\eta=U \sim \cU_d$, where $\cU_d$ is the uniform distribution on the $d$-sphere. In this case the components of $U$ are no longer independent and $\varepsilon$ is said to follow an elliptical distribution \citep{kelker1970distribution}. It follows that, generally, $\varepsilon$ will not have independent components. The exception is the case where $\tau^2$ follows a $\chi^2_d$ distribution, such that $\varepsilon$ is standard normal. 

In general, for elliptical errors $A$ can never be recovered beyond the set $\Omega\;=\;\{QA: Q\in {\rm O}(d)\}$. This follows directly because the distribution of a spherical random vector is invariant under the orthogonal transformations.  This limitation has been pointed out already e.g. in \cite{palmer2007modeling} who modify the Gaussian distribution to a distribution that is not rotationally invariant. 

Here we follow the general idea of \cite{forbes2014new} and generalize the elliptical distribution by defining $\varepsilon=\tau\odot U$ with $\tau$ a $d$-dimensional vector. Such \emph{multiple scale elliptical} distribution continues to have non-independent components but any variation in the components of $\tau$ will allow us to identify $A$ using the higher order moment/cumulant restrictions introduced below. Moreover, this distribution has the attractive property that it allows to model different tail behavior in different dimensions; e.g. Gaussian in one dimension and Cauchy in another (see the discussion in \cite{azzalini2008robust}).

Formally, the multiple scale elliptical components model is given  
\begin{equation}\label{eq:generalmultielliptical}
	AY \;=\; \varepsilon~, \quad \text{with} \quad \varepsilon\;=\;\tau \odot U \quad \text{and} \quad U \sim \cU_d~, 
\end{equation}
with $\tau \in \R^d$ and $U$ independent. The term \emph{elliptical components analysis} was coined in \cite{han2018eca}, but their interest was in recovering the top eigenvectors of $A A'$ with $A = (a_1 , \ldots, a_d)'$. Also, in the same model \cite{vogel2011elliptical} and \cite{rossell2021dependence} studied recovering $AA'$. These works were motivated by the robustness of the elliptical distribution and its ability to preserve much of the dependence relationships offered by the Gaussian distribution. 

In contrast, we are interested in recovering the full matrix $A$ while taking advantage of the robustness of the elliptical distribution and the multiple scale generalization. However, as mentioned, in this case $\varepsilon$ does not have independent components and we cannot rely on the classical ICA identification result. To circumvent this, Section~\ref{sec:mainres} establishes new identification results for $A$ in the multiple scale elliptical components model \eqref{eq:generalmultielliptical}. 

\subsection{Mean independent component models} \label{sec:MI}

Besides specific generative models or probabilistic models, we can also define non-independent components models more directly by relaxing the strict independence requirement. Consider the following definition of a mean independent component model 
\begin{equation}\label{eq:meanindependence}
	a'_iY \;=\; \varepsilon_i~, \quad \text{with} \quad \mathbb E(\varepsilon_i | \varepsilon_{-i} ) = 0~, \quad \text{for} \ \  i=1, \ldots,d~,  
\end{equation}
where $\varepsilon_{-i}$ drops $\varepsilon_i$ from $\varepsilon$. This relaxed notion of independence is attractive in practice as it avoids restricting terms like $\mathbb E(\varepsilon^k_i | \varepsilon_{-i})$ for $k=2,3,\ldots$, and for $\mathbb E(\varepsilon_i | \varepsilon_{-i}) = 0$ there often exist subject specific knowledge that can be used to justify the restriction. 

To give a concrete example, suppose that $Y = (q, p)'$ where $q$ is the quantity demanded of a good and $p$ its price. In a baseline econometric model $\varepsilon_1$ and $\varepsilon_2$ are then known as demand and supply shocks \citep[e.g][Chapter 3]{Hayashi.00}. Economic theory generally suggests that demand and supply shocks should not be able to predict each other, i.e. $\mathbb E(\varepsilon_1 | \varepsilon_2) = 0$ and vice versa. At the same time, restrictions of the form $\mathbb E(\varepsilon_1^k | \varepsilon_2) = 0$, for $k=2,3,\ldots$ are typically not motivated by economic theory. Additional economic motivation for specific moment conditions in supply and demand models is given in \cite{Bekaert2021,Bekaert2022}. 

In Section~\ref{sec:mainres}  we provide new identification results for the mean independent component model \eqref{eq:meanindependence}. The supplementary material Section~\ref{appsec:addmotivation} provides additional motivation for non-independent components models.

\section{Basic tensor notation}\label{sec:basiccum}

Consider the random vector $X=(X_1,\ldots,X_d)'$ and let $M_X(\bs t) = \E e^{\bs t'X}$ and $K_X(\bs t)=\log \E e^{\bs t'X}$ denote the corresponding moment and cumulant generating functions, respectively. We write $\mu_r(X )$ to denote the $r$-order $d\times \cdots\times d$ moment tensor, that is an $r$-dimensional table whose $(i_1,\ldots,i_r)$-th entry is 
$$
\mu_r(X)_{i_1\cdots i_r}\;=\;\E X_{i_1}\cdots X_{i_r}\;=\;\frac{\partial^{r} }{\partial t_{i_1}\cdots \partial t_{i_r}} M_X(\bs t)\Big|_{\bs t=0} ~.
$$ 
Similarly, the cumulant tensor $\kappa_r(X)$ is defined as 
$$
\kappa_r(X)_{i_1\cdots i_r}\;=\;{\rm cum}(X_{i_1},\ldots,X_{i_r})\;=\;\frac{\partial^{r} }{\partial t_{i_1}\cdots \partial t_{i_r}} K_X(\bs t)\Big|_{\bs t=0} ~.
$$ 
We have $\kappa_1(X)=\mu_1(X)$, $\kappa_2(X)=\mu_2(X) - \mu_1(X)\mu'_1(X)$ and $\kappa_3(X)$ is a $d\times d \times d$ tensor filled with the third order central moments of $X$. The relationship between $\mu_r(X)$ and $\kappa_r(X)$ for higher order $r$ is more cumbersome but very well understood \cite{speed1983cumulants,mccullagh2018tensor}; see the supplementary material ~\ref{app:cumulcomb}. 
Directly by construction, $\mu_r(X)$ and $\kappa_r(X)$ are  symmetric tensors, i.e. they are invariant under an arbitrary permutation of the indices. The space of real symmetric $d\times \cdots\times d$ order $r$ tensors is denoted by $S^r(\R^d)$. Writing $[d]=\{1,\ldots,d\}$, the set of indices of an order $r$ tensor is $[d]^r$. However, $S^r(\R^d)\subset \R^{d\times\cdots \times d}$ has dimension $\binom{d+r-1}{r}$ and the unique entries of $T\in S^r(\R^d)$ are $T_{i_1\cdots i_r}$ for $1\leq i_1\leq \ldots\leq i_r\leq d$. 

The vast majority of results in this paper holds for both moment and cumulant tensors. To avoid excessive notation we denote a given $r$th order moment or cumulant tensor by $h_r(X)$. Whenever distinguishing between moments or cumulants is required we specify towards $\mu_r(X)$ or $\kappa_r(X)$. 

A critical feature of moment and cumulant tensors that we use to study identification in model \eqref{nica} comes from multilinearity, i.e. for every $A\in \R^{d\times d}$ we have  
\begin{equation}\label{eq:Amu}
	h_r(AX)=A\bullet  h_r(X),
\end{equation} 
where $A\bullet T$ for $T\in S^r(\R^d)$ denotes the standard multilinear action
$$
(A\bullet T)_{i_1 \cdots i_r}\;=\;\sum_{j_1=1}^d\cdots \sum_{j_r=1}^d A_{i_1 j_1}\cdots A_{i_r j_r}T_{j_1\cdots j_r}
$$
for all $(i_1,\ldots,i_r)\in [d]^r$, see, for example, Section~2.3 in \cite{zwiernik2016semialgebraic}.

Since $A\bullet T\in S^{r}(\R^d)$ for all $T\in S^r(\R^d)$ we say that $A\in \R^{d\times d}$ \emph{acts} on $S^r(\R^d)$. The notation $A\bullet T$ is a special case of a general notation for multilinear transformations $\R^{n_1\times \cdots \times n_r}\to \R^{m_1\times\cdots \times m_r}$ given by matrices $A^{(1)}\in \R^{m_1\times n_1}$, \ldots, $A^{(r)}\in \R^{m_r\times n_r}$:
\begin{equation}\label{eq:genermlin}
	[(A^{(1)},\ldots,A^{(r)})\cdot T]_{i_1\cdots i_r}\;=\;\sum_{j_1=1}^{n_1}\cdots \sum_{j_r=1}^{n_r}A^{(1)}_{i_1 j_1}\cdots A^{(r)}_{i_r j_r} T_{j_1\cdots j_r}.
\end{equation}
See, for example \cite{lim2021tensors} for an overview of the computational aspects of tensors. 

\begin{rem}\label{rem:multilin}
	The multilinearity property \eqref{eq:Amu} is not exclusive to moments and cumulant tensors, as central moments, free cumulants and boolean cumulants, for instance, also share this property; see \citet[][Section 5.2]{Zwiernik2012} for a more complete characterization. Our main results rely only on the property \eqref{eq:Amu} and so the definition of $h_r(X)$ can be extended beyond moments and cumulants if needed. 
\end{rem}

The following well-known characterization of independence is of importance in our work.

\begin{prop}\label{prop:indep}
	The components of $X$ are independent if and only if $\kappa_r(X)$ is a diagonal tensor for every $r\geq 2$. 
\end{prop} 
This result highlights that the necessity of the independence assumption in ICA can be investigated by studying the consequences of making appropriate higher order cumulant tensors elements non-zero. The relationship to the Gaussian distribution can be understood from a version of the Marcinkiewicz classical result \cite{marcinkiewicz1939propriete,lukacs1958some}.
\begin{prop}\label{prop:marcinkiewicz}
	If $X\sim \mathcal N_d(\mu,\Sigma)$ then $\kappa_1(X)=\mu$, $\kappa_2(X)=\Sigma$, and $\kappa_r(X)=\bs 0$ for $r\geq 3$. Moreover, the Gaussian distribution is the only probability distribution such that there exists $r_0$ with the property that $\kappa_r(X)=\bs 0$ for all $r\geq r_0$.
\end{prop}
As we formalize below, this result implies that we require deviations from the Gaussian distribution to ensure identification in model \eqref{nica}, similar as required in the classical ICA result \citep{C94}.

\section{Identification with zero constraints}\label{sec:simplecontraints} 

Since $AY=\varepsilon$ with $\E \varepsilon=0$ and ${\rm var}(\varepsilon)=I_d$, the variance of $Y$ satisfies ${\rm var}(Y)=(A'A)^{-1}$ and so it is enough to narrow down potential candidates for $A$ to the compact set
$$\Omega \;:=\; \{QA: Q\in {\rm O}(d)\}~.$$
Our main insight is as follows: Since $\varepsilon$ is unobserved, multiplying \eqref{nica} by $Q\in {\rm O}(d)$ gives an alternative representation $\widetilde A Y=\tilde \varepsilon$, where $\E \tilde \varepsilon=0$ and ${\rm var}(\tilde \varepsilon)=I_d$. The goal is to define suitable additional restrictions on the distribution of $\varepsilon$ so that the distribution of $\tilde\varepsilon=Q\varepsilon$ does not satisfy these restrictions unless $Q$ is very special. The main result of \citep{C94} proposes to use non-Gaussianity and independence. We show how to exploit additional structure in some $h_r(X)$ to obtain similar results. 

\subsection{Exploiting general constraints}

Suppose that we have some additional information about a fixed higher-order tensor $T=h_r(\varepsilon)\in S^r(\R^d)$, for example we know that $T \in \cV$ for some subset $\cV\subseteq S^r(\R^d)$. By multilinearity \eqref{eq:Amu} we have 
\begin{equation}\label{eq:kappaA}
	T=h_r(AY)=A\bullet h_r(Y)~,	
\end{equation}
and for any given $Q\in {\rm O}(d)$,  $QA\in \Omega$ remains a valid candidate if
\begin{equation}\label{eq:kappaA2}
	(QA)\bullet h_r(Y)\;\in\; \cV~.
\end{equation}
However,
$$
(QA)\bullet h_r(Y)\;=\;Q\bullet (A\bullet h_r(Y))\;=\;Q\bullet T
$$
and so \eqref{eq:kappaA} and \eqref{eq:kappaA2} hold together if and only if $Q\bullet T\in \cV$. For $T \in \cV$, we define
\begin{equation}\label{eq:GT}
	\cG_{T}(\cV)\;:=\;\{Q\in {\rm O}(d): Q\bullet T \in \cV\} ~, 
\end{equation}
which is the subset of $\Omega$ that can be identified from $\cV$. 
Below we sometimes drop $\cV$, writing $\cG_T$, if the context is clear.
We always have $I_d\in \cG_T(\cV)$ but in general $\cG_T(\cV)$ will be larger. 

We summarize the general identification problem as follows.  
\begin{prop}\label{prop:mainreduction}
	Consider the model \eqref{nica} with $\E \varepsilon=0$ and ${\rm var}(\varepsilon)=I_d$.  Suppose we know, for a fixed $r\geq 3$, that $T=h_r(\varepsilon)\in \cV\subset S^r(\R^d)$. Then $A$ can be identified up to the set 
	\begin{equation}\label{eq:Omega0}
		\Omega_0\;=\; \{QA: Q\in \cG_T(\cV)\}.
	\end{equation} 
\end{prop}	
In the ideal situation $\cG_T(\cV)$ is a singleton, in which case $A$ can be recovered exactly. But we also expect that, in general, exact recovery will not be possible. We are therefore looking for restrictions $\cV$ that assure that $\cG_T(\cV)$ is a finite set, possibly with some additional structure. The leading structure of interest is the set of signed permutations for which we recover the original ICA result under strictly weaker assumptions. We denote the set of $d\times d$ signed permutation matrices by ${\rm SP}(d)$. These are the $2^d d!$ matrices that are of the form $D P$, where $D,P\in {\rm O}(d)$ with $D$ diagonal and $P$ a permutation matrix.

\subsection{Zero restrictions}\label{ssec:zeroexamples} 

Clearly, there exists a plethora of restrictions on the higher order moment or cumulants that can be considered. For instance, the ICA assumption imposes that $\kappa_r(\varepsilon) = {\rm cum}_r(\varepsilon)$ has zero off-diagonal elements for all $r$ (i.e.  Proposition~\ref{prop:indep}). At the same time we know from the discussion in Section~\ref{sec:motivation} that several generative models do not satisfy these restrictions and hence we seek relaxations that accommodate such models yet still yield identification of $A$. 

We formalize zero restrictions by choosing {a subset $\cI$ of $r$-tuples $(i_1,\ldots,i_r)$ satisfying $1\leq i_1\leq\cdots\leq i_r\leq d$ }and by defining the vector space $\cV=\cV(\cI)$ of symmetric tensors $T\in S^r(\R^d)$ such that $T_{\bs i} =0$ for all $\bs i=(i_1,\ldots,i_r) \in \cI$. In symbols:
$$
\cV\;=\;\cV(\cI)\;=\;\{T\in S^r(\R^d):\;T_{\bs i} =0\mbox{ for }\bs i \in \cI\}~.
$$
Note that the codimension of $\cV$ in $S^r(\R^d)$ is precisely ${\rm codim}(\cV)=|\cI|$.

The following example clarifies our notation and illustrates how higher order moment or cumulant restrictions can be used for identification.
\begin{ex}
	Suppose that $\cV\subset S^3(\R^2)$ is given by $T_{112}=T_{122}=0$. This is a two-dimensional subspace parametrized by $T_{111}$ and $T_{222}$. The condition $Q\bullet T\in \cV$ is given by the system of two cubic equations in the entries of $Q$
	\begin{gather*}
		Q_{11}^2Q_{21}T_{111}+Q_{12}^2Q_{22}T_{222}\;=\;0\\
		Q_{11}Q_{21}^2T_{111}+Q_{12}Q_{22}^2T_{222}\;=\;0.
	\end{gather*}
	In a matrix form this can be written as
	$$
	Q \cdot\begin{bmatrix}
		Q_{11} & 0\\
		0 & Q_{22}
	\end{bmatrix}\cdot \begin{bmatrix}
		Q_{21} & 0\\
		0 & Q_{12}
	\end{bmatrix}\cdot \begin{bmatrix}
		T_{111} \\
		T_{222} 
	\end{bmatrix}\;=\;\begin{bmatrix}
		0 \\
		0
	\end{bmatrix}.
	$$
	Since $Q$ is orthogonal, each of the two diagonal matrices above is either identically zero or it is invertible. If it is identically zero then $Q$ must be a sign permutation matrix and the equation clearly holds. If they are both invertible we immediately see that the equation cannot hold unless $T_{111}=T_{222}=0$, in which case $T$ is the zero tensor showing that for every nonzero $T\in \cV$ we have that $\cG_T(\cV)={\rm SP}(2)$. 
\end{ex}

Unfortunately, the direct arguments that we used in this example to determine $\cG_T(\cV)$ do not generalize for higher $r$ and $d$. Handling such cases requires a more systematic approach which we develop in Section~\ref{sec:mainres}. 

\begin{rem}\label{rem:constants}
	For exposition purposes we only consider cases where some entries of $T = h_r(\varepsilon)$ are set to zero, but we note that our results can be extended for cases where entries of $T$ are non-zero but \emph{known} to the researcher. An example with 4th order moment restrictions arises when $\E \varepsilon = 0$, ${\rm var}(\varepsilon) = I_d$ and $T_{iijj} = \E \varepsilon_i^2 \varepsilon_j^2 = 1$ for $i \neq j$. 
\end{rem}

\section{Identification up to Sign and Permutation}\label{sec:mainres}

In this section we discuss specific sets of zero restrictions that allow to identify $A$ up to sign and permutation. For each specific set of zero restrictions we give concrete examples of models that can be identified using such restrictions. 

\subsection{Diagonal tensors}

Denote by $T =h_r(\varepsilon)$ the $r$th order moment or cumulant tensor of $\varepsilon$. A simple assumption that facilitates identification is that $T$ is a diagonal tensor. \begin{defn}\label{def:diagonal}
	A tensor $T\in S^r(\R^d)$ is called diagonal if it has entries $T_{\bs i}=0$ unless $\bs i=(i,\ldots,i)$ for some $i=1,\ldots,d$. We define $\cV^{\rm diag}$ as the set of diagonal tensors in $S^r(\R^d)$. 
\end{defn}

Of course, if the components of $\varepsilon$ are independent then $\kappa_r(\varepsilon)$ is diagonal for all $r \geq 2$ (see Proposition~\ref{prop:indep}). Assuming that $T$ is diagonal is much less restrictive than full independence as any $T$ can be chosen without imposing restrictions on other cumulants, or moments. This allows for instance to assume that only the cross-third moments of $\varepsilon$ are zero, without imposing any restrictions on the higher order moments. 

For verifying whether $\cV^{\rm diag}$ provides sufficient identifying restrictions we will study the tensors $T$ and $Q \bullet T$ via their associated homogeneous polynomials in variables $x=(x_1,\ldots,x_d)$. We have   
\begin{equation}\label{eq:fT}
	f_T(x)=\sum_{i_1=1}^d\cdots \sum_{i_r=1}^d T_{i_1\ldots i_r}x_{i_1}\ldots x_{i_r}\;=\;\sum_{\bs i} T_{\bs i}(x^{\otimes r})_{\bs i}\;=\;\<T,x^{\otimes r}\>,
\end{equation}
where $x^{\otimes r}\in S^r(\R^d)$ denotes the tensor with coordinates 
$
(x^{\otimes r})_{i_1\cdots i_r}\;=\;x_{i_1}\cdots x_{i_r}.
$
If $r=2$ then $T$ is a symmetric matrix and $f_T(x)=x' T x$ is the standard quadratic form associated with $T$. 

\begin{lem}\label{lem:fQT}
	If $T\in S^r(\R^d)$ and $A\in \R^{d\times d}$ then  $f_{A\bullet T}(x)	=f_T(A'x)$. Moreover, $\nabla f_{A\bullet T}=A\nabla f_T(A'x)$ and $\nabla^2 f_{A\bullet T}=A\nabla^2 f_T(A'x) A'$.
\end{lem}
\begin{proof}
	The first claim follows because
	$$f_{A\bullet T}(x)\;=\;\<A\bullet T,x^{\otimes r}\>\;=\;\< T,(A'x)^{\otimes r}\>\;=\;f_T(A'x).	
	$$
	The second claim is then a direct check. 
\end{proof}
This will be useful for deriving our first main result. 
\begin{thm}\label{th:diagonal}
	Let $T \in S^r(\R^d)$ for $r\geq 3$ be a diagonal tensor satisfying 
	\begin{equation}\label{eq:genericitydiagonal}
		T_{i\ldots i} \neq 0 \qquad \text{for at least $d-1$ different values of $i=1,\ldots,d$ }~. 
	\end{equation}	
	Then $Q\bullet T\in \cV^{\rm diag}$ if and only if $Q\in {\rm SP}(d)$, i.e. $\cG_{T}(\cV^{\rm diag}) = SP(d)$. If $r=2$, the same conclusion holds under a stronger condition that $T_{jj}\neq T_{kk}$ for $j\neq k$.
\end{thm}
\begin{proof}
	The left direction is straightforward. For the right direction, note that the tensor $T$ is diagonal if and only if $\nabla^2 f_T(x)$ is a diagonal polynomial matrix. By Lemma~\ref{lem:fQT}, we have $f_{Q\bullet T}(x)=f_T(Q'x)$ and
	$$
	\nabla^2 f_{Q\bullet T}(x)\;=\;Q\nabla^2 f_T(Q'x) Q'.
	$$ 
	Thus, $Q\bullet T$ is diagonal if and only if $Q\nabla^2 f_T(Q'x) Q'=D(x)$ for a diagonal matrix $D(x)$. Equivalently, for every $i,j$
	$$
	Q_{ij} \frac{\partial^2}{\partial x_j^2} f_T(Q'x)=D_{ii}(x)Q_{ij},
	$$
	where we also used the fact that $\nabla^2 f_T(x)$ is a diagonal matrix. 	If each row of $Q$ has exactly one non-zero entry then $Q\in {\rm SP}(d)$ and we are done. So suppose $Q_{ij},Q_{ik}\neq 0$. Then, by the above equation
	$$
	\frac{\partial^2}{\partial x_j^2} f_T(Q'x)\;=\;D_{ii}(x)\;=\;\frac{\partial^2}{\partial x_k^2} f_T(Q'x).
	$$
	This is an equality of polynomial functions and thus, equivalently, $\frac{\partial^2}{\partial x_j^2} f_T(x)=\frac{\partial^2}{\partial x_k^2} f_T(x)$, which simply states that 
	$$
	T_{j\cdots j} x_j^{r-2}\;=\;T_{k\cdots k} x_k^{r-2}.
	$$
	If $r\geq 3$, this equality can hold only if  $T_{j\cdots j}=T_{k\cdots k}=0$, which is impossible by the genericity condition \eqref{eq:genericitydiagonal}. If $r=2$ this cannot hold by the stronger condition $T_{jj}\neq T_{kk}$ for $j\neq k$. %our genericity assumption. 
\end{proof}

\begin{rem}
	The genericity condition is a necessary condition. Indeed, if, for example $T_{1\cdots 1}=T_{2\cdots 2}=0$ then $\frac{\partial}{\partial x_1^2} f_T(x)=\frac{\partial^2}{\partial x_2^2} f_T(x)=0$. Thus, $\nabla^2 f_{Q\bullet T}(x)$ is diagonal for any block matrix of the form
	$$
	Q\;=\;\begin{bmatrix}
		Q_0 & 0\\
		0 & I_{d-2}
	\end{bmatrix}
	$$ 
	where $Q_0\in O(2)$ is an orthogonal matrix. The family of such matrices is infinite. 
\end{rem}

Combining Proposition~\ref{prop:mainreduction} and Theorem~\ref{th:diagonal} implies the following result.
\begin{thm}\label{thm:diagonalnica}
	Consider the model \eqref{nica} with $\E \varepsilon =0$, ${\rm var}(\varepsilon)=I_d$ and suppose that for some $r\geq 3$ the tensor $h_r(\varepsilon)$ is diagonal with at most one zero on the diagonal. Then $A$ in \eqref{nica} is identifiable up to permuting and swapping signs of its rows.  
\end{thm} 

\begin{rem}
	Note that if $\varepsilon$ is standard Gaussian then all higher order cumulants vanish. All odd-order moments vanish too. Even-order moment tensors are not zero but they are also not diagonal. 
\end{rem}

Subsequently, based on Theorem \ref{thm:diagonalnica} we can provide an identification result for the common variance and mean independent component models that were discussed in Section~\ref{sec:motivation}. 
\begin{cor}\label{cor:maindiag}
	Consider the model $AY = \varepsilon$ with $\mathbb E \varepsilon = 0$ and ${\rm var}(\varepsilon) = I_d$. Suppose that additionally $\varepsilon$ satisfies one of the following conditions. 
	\begin{itemize}
		\item[(a)] $\varepsilon = \tau \odot \eta$, $\tau\;=\;\phi(KZ)$, where $K \in \R^{d \times m}$ is a fixed matrix, $\eta\in \R^d $ and $Z\in \R^m$ are independent random vectors with independent components and the function $\phi:\R\to \R_{>0}$, is applied coordinatewise.
		\item[(b)] $\mathbb E(\varepsilon_i | \varepsilon_{-i}) = 0$ for $i=1,\ldots,d$.  
	\end{itemize}
	Then $h_3(\varepsilon) = \mu_3(\varepsilon) = \kappa_3(\varepsilon)$ is a diagonal tensor. If additionally, the genericity condition \eqref{eq:genericitydiagonal} holds  
	%$h_r(\varepsilon)$ has at most one zero on the diagonal 
	then $A$ is identifiable up to permuting and swapping signs of its rows.
\end{cor}

\begin{proof}
	The proof is based on Theorem \ref{thm:diagonalnica}.  For (a) note that $h_3(\varepsilon)$ is diagonal as $\mathbb E \varepsilon_i \varepsilon_j \varepsilon_k = \mathbb E \tau_i \tau_j \tau_k \mathbb E \eta_i \eta_j \eta_k = 0$ unless $i=j=k$. For (b), consider the triple $(i,j,k)$. Unless $i=j=k$, there will be at least one element that appears only once. Without loss of generality assume $i\neq j$ and $i\neq k$. We have $$\mathbb E \varepsilon_i \varepsilon_j \varepsilon_k=\E(\mathbb E(\varepsilon_i \varepsilon_j \varepsilon_k|\varepsilon_{-i}))=\E(\varepsilon_j \varepsilon_k\mathbb E(\varepsilon_i |\varepsilon_{-i}))=0$$
	again confirming that the third order moment/cumulant tensor is diagonal.	 
\end{proof}
Our result for diagonal tensors cannot be used for the scaled elliptical distributions in \eqref{eq:generalmultielliptical}. In this case all odd-order moments/cumulants are zero (not generic) and the even-order moment/cumulant tensors are not diagonal. This motivates our next section.

\subsection{Reflectionally invariant tensors}\label{sec:reflectionally}

In some applications the assumption that $T$ is diagonal may be unattractive. A leading example is the scale elliptical components models where the third order tensors are zero but the fourth order tensor is not diagonal as entries of the form $T_{iijj}$ cannot be restricted to zero (or some other constant). Note that the latter zero restriction is also invalid in the common variance model \eqref{eq:generalcommonvariance}. 

These observations motivate the following tensor restrictions. 
\begin{defn}\label{def:reflectionally}
	A tensor $T\in S^r(\R^d)$ is called \emph{reflectionally invariant} if the only potentially non-zero entries in $T$ are the entries $T_{i_1\cdots i_r}$ where each index appears in the sequence $(i_1,\ldots, i_r)$ even number of times. If $r$ is odd, the only reflectionally invariant tensor is the zero tensor. We define $\cV^{\rm refl}$ as the set of reflectionally invariant tensors in $S^r(\R^d)$. 
\end{defn}

To prove that reflectionally invariant tensors can be used to identify $A$ in \eqref{nica}, 
recall from \eqref{eq:fT} that any $T\in S^r(\R^d)$ has an associated homogeneous polynomial $f_T(x)$ of order $r$ in $x=(x_1,\ldots,x_d)$. 
It %is clear 
 follows
from the definition that a non-zero $T\in S^r(\R^d)$ is reflectionally invariant if and only if $r$ is even and there is a homogeneous polynomial $g_T$ of order $l:=r/2$ such that $f_T(x)=g_T(x_1^2,\ldots,x_d^2)$.  We have the following useful characterization of reflectionally invariant tensors. 
\begin{lem}\label{lem:reflinvcond}
	The tensor $T\in S^r(\R^d)$ is reflectionally invariant if and only if $f_T(x)=f_T(Dx)$ for every diagonal matrix $D$ with $\pm 1$ on the diagonal.
\end{lem}
\begin{proof}
	By Lemma~\ref{lem:fQT}, $f_T(x)=f_T(Dx)$ is equivalent to saying that $D\bullet T=T$ for every diagonal %$D\in \Z^d_2$. 
	$D$ with $\pm 1$ on the diagonal.
	If $T$ is reflectionally invariant then $$f_T(Dx)=g_T(D_{11}^2x_1^2,\ldots,D_{dd}^2x_d^2)=f_T(x)~,$$
	which establishes the right implication. For the left implication note that  $f_T(x)=f_T(Dx)$, for each $D$, implies that $f_T$ does not depend on the signs of the components of $x$. Since this is a polynomial, we must be able to write it in the form $g_T(x_1^2,\ldots,x_d^2)$ {(this is obvious in one dimension and, in general, can be proved in each dimension separately)}. This is equivalent with $T$ being reflectionally invariant.
\end{proof}

%\noindent In the theorem below, for a tensor $T \in S^r(\R^d)$ we use the notation 
%$$T_{+\cdots +ij}\;:=\;\sum_{i_1=1}^d\cdots \sum_{i_{r-2}=1}^d T_{i_1\cdots i_{r-2}ij}.$$
%{\color{red} In words, $T_{+\cdots +ij}$ is ... .}

\begin{thm}\label{th:reflinv}
	Suppose that $T\in S^r(\R^d)$ for an even $r$ is a reflectionally invariant tensor. Let $l=(r-2)/2$ and suppose, in addition, that $T$ satisfies 
	\begin{equation}\label{eq:genrefl}
		\sum_{i_1,\ldots,i_{l}} T_{i_1i_1\cdots i_l i_l jj }\neq \sum_{i_1,\ldots,i_{l}} T_{i_1i_1\cdots i_l i_l kk }\quad\mbox{for all }j\neq k.
	\end{equation}
	Then $Q\bullet T\in \cV^{\rm refl}$ if and only if $Q\in {\rm SP}(d)$, i.e. $\cG_{T}(\cV^{\rm refl}) = SP(d)$.
\end{thm}

\begin{rem}
	We emphasize that the genericity condition in \eqref{eq:genrefl} simply states that $T$ lies outside of $\binom{d}{2}$ explicit linear hyperplanes in $S^r(\R^d)$. It is interesting to observe how the genericity condition evolves when zero restrictions are relaxed. First, in the classical ICA result \citep{C94} the condition is that for each $i = 1, \ldots,d$ the corresponding diagonal entries of $h_r(\varepsilon)$ across $r$ cannot all be zero. The diagonal tensor identification result (Theorem \ref{thm:diagonalnica}) replaces this condition by the requirement that at most one diagonal entry for \emph{a given} $h_r(\varepsilon)$ can be zero. Finally, the reflectional invariant condition \eqref{eq:genrefl} extends the condition to the specific $\binom{d}{2}$ hyperplanes of $S^r(\R^d)$ which include the previous genericity conditions as isolated points. 
\end{rem}
\begin{rem}
	If $r=2$ then $\cV^{{\rm refl}}=\cV^{{\rm diag}}$ and it corresponds to the set of diagonal matrices. In this case, the genericity condition \eqref{eq:genrefl} and the genericity condition in Theorem~\ref{th:diagonal} coincide: $T_{jj}\neq T_{kk}$ for $j\neq k$.
\end{rem}
Theorem~\ref{th:reflinv} is proven using the following lemma. 
\begin{lem}\label{lem:mainreflinv}
	Let $r$ be even and suppose that $T\in S^r(\R^d)$ is a reflectionally invariant tensor satisfying  \eqref{eq:genrefl}. Then $Q\bullet T=T$ for $Q\in {\rm O}(d)$ if and only if $Q$ is a diagonal matrix    with $\pm 1$ on the main diagonal. \end{lem}
\begin{proof}
	%The left implication is clear because $f_T(x)=f_{T}(Dx)=f_{D\bullet T}(x)$ by Lemma~\ref{lem:reflinvcond}. 
	 For the left implication assume that $Q$ is diagonal with $\pm 1$ on the diagonal. Then $f_{Q \bullet T}(x) = f_T(Q'x) = f_T(x)$, where the first equality follows from Lemma~\ref{lem:fQT} and the second from Lemma~\ref{lem:reflinvcond} as $T$ is reflectionally invariant. By result $Q \bullet T = T$.
	
	We prove the right implication by induction. The base case is $r=2$, where the set of reflectionally invariant tensors corresponds to diagonal matrices. In this case the equation $Q\bullet T=T$ becomes $Q TQ'=T$ or, equivalently, $Q T =TQ$. This implies that for each $1\leq i\leq j\leq d$
	$$
	Q_{ij}T_{jj}=T_{ii}Q_{ij}.
	$$
	By the genericity condition \eqref{eq:genrefl},  all the diagonal entries of the matrix $T$ are distinct. In this case, for every $i\neq j$, we necessarily have $Q_{ij}=0$. Proving that $Q$ must be diagonal. Note also that this genericity condition is necessary: If two diagonal entries of $T$ are equal, then the entries of the $2\times 2$ submatrix $Q_{ij,ij}$ are not constrained, so $Q$ does not have to be diagonal.

	Suppose now that the claim is true for $r\geq 2$ and let $T\in S^{r+2}(\R^d)$ with $Q\bullet T=T$. Rewrite $Q\bullet T=T$, using the general multilinear notation \eqref{eq:genermlin}, as $Q\bullet T = (Q, \ldots, Q ) \cdot T$. Now note that since $Q\bullet T=T$ we can replace any $Q$ by $I_d$. We have $(Q, \ldots, Q , I_d, I_d ) \cdot T = T$ where we replaced the last two $Q$'s by $I_d$. Now acting on both sides with $(I_d,\ldots,I_d,Q’,Q’)$ and using $Q'Q = I_d$ gives 
	\begin{equation}\label{eq:reflinvqqproof}
		(Q,\ldots,Q,I_d,I_d)\cdot T\;=\;(I_d,\ldots, I_d,Q',Q')\cdot T~. 
	\end{equation}
	We want to show that this equality implies that $Q$ is a diagonal matrix   with $\pm 1$ on the main diagonal. Let $\bs i=(i_1,\ldots,i_r)$ and consider all $(r+2)$-tuples $(\bs i,u,u)$ for some $u\in \{1,\ldots,d\}$. Writing \eqref{eq:reflinvqqproof} restricted to these indices gives
	$$
	\sum_{j_1,\ldots,j_r}Q_{i_1j_j}\cdots Q_{i_rj_r} T_{j_1\cdots j_r u u}\;=\;\sum_{j_{r+1},j_{r+2}} Q_{j_{r+1}u}Q_{j_{r+2}u}T_{i_1\cdots i_r j_{r+1} j_{r+2}}. 
	$$
	Now sum both sides over all $u=1,\ldots,d$. Using the fact that $Q$ is orthogonal we get that $\sum_u Q_{j_{r+1}u}Q_{j_{r+2}u}$ is zero if $j_{r+1}\neq j_{r+2}$ and it is $1$ if $j_{r+1}= j_{r+2}$. Denoting $S_{\bs i}=\sum_u T_{\bs i u u}$,  summation over $u$ yields
	$$
	\sum_{j_1,\ldots,j_r}Q_{i_1j_j}\cdots Q_{i_rj_r} S_{j_1\cdots j_r}\;=\;\sum_{v} T_{i_1\cdots i_r vv}\;=\;S_{i_1\cdots i_r}. 
	$$
	Since this equation holds for every $\bs i=(i_1,\ldots,i_r)$, we conclude $Q\bullet S=S$, where $S \in S^r(\R^d)$ with entries $S_{\bs i} = S_{i_1\cdots i_r}$.   %, where $S=(S_{\bs i})\in S^r(\R^d)$. 
	
	Note however that $S$ is a reflectionally invariant tensor. Indeed, if  some index appears in $\bs i$ odd number of times then $T_{\bs i uu} = 0$ for all $u = 1, \ldots, d$ as $T$ is reflectionally invariant. By definition, $S_{\bs i}$ is then also zero.
	
	%$S_{\bs i}=T_{\bs i uu}=0$ as the same index appears in $(\bs i, u, u)$ odd number of times. 

	Since $T$ satisfies \eqref{eq:genrefl}, $S$ satisfies \eqref{eq:genrefl} too. Indeed, for $l = (r-2)/2$ we have 
	$$
	\sum_{i_1,\ldots,i_{l}}S_{i_1i_1\cdots i_{l}i_{l}jj}=\sum_{i_1,\ldots,i_{l}}\sum_{i_{l+1}}T_{i_1i_1\cdots i_{l}i_{l}i_{l+1} i_{l+1} jj}
	$$
	and so these quantities are distinct for all $j=1,\ldots,d$ by assumption on $T$. 

	Now, by the induction assumption, we conclude that $Q$ is diagonal. 
\end{proof}

\begin{proof}[Proof of Theorem~\ref{th:reflinv}]
	The left implication is straightforward. %clear. 
	For the right implication, suppose $Q\in {\rm O}(d)$ is such that $Q\bullet T$ is reflectionally invariant. By Lemma~\ref{lem:reflinvcond}, equivalently, $f_{Q\bullet T}(x)=f_{Q\bullet T}(Dx)$ for every diagonal $D\in {\rm O}(d)$, which gives $f_T(Q'x)=f_T(Q'D x)$. This polynomial equation implies that 
	$$
	f_T(x)=f_T(Q'D Q x)
	$$
	but since $Q'D Q\in O(d)$, Lemma~\ref{lem:mainreflinv} implies that $\bar D=Q'D Q$ must be diagonal. Therefore, the equation $DQ=Q\bar D$ shows that switching the signs in the $i$-th row of $Q$ is equivalent to switching some columns of $Q$. Suppose that there are at least two non-zero entries $Q_{ik}$, $Q_{il}$ in the $i$-th row of $Q$ and let $D$ be such that $D_{ii}=-1$ and $D_{jj}=1$ for $j\neq i$. The equality $D Q=Q\bar D$ requires that $\bar D_{kk}=\bar D_{ll}=-1$ and that $Q$ has no other non-zero entries in $k$-th and $l$-th columns. Since these columns are orthogonal we get a contradiction concluding that the $i$-th row of $Q$ must contain at most (and so exactly) one non-zero entry. Applying this to each $i=1,\ldots,d$, we conclude that $Q\in {\rm SP}(d)$.
\end{proof}

Combining Proposition~\ref{prop:mainreduction} and Theorem~\ref{th:reflinv} implies the following result.
\begin{thm}\label{thm:reflexctionnica}
	Consider the model \eqref{nica} with $\E \varepsilon =0$, ${\rm var}(\varepsilon)=I_d$ and suppose that for some even $r$ the tensor $h_r(\varepsilon)$ is reflectionally invariant and it satisfies the genericity condition \eqref{eq:genrefl}. Then $A$ is identifiable up to permuting and swapping signs of its rows.   
\end{thm}   
Using Theorem \ref{thm:reflexctionnica} we can provide identification results for all the models from Section~\ref{sec:motivation}. We summarize these all in the following statement. 
\begin{cor}\label{cor:mainreflec}
	Consider the model $AY = \varepsilon$ with $\mathbb E(\varepsilon) = 0$ and ${\rm var}(\varepsilon) = I_d$. If $\varepsilon$ follows the general common covariance model \eqref{eq:generalcommonvariance} then $h_4(\varepsilon)$ is reflectionally invariant. If $\varepsilon$ follows the the multiple scaled elliptical distribution \eqref{eq:generalmultielliptical}, then $h_r(\varepsilon)$ is reflectionally invariant for every {even} $r\geq 4$. If $\varepsilon$ is mean independent as in \eqref{eq:meanindependence} then $h_4(\varepsilon)$ is reflectionally invariant. 
	If additionally, the genericity condition \eqref{eq:genrefl} holds then $A$ is identifiable up to permuting and swapping signs of its rows.
\end{cor}
\begin{proof}
	For the first statement consider any element of $h_4(\varepsilon)$ such that one index appears odd number of times. Then we can assume it appears exactly once. So let $j,k,l\neq i$ and then 
	$$
	\E(\varepsilon_i \varepsilon_j \varepsilon_k \varepsilon_l )=\E(\eta_i)\E(\eta_j\eta_k \eta_l \tau_i\tau_j\tau_k \tau_l)=0.
	$$
	Similar calculations hold for cumulants. For the second statement, observe that if $D_i$ is the diagonal matrix with $-1$ on the $(i,i)$-th entry and $1$ on the remaining diagonal entries, then 
	\begin{equation}\label{eq:etasign}
		D_i (\tau\odot U)\;=\;\tau\odot (D_i U)\;\overset{d}{=}\;\tau\odot U,
	\end{equation}
	where $U,\tau$ is like in \eqref{eq:generalmultielliptical}. This assures invariance of the distribution (and so also the moments/cumulants) with respect to sign swapping of single coordinates. By Lemma~\ref{lem:reflinvcond}, all moment/cumulant tensors must be reflectionally invariant. For the third statement we start as for the first. So let $j,k,l\neq i$ and then
	$$
	\E(\varepsilon_i \varepsilon_j \varepsilon_k \varepsilon_l )=\E( \varepsilon_j \varepsilon_k \varepsilon_l\E(\varepsilon_i|\varepsilon_{-i}) )=0
	$$
	with similar calculations for cumulants.
\end{proof}

Note that if $\varepsilon$ is standard normal (which arises as a degenerate case of both \eqref{eq:generalcommonvariance}  and \eqref{eq:generalmultielliptical}) then the even-order moment tensors are non-zero reflectionally invariant tensors. However, they are not generic in the sense of \eqref{eq:genrefl}. For example, if $r=4$ then $\E \varepsilon_i^4=3$ and $\E \varepsilon_i^2\varepsilon_j^2=1$ for $i\neq j$. This gives $\sum_{i} T_{i i jj }=d+2$ for all $i=1,\ldots,d$.
%For example, if $r=4$ then $\E \varepsilon_i^4=3$ and $\E \varepsilon_i^2\varepsilon_j^2=1$ for $i\neq j$. This gives $T_{++ii}=d+2$ for all $i=1,\ldots,d$. 

\begin{rem}
	The genericity condition \eqref{eq:genrefl} can be worked out explicitly for each model. For instance, for the general common variance model (a) we have for $\mu_4(\varepsilon)$ that  
	\[
	\sum_{l=1}^d \mathbb E(\tau_l^2 \tau_j^2 ) \mathbb E(\eta_l^2 \eta_j^2 )  \neq \sum_{l=1}^d \mathbb E(\tau_l^2 \tau_k^2 ) \mathbb E(\eta_l^2 \eta_k^2 )  \qquad \text{for all } j \neq k ~. 
	\] 
	This simplifies in the case for the simple common variance model \eqref{eq:commonvariance} to become $\mathbb E(\eta_j^4) \neq \mathbb E(\eta_k^4)$ for all $j \neq k$.  In the multiple scaled elliptical model we can exploit the summetry in the moments of $\eta\sim \cU_d$ to show that $\mu_4(\varepsilon)$ is generic as long as 
	$$
	2\E \tau_j^4+\sum_{l\neq j} \E(\tau_j^2\tau_l^2)\;\neq\;2\E \tau_k^4+\sum_{l\neq k} \E(\tau_k^2\tau_l^2)\qquad\mbox{for all }j\neq k.
	$$ \end{rem}

\subsection{Generalizations} \label{sec:conjectures}

Theorems \ref{thm:diagonalnica} and \ref{thm:reflexctionnica} highlight key zero moment and cumulant patterns that can be used to identify $A$ up to sign and permutation for the general model \eqref{nica}. Such restrictions are equally sufficient for identification in the class of linear simultaneous equations models $A Y = BX +\varepsilon$ when $X$ is exogenous, and various dynamic extensions of such models \citep[e.g.][]{KilianLutkepohl.17}. 

That said it is also of interest to explore whether relaxing additional zero restrictions still leads to identification (up to sign and permutation), and which genericity conditions are required. Since $\dim({\rm O}(d))=\binom{d}{2}$, we need at least that many constraints to assure $\cG_T$ is finite. However, as we formally show in the supplementary material Section~\ref{appsec:localident} the minimal set 
$$ \cI\;=\;\{(i,j,\ldots,j):\;\;1\leq i<j\leq d\} ~, $$
implies that $\cG_T$ is finite, but in general $\cG_T \neq {\rm SP}(d)$. Example \ref{ex:simple3} explicitly computes the difference between $\cG_T$ and ${\rm SP}(d)$ for the illustrative case with $r=3$ and $d=2$. This finding has important implication that it is, in general, not sufficient to prove that the Jacobian of the moment or cumulant restrictions is full rank in order to establish that the identified set is equal to the set of signed permutations. 

Motivated by these calculations, consider a special model with 
$$
\cI\;=\;\{(i,j,\ldots,j): 1\leq i<j\leq d\}\cup \{(i,\ldots,i,j): 1\leq i<j\leq d\}~.
$$ 
\begin{conj}\label{conj:onezero}
	If $T$ is a generic tensor in $\cV(\cI)$ then $Q\bullet T\in \cV(\cI)$ if and only if $Q\in {\rm SP}(d)$.
\end{conj}
The case when $d=2$ is very special because ${\rm O}(2)$ has dimension 1. In this case the analysis of zero patterns can be often done using classical algebraic geometry tools. In particular, we can show that the conjecture holds for $S^r(\R^2)$ tensors for any $r\geq 3$.
\begin{prop}\label{prop:binaryMI}
	Suppose that $T\in S^r(\R^2)$ satisfies $T_{12\cdots 2}=T_{1\cdots 12}=0$ but is otherwise generic. Then $Q\bullet T\in \cV(\cI)$ if and only if $Q\in {\rm SP}(2)$.    
\end{prop}
We prove this result in Supplement~\ref{appsec:mainproofs}. The genericity conditions are again linear and can be recovered from the proof.
\begin{rem}
	Consider the two-dimensional supply and demand model described in Section~\ref{sec:MI}. Proposition~\ref{prop:binaryMI} assures that the matrix $A$ can be identified up to scaling and swapping rows from $h_r(\varepsilon)$ for any $r\geq 3$ as long as it satisfies the genericity conditions. 
\end{rem}

\subsubsection{Towards point identification} 

As stated, the results so far provided identification results for $A$ up to the set of signed-permutation matrices. In practice, it is often of interest to reduce this set further to, perhaps, a singleton. Restricting additional higher order tensors to zero cannot help with this objective, as even the most stringent selection of zero restrictions, i.e. all higher order tensors are diagonal (the independent components case), only yields identification up to signed permutations. 

Therefore, we briefly mention a few existing routes that different strands of literature have adopted for further shrinking the identified set. First, topographical ICA, which was motivated by the common variance model, suggests to explicitly model the common variance structure, i.e. model $K$ in \eqref{eq:generalcommonvariance}. It then imposes the additional assumption that only nearby latent components have the higher order dependency in order to pin down a unique permutation \citep[e.g.][]{hyvarinen2001topographic}. Second and related, \cite{Shimizu.06} impose the additional assumption that there exists a permutation $A$ that is lower triangular. This restriction further pins down the identified set up to sign changes. Moreover, the implied directed acyclic graphical model has a causal interpretation. Third, in econometrics sign restrictions on $A$ are a popular tool to weed out economically uninteresting permutations of $A$. A canonical case arises when model \eqref{nica} represents a demand and supply equation (cf Section~\ref{sec:MI}), in which case it is natural to impose that the demand equation has a downward slope and the supply equation an upward slope. Together with the normalization that the scales on the errors are positive yields a unique permutation. Such schemes can be generalized for larger values of $d$.

\section{Inference for non-independent components models}\label{sec:inference}

Given our new identification results, there exist numerous possible routes for estimating $A$ in $AY = \varepsilon$ given a sample $\{Y_s \}_{s=1}^n$. A natural approach is based on
the concept of minimum distance estimation, where (i) the identifying zero restrictions from Section~\ref{sec:mainres} are replaced by their sample equivalents and (ii) the parameter $A$ is chosen such that the sample restrictions are as close as possible to zero, i.e. the distance is minimized. When the restrictions are placed on moment tensors $\mu_r(\varepsilon)$ this approach falls in the class of generalized moment estimation (GMM) \citep[e.g.][]{Hansen1982}, see \cite{Hall2005} for a textbook treatment. When the restrictions are placed on cumulant tensors no general framework exists, but a similar route as for moments can be followed. 

It is interesting to note that minimum distance estimators are commonly adopted in the ICA literature using diagonal tensor restrictions and Euclidean distance to measure distance \cite[e.g.][Chapter 11]{KHO01}. For instance, the JADE algorithm of \cite{CardosoSouloumiac1993} solves a minimum distance problem that considers (after pre-whitening) cumulant restrictions on $\kappa_4$. 
In our set-up we follow the GMM literature and measure distance in a statistically meaningful way in order to get optimal efficiency of the associated estimator and based on the results of Section~\ref{sec:reflectionally} we also consider non-diagonal tensor restrictions.  

To set up our approach let $A_0$ denote the true $A$. 
%We fix $\cV=\cV(\cI)$, with $\cI$ defined by either Definition \ref{def:diagonal} or \ref{def:reflectionally}, and let $\pi_\cV$ be defined as the orthogonal projection from $S^r(\R^d)$ to $\cV^\perp$. Note that $\pi_\cV(T)$ simply gives the coordinates $T_{\bs i}$ for $\bs i\in \cI$, i.e. the set of zero restricted higher order tensor entries.
We take $\cV$ as $\cV^{\rm diag}$ or $\cV^{\rm refl}$ as given in Definitions \ref{def:diagonal} and \ref{def:reflectionally}, respectively, and let $\pi_\cV$ be defined as the orthogonal projection from $S^r(\R^d)$ to $\cV^\perp$. Note that $\pi_\cV(T)$ simply gives the coordinates $\bm i$ for which $T_{\bs i} = 0$, i.e. the set of zero restricted higher order tensor entries. For diagonal tensors this is the set of all off-diagonal elements, while for reflectionally invariant tensors this is the set where at least one index appears an uneven number of times.

For $h_r(\varepsilon) \in \cV$ we define the function 
\begin{equation}\label{eq:gst}
	g(A)\;:=\;{\rm vec}_u(A\bullet h_2(Y)-I_d,\;\pi_{\cV}(A\bullet h_r(Y)))\;\in \;\R^{\binom{d+1}{2}+|\cI|}~,
\end{equation}
where ${\rm vec}_u$ is the vectorization that takes the unique entries of an element in $S^2(\R^d)\oplus \cV^\perp$ and stacks them in a vector that has length $d_g = \binom{d+1}{2}+|\cI|$. The sample equivalent of $g(A)$ is given by 
\begin{equation}\label{eq:gstest}
	\hat g_n(A)\;:=\;{\rm vec}_u(A\bullet \hat h_2-I_d,\;\pi_{\cV}(A\bullet \hat h_r))\;\in \;\R^{\binom{d+1}{2}+|\cI|}~,	
\end{equation}
%$\hat g_n(A)$ and replaces $h_j$ by $\hat h_j$, for $j=2,r$, 
where $\hat{h}_j$ denotes either the sample moments, denoted by $\hat \mu_j$, or the $j$th order $\mathsf k$-statistic, denoted by $\mathsf k_j$, which are computed from a given sample $\{Y_s \}_{s=1}^n$. The computation of the sample moments $\hat \mu_j$ requires no explanation and for $\mathsf k$-statistics we refer to \citet[][Chapter 4]{mccullagh2018tensor} as well as the supplement~\ref{appsec:cumulant} where we provide explicit computational formulas.

The population and sample objective functions that we consider are given by
\begin{equation}
	L_{W}(A) \;=\;  \| g(A)\|_W^2\quad \mbox{and}\quad  \hat L_{W}(A) \;=\;  \| \hat g_n(A)\|_W^2,
\end{equation}  
where $W$ is an $d_g \times d_g$ positive definite weighting matrix, $\| v \|_W^2=v'W v$. The following result is straightforward.
\begin{lem}\label{lem:mainrephrase}
	Suppose that \eqref{nica} holds with $h_2(\varepsilon)=I_d$ and $h_r(\varepsilon)$  is a tensor that satisfies the conditions in Theorem \ref{th:diagonal} or \ref{th:reflinv}, then $L_W(A)=0$ if and only if $A=QA_0$ for $Q\in {\rm SP}(d)$. 
\end{lem}
%\begin{proof}
%	We have $L_W(A)=0$ if and only if $g(A)=0$, which is equivalent $A\bullet h_2(Y)=I_d$ and $ A\bullet h_r(Y)\in \cV$. Since \eqref{nica} holds, we also have
%	$A_0\bullet h_2(Y)= I_2$ and $A_0\bullet h_r(Y)\in \cV$. It follows that $A_0^{-1}A\in {\rm O}(d)$, or in other words, $A=QA_0$ for some $Q\in {\rm O}(d)$. Further, 
%	$$
%	A\bullet h_r(Y)=QA_0\bullet h_r(Y)=Q\bullet h_r(\varepsilon)\in \cV,
%	$$
%	which implies that $Q\in \cG_T(\cV)$ and by Theorem \ref{th:diagonal} or \ref{th:reflinv} we have $\cG_T(\cV)={\rm SP}(d)$.
%\end{proof}
The lemma implies that the identifying results from Section \ref{sec:mainres} can be translated into an minimization problem.
Given a sample $\{ Y_s \}_{s=1}^n$, and a sequence of positive semidefinite matrices $W_n$ we define the estimator 
\begin{equation}\label{estimator}
	\widehat{A}_{W_n}\;:=\;\arg\min_{A\in \cA} \hat L_{W_n}(A),
\end{equation}
where $\mathcal A \subseteq GL(d)$ is fixed in advance and $W_n$ is a weighting matrix that may depend on the sample. Here by $\arg\min_{A\in \cA}$ we mean an arbitrarily chosen element from the set of minimizers of $\hat L_{W_n}(A)$. 

\subsection{Consistency}

We can show that this class gives consistent estimates for the true $A_0$ up to sign and permutation. A possible set of conditions is as follows. 
\begin{prop}[Consistency]\label{prop:consist}
	Suppose that $\{ Y_s \}_{s=1}^n$ is i.i.d from model \eqref{nica} and (i) $h_r(\varepsilon)$ satisfies the conditions in Theorem \ref{th:diagonal} or \ref{th:reflinv}, (ii) $\mathcal A \subset GL(d)$ is compact and $Q A_0 \in \mathcal A$ for {some} $Q \in SP(d)$  (iii) $W_n \stackrel{p}{\to } W$ and $W$ is positive definite, (iv) $\mathbb E \|Y_{s}\|^r  < \infty$. Then $\widehat A_{W_n} \stackrel{p}{\to} Q A_0$ as $n \to \infty$ for some $Q \in {\rm SP}(d)$. 	
\end{prop}
%The proof is included in the supplementary material. 

Condition (i) corresponds to the identification assumptions that were derived in the previous section. Condition (ii) imposes that the permutations $Q A_0$ lie in some compact subset $\mathcal A \subset GL(d)$. This can be relaxed at the expense of a more involved derivation for the required uniform law of large numbers. Condition (iii) imposes that the weighting matrix is positive definite and we will determine an optimal choice for $W$ below. The moment condition (iv) is necessary for applying the law of large numbers. 

\subsection{Asymptotic normality}

The weighting matrix $W_n$ can take different forms. In the ICA literature $W_n$ is often taken as the identity matrix \citep[e.g.][Chapter 5]{ComonJutten2010handbook}, but we will show that different choices for $W_n$ yield, at least in theory, more efficient estimates provided that sufficient moments of $Y$ exist. Specifically, when we take $W_n$ such that it is consistent for the inverse of 
\begin{equation}\label{eq:Sigma}
	\Sigma \;=\; \lim_{n\to \infty} {\rm var}(\sqrt{n} \hat g_n(QA_0))    
\end{equation}
we can ensure that the resulting estimate $\widehat A_{W_n}$ achieves minimal variance in the class of generalized cumulant estimators \eqref{estimator}.  

Let $G(A)\in \R^{d_g\times d^2}$ be the Jacobian matrix representing the derivative of the function $g: \R^{d\times d}\to \R^{d_g}$ defined in \eqref{eq:gst}. Here, defining the Jacobian we think about $g$ as a map from $\R^{d^2}$ vectorizing $A$. 
\begin{prop}[Asymptotic normality]\label{prop:assnormal}
	Suppose that the conditions of Proposition~\ref{prop:consist} hold, (v) $Q A_0 \in {\rm int}(\mathcal A)$ for some $Q \in SP(d)$, (vi) $\mathbb E \|Y_i\|^{2r}< \infty$, and denote by $G = G(Q A_0)$.
	%=  \nabla_{{\rm vec}(a)'} m(A)$, 
	Then 
	\begin{equation}\label{eq:assnormal}
		\sqrt{n}{\rm vec}[\widehat A_{W_n} - Q A_0  ] \;\;\stackrel{d}{\to}\;\; N(0, (G' W G)^{-1} G'W \Sigma W G(G'WG )^{-1}  ) 
	\end{equation}
	for some $Q \in {\rm SP}(d)$, where $\Sigma$ is given in \eqref{eq:Sigma}. Moreover, for any $\widehat \Sigma_n \stackrel{p}{\to} \Sigma$ we have that
	\[
	\sqrt{n}{\rm vec}[\widehat A_{\widehat \Sigma_n^{-1}} - Q A_0  ] \stackrel{d}{\to} N(0, S   )~.  
	\]
	for some $Q \in {\rm SP}(d)$ and $S = (G' \Sigma^{-1} G)^{-1}$.
\end{prop}
This result allows for an interesting comparison. For ICA models under full independence an efficient estimation method is developed in \cite{chenbickel}. When we relax the independence assumption, and instead only restrict higher order cumulant entries, the efficient estimator is given by $\widehat A_{\widehat \Sigma_n^{-1}}$. Here efficiency is understood in the sense that $S$ is smaller (in the L\"{o}wner ordering) when compared to the variance in \eqref{eq:assnormal} for any $W$. \cite{Chamberlain1987} shows that for moment restrictions the estimator $\widehat A_{\widehat \Sigma_n^{-1}}$ attains the semi-parametric efficiency bound in the class of non-parametric models characterized by restrictions $T_{\bm i} = 0$ for $\bm{i} \in \cI$.  

Implementing this estimator can be done in different ways. Proposition~\ref{prop:consist} shows that $Q A_0$ can be consistently estimated regardless of the choice of weighting matrix. Given such first stage estimate, using say $W_n = I_{d_g}$, we can estimate $\Sigma$ consistently (under the assumptions of Proposition~\ref{prop:assnormal}). 
With this estimate we can compute $\widehat A_{\widehat \Sigma_n^{-1}}$ from \eqref{estimator}. While this estimate is efficient, the procedure can obviously be iterated until convergence to avoid somewhat arbitrarily stopping at the first iteration, see \cite{HansenLee2021} for additional motivation for iterative moment estimators. Additionally, we may also consider $W_n = \widehat \Sigma_n(A)^{-1}$ as a weighting matrix, hence parametrizing the asymptotic variance estimate as a function of $A$, and minimize the objective function \eqref{estimator} using this weighting matrix \citep[e.g.][]{HansenHeatonYaron1996}. The methodology for estimating $\Sigma$ and $S$, under both moment and cumulant restrictions, is discussed in the supplement \ref{appsec:assvar}. For moment tensors we recommend using a standard plug-in estimator, but for cumulant tensors it often more convenient to adopt a simple bootstrap that resamples the residuals to estimate $\Sigma$ as the analytical form of $\Sigma$ depends on Jacobian of the transformation of cumulants to moments which can be tedious to derive in practice. This residual bootstrap is valid under the assumptions Proposition~\ref{prop:assnormal}, see supplement \ref{appsec:assvar}. In supplement~\ref{appsec:inferencedetails} we also discuss several hypothesis tests that allow us to test the higher order tensor restrictions.

\section{Numerical illustration}\label{sec:sims} 

In this section we evaluate the numerical performance of the minimum distance estimators introduced above. We simulate data from some of the non-independent components models of Section~\ref{sec:motivation} and compare the estimation accuracy of the minimum distance estimators of Section~\ref{sec:inference} to several popular ICA methods.  

{\small{
		\begin{table} 
			\begin{center}
				\caption{\sc Distributions for Errors \ \ \ \ \  \ \ \ \ \ \ \ \ }\label{tab:ng_dist}
				\begin{tabular}{lll}
					\hline\hline 
					Code & Name & Definition\\ \hline 
					$\mathcal N $ & Gaussian & $\frac{1}{\sqrt{2\pi}} \exp \left(-\frac{1}{2} x^2 \right)$ \\ [2ex]
					$t(\nu=5)$ & Student's $t$ & $\frac{\Gamma\left( \frac{\nu+1}{2}\right)}{\sqrt{\nu\pi} \Gamma\left( \frac{\nu}{2} \right)} \left( 1 + \frac{x^2}{\nu}  \right)^{\left( -\frac{\nu+1}{2} \right)} $  \\ [2ex] 
					SKU & Skewed Unimodal & $\frac{1}{5}\mathcal N\big(0,1\big) + \frac{1}{5}\mathcal N\big(\frac{1}{2},(\frac{2}{3})^2\big) + \frac{3}{5}\mathcal N\big(\frac{13}{12},(\frac{5}{9})^2\big)$\\[2ex]
					KU & Kurtotic Unimodal & $\frac{2}{3}\mathcal N\big(0,1\big) + \frac{1}{3}\mathcal N\big(0,(\frac{1}{10})^2\big)$\\[2ex]
					BM & Bimodal & $\frac{1}{2}\mathcal N\big(-1,(\frac{2}{3})^2\big) + \frac{1}{2}\mathcal N\big(1,(\frac{2}{3})^2\big)$\\[2ex]
					SBM & Separated Bimodal & $\frac{1}{2}\mathcal N\big(-\frac{3}{2},(\frac{1}{2})^2\big) + \frac{1}{2}\mathcal N\big(\frac{3}{2},(\frac{1}{2})^2\big)$\\[2ex]
					SKB & Skewed Bimodal & $\frac{3}{4}\mathcal N\big(0,1\big) + \frac{1}{4}\mathcal N\big(\frac{3}{2},(\frac{1}{3})^2\big)$\\[2ex]
					TRI & Trimodal & $\frac{9}{20}\mathcal N\big(-\frac{6}{5},(\frac{3}{5})^2\big) + \frac{9}{20}\mathcal N\big(\frac{6}{5},(\frac{3}{5})^2\big) + \frac{1}{10}\mathcal N\big(0,(\frac{1}{4})^2\big)$\\[2ex]
					CL & Claw & $\frac{1}{2}\mathcal N\big(0,1 \big) + \sum_{l=0}^4 \frac{1}{10}\mathcal N\big(l/2-1,(\frac{1}{10})^2\big) $\\[2ex]
					ACL & Asymmetric Claw & $\frac{1}{2}\mathcal N\big(0,1 \big) + \sum_{l=-2}^2 \frac{2^{1-l}}{31} \mathcal N\big(l+1/2,(2^{-l} / 10)^2\big) $ \\[1ex]
					\hline\hline 
					& & \\ 
					\multicolumn{3}{l}{\begin{minipage}{12.5cm}
							\emph{Notes:} The table reports the distributions that are used in the simulation studies to draw the errors. The mixture distributions are taken from \cite{MW92}, see their table 1.
					\end{minipage}}
				\end{tabular}
			\end{center}
		\end{table}
}}

\subsection{Common variance component models} 

\begin{table}
	\caption{\sc Amari Errors: Common Variance Model}
	\begin{center}
		\begin{tabular}{lcccccccccc} 
			\hline\hline 
			& \multicolumn{10}{c}{Non-Independent Components Analysis} \\[0.5ex]
			Method       & $\mathcal N$ & $t(5)$ & SKU & KU & BM & SBM & SKB & TRI & CL & ACL \\   \hline 
			$\mu_3^{d,I} $   &0.45 &   0.35  &  0.33 &   0.35  &  0.53  &  0.65 &   0.46  &  0.56  &  0.46  &  0.43 \\
			$\mu_3^{d, \widehat \Sigma_n^{-1}} $ & 0.40 &   0.34 &   0.31  &  0.33  &  0.45  &  0.54  &  0.39 &   0.46 &   0.40  &  0.37 \\
			$\mu_4^{r,I}$ & 0.29  &  0.28  &  0.29  &  0.25 &   0.26 &   0.18  &  0.29  &  0.26  &  0.31  &  0.30 \\
			$\mu_4^{r,\widehat \Sigma_n^{-1}}$ &0.30   & 0.28 &   0.30   & 0.25  &  0.24  &  0.12  &  0.27 &   0.22   & 0.28  &  0.29 \\ 
			& & & & & & & & & & \\
			TICA & 0.37  &  0.19  &  0.27 &   0.05  &  0.80   & 0.91  &  0.70  &  0.83  &  0.53  &  0.46 \\ 
			%& & & & & & & & & & \\ 
			\hline\hline 
			& \multicolumn{10}{c}{Independent Components Analysis} \\ [0.5ex]
			Method  	& $\mathcal N$ & $t(5)$ & SKU & KU & BM & SBM & SKB & TRI & CL & ACL \\   \hline  
			Fast  &0.44 &   0.37  &  0.39 &   0.35 &   0.57 &   0.66 &   0.51  &  0.58  &  0.46  &  0.47 \\
			JADE &0.44  &  0.35  &  0.37  &  0.28  &  0.60  &  0.67  &  0.55  &  0.62  &  0.49  &  0.48 \\
			Kernel  & 0.45   & 0.36  &  0.39 &   0.36   & 0.55 &   0.64 &   0.49  &  0.56  &  0.46  &  0.46 \\
			ProDen   &0.45   & 0.34  &  0.39  &  0.31   & 0.60 &   0.66 &   0.54  &  0.61   & 0.50  &  0.49\\	
			Efficient  & 0.44 &   0.48 &   0.46 &   0.48 &   0.40&   0.44 &   0.40   & 0.42  &  0.41 &   0.41 \\
			NPML	 & 0.44  &  0.42  &  0.43  &  0.43 &  0.45  &  0.42 &   0.44  &  0.43   & 0.43 &  0.42 \\	
			\hline\hline 
			& & & & & & & & & & \\
			\multicolumn{11}{l}{ \begin{minipage}{12.5cm} \footnotesize \emph{Notes:} The table reports the average Amari errors (across $S=1000$ simulations) for data sampled from the common variance model \eqref{eq:generalcommonvariance} with $d=2$ and $n=200$. The columns correspond to the different errors considered for the components of $\eta$, see Table \ref{tab:ng_dist}. The top panel reports the errors for the minimum distance methods and Topographical ICA (TICA). For the minimum distance methods we consider diagonal ($d$) and reflectionally invariant ($r$) restrictions for different order tensors $\mu_3, \mu_4$, combined with weighting matrices $W_n = I_d, \widehat \Sigma_n^{-1}$. The bottom panel reports comparison results for different independent component analysis methods: FastICA \citep {Hyvarinen.99}, JADE \cite{CardosoSouloumiac1993}, kernel ICA \citep{BachJordan.03}, ProDenICA \citep{HastieTibshirani.02}, efficient ICA \citep{chenbickel} and non-parametric ML ICA \citep{SamworthYuan_2012}. \end{minipage}  }
		\end{tabular} 
	\end{center}
	\label{tab:estimatesbaselinecv}
\end{table}

We start by simulating independent samples $\{Y^1, \ldots, Y^n\}$ from the common variance model \eqref{eq:generalcommonvariance}, where $K = (1, \ldots, 1)'$, $\tau  \sim {\rm gamma}(1,1)$, $\phi$ is the identity function and $\eta \in \R^d$ has independent components that are simulated from different univariate distributions that are summarized in Table \ref{tab:ng_dist}. The draws for $\eta_i$ are standardized such that $\varepsilon_i = \tau \eta_i$ has mean zero and unit variance. 

The matrix $A$ is defined by $A = R' L$, where $L$ is lower triangular with ones on the main diagonal and zeros elsewhere, and $R$ is a rotation matrix that is parametrized by the Cayley transform of a skewed symmetric matrix who entries are randomly drawn for each sample from a $\mathcal N(0,I_l)$ distribution, with $l = d(d-1)/2$.  

We compare the performance of several estimators. The minimum distance estimators from Section~\ref{sec:inference} are used with either $\kappa_3(\varepsilon) = \mu_3(\varepsilon)$ set to have zero off-diagonal elements, or $\mu_4(\varepsilon)$ restricted to be reflectionally invariant. We consider the weighting matrices $W_n = I_{d_g}$ and the asymptotically optimal choice $W_n =  \widehat \Sigma_n^{-1}$. 

As an alternative non-independent component method we include topographical ICA (TICA) of \cite{hyvarinen2001topographic}. We note that TICA assumes that mapping from the independent components that determines the variance to the errors is known, i.e. $K$ is assumed known and explicitly used in the construction of the objective function \cite[see][equation 3.10]{hyvarinen2001topographic}. The minimum distance estimators that we propose do not exploit this knowledge. 

For comparison purposes we also include FastICA \citep{Hyvarinen.99}, JADE \citep{CardosoSouloumiac1993}, kernel ICA \citep{BachJordan.03}, ProDenICA \citep{HastieTibshirani.02}, efficient ICA \citep{chenbickel} and non-parametric MLE ICA \citep{SamworthYuan_2012}. We stress that none of these alternative methods are designed for non-independent components models and they are merely included to highlight that incorrectly imposing independence leads to distorted estimates. 

For each simulation design we sample $S = 1000$ datasets and measure the accuracy of the estimates using the \cite{Amari.96} error: 
\[
d_A(\widehat A_{W_n} , A_0 ) = \frac{1}{2d} \sum_{j=1}^d \left( \frac{\sum_{j=1}^n |a_{ij}|}{\max_j |a_{ij}|} -1 \right) + \frac{1}{2d} \sum_{j=1}^d \left( \frac{\sum_{i=1}^n |a_{ij}|}{\max_i |a_{ij}|} -1 \right)~, 
\]
where $a_{ij}$ is the $i,j$ element of $A_0\widehat A_{W_n}^{-1}$. We report the averages of this error over the $S$ datasets. In the supplementary material we also show results for the minimum distance index \citep{IlmonenNordhausenOjaOllila.10}.   

Table \ref{tab:estimatesbaselinecv} shows the baseline estimation results for $d=2$ and $n=200$. We find that minimum distance methods based on fourth order reflectionally invariant tensors always perform better when compared to the methods that rely on the independence assumption. The magnitude of the increase in the Amari errors differs across the different choices for the underlying densities. Notably with multi-modal densities the gains in estimation accuracy are large, often reducing the Amari error by more than half. Using the efficient weighting matrix is generally preferable when compared to the identity weighting, although the differences are not always large.  

Minimum distance methods based on third order diagonal moment restrictions perform well when the true density has strong skewness (e.g. SKU). When used with the efficient weighting matrix this minimum distance method nearly always outperform the ICA methods, but even then the reflectionally invariant approach appears preferable. 

As an alternative non-independent component method, TICA works well when the true likelihood is close to the imposed objective function of TICA. Generally, this is the case for all densities that are similar to the Student's $t$ density. When the true density is far from the approximating densities TICA can have very large errors. A similar observation was made in the context of independent component analysis for pseudo maximum likelihood methods in \cite{LM21}.   

In Figure \ref{fig:commonvariance} we show that these findings persist for higher dimensional models and for large sample sizes. We show the results for $\eta_i \sim t(5)$ and $\eta_i \sim BM$. A larger selection of experiments is shown in the supplementary material. Two observations are worth pointing out. First, as the sample size $n$ increases the Amari errors become smaller for the minimum distance methods, supporting the consistency result from Proposition \ref{prop:consist}. In contrast, for the ICA methods (here exemplified by FastICA) there is no change in accuracy when increasing $n$. Second, when the dimensions increase the ordering, in terms of performance, remains similar as above.

\begin{figure}
	\caption{\sc Selected Common Variance Experiments}\label{fig:commonvariance}
	\begin{center}
		\begin{subfigure}{	
				\includegraphics[scale=0.45]{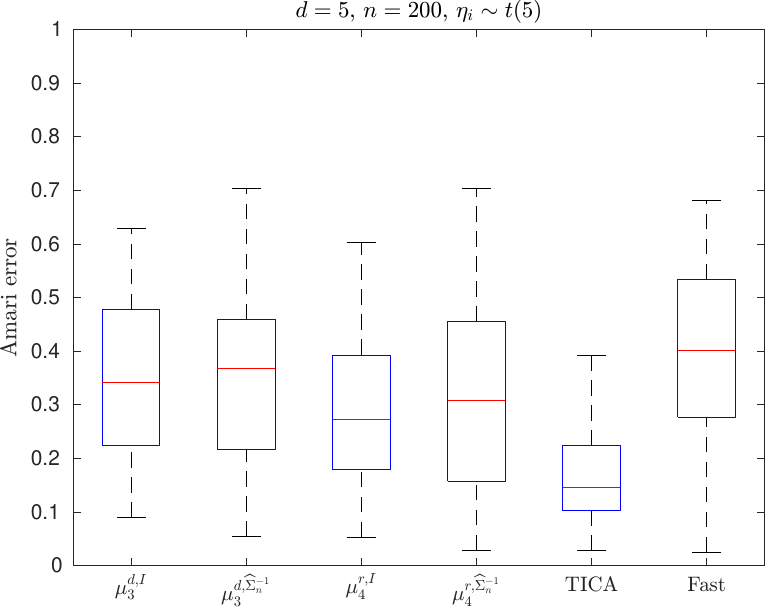}\hfill} 
		\end{subfigure}
			\begin{subfigure}{	
				\includegraphics[scale=0.45]{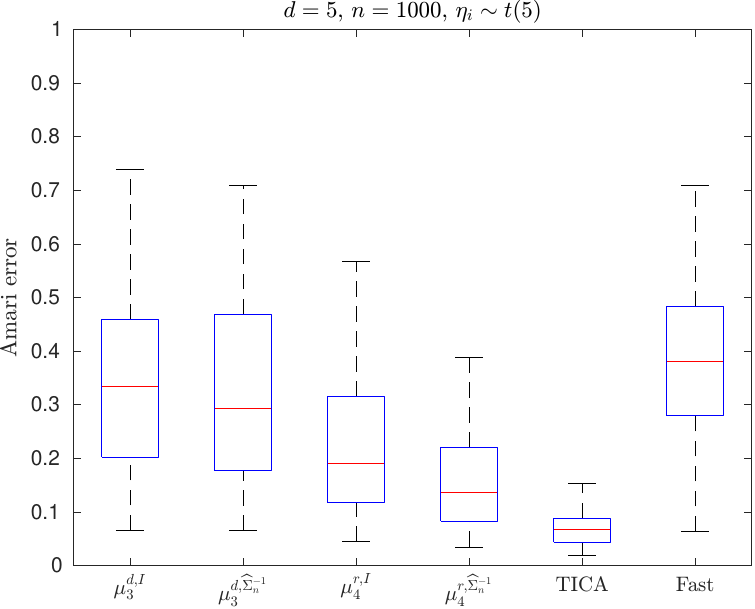}} 
			\end{subfigure}
		\begin{subfigure}{
				\includegraphics[scale=0.45]{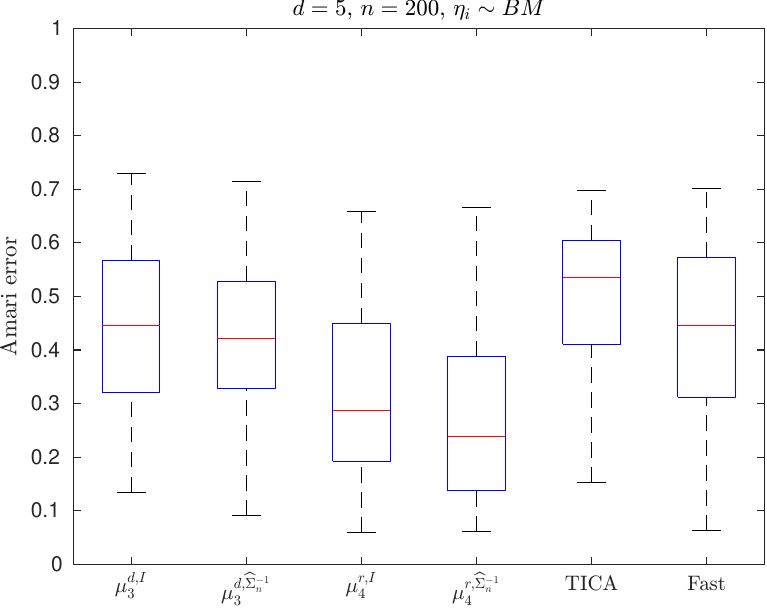}	
			}
		\end{subfigure}
		\begin{subfigure}{
				\includegraphics[scale=0.45]{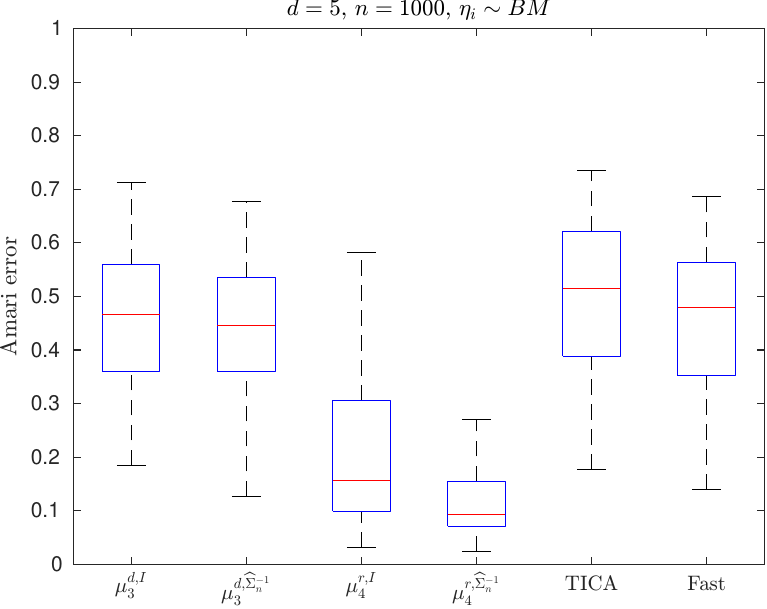}	
			}
		\end{subfigure}
%	\begin{subfigure}{
%			\includegraphics[scale=0.9,width=0.97\linewidth,height=0.59\linewidth]{Simulations/Revision/Replication/Figures/Figure_commonvariance_d=5_n=1000_eta=BM_amari}	
%		}
%	\end{subfigure}
	\end{center}
	{\begin{minipage}{12.5cm} \footnotesize \emph{Notes:} The figure shows the boxplots for the Amari errors (across $S=100$ simulations) for data sampled from the common variance model \eqref{eq:generalcommonvariance}. The different settings for the simulations designs are described in the titles and the $x$-labels indicate the different estimation methods used. \end{minipage} }
\end{figure}

\subsection{Multiple scaled elliptical component models}\label{ssec:multscaledellip} 

\begin{table}
	\caption{\sc Amari Errors: Scaled Elliptical \ \ \ \ \  \ \ \ \ \ \ \ \ }
	\begin{center}
		\begin{tabular}{lcccccccccc} 
			\hline\hline 
			& \multicolumn{10}{c}{Non-Independent Components Analysis} \\[0.5ex]
			Method       & $\mathcal N$ & $t(5)$ & SKU & KU & BM & SBM & SKB & TRI & CL & ACL \\   \hline 
			$\mu_3^{d, I} $   &0.44   & 0.44 &   0.45 &   0.44  &  0.45 &   0.43  &  0.44 &   0.45 &   0.43&    0.43 \\
			$\mu_3^{d, \widehat \Sigma_n^{-1}} $   &0.43  &  0.41  &  0.42 &   0.42  &  0.44 &   0.43 &   0.43&    0.44 &   0.41  &  0.41 \\
			$\mu_4^{r, I} $   & 0.21  &  0.25  &  0.23  &  0.23  &  0.21 &   0.21  &  0.22 &   0.25 &   0.24&    0.25 \\
			$\mu_4^{r, \widehat \Sigma_n^{-1}} $   &0.26  &  0.24 &   0.25  &  0.26  &  0.26 &   0.21 &   0.27 &   0.24 &   0.23  &  0.24 \\
			& & & & & & & & & & \\
			TICA & 0.34  &  0.38  &  0.34   & 0.36  &  0.34 &   0.36 &   0.36  &  0.35 &   0.35  &  0.36 \\
			& & & & & & & & & & \\\hline\hline 
			& \multicolumn{10}{c}{Independent Components Analysis} \\ [0.5ex]
			Method  	& $\mathcal N$ & $t(5)$ & SKU & KU & BM & SBM & SKB & TRI & CL & ACL \\   \hline  
			Fast  &0.43  &  0.43 &   0.44   & 0.43 &   0.42   & 0.44  &  0.44  &  0.44 &   0.44 &   0.46 \\
			JADE & 0.45   & 0.43  &  0.44   & 0.44 &   0.44   & 0.43  &  0.45  &  0.45 &   0.45 &   0.44 \\
			Kernel & 0.46   & 0.43 &   0.44 &   0.44 &   0.42   & 0.45 &   0.44  &  0.44 &   0.44 &   0.45 \\
			ProDen   &0.45  &  0.43 &   0.44&    0.45 &   0.42  &  0.45&    0.45  &  0.45 &   0.44 &   0.46\\	
			Efficient &0.44 &   0.43 &   0.45&    0.43 &   0.44 &   0.42&    0.44  &  0.44  &  0.44 &   0.44 \\
			NPML   & 0.42   & 0.43  &  0.43  &  0.43   & 0.42   & 0.42   & 0.43  &  0.42 &   0.43 &   0.42 \\	
			\hline\hline 
			& & & & & & & & & & \\
			\multicolumn{11}{l}{ \begin{minipage}{12.5cm} \footnotesize \emph{Notes:} The table reports the average Amari errors (across $S=1000$ simulations) for data sampled from the multiple scaled elliptical model \eqref{eq:generalmultielliptical} with $d=2$ and $n=200$. The columns correspond to the different errors considered for the components of $\eta$, see Table \ref{tab:ng_dist}. The top panel reports the errors for the minimum distance methods and Topographical ICA (TICA). For the minimum distance methods we consider diagonal ($d$) and reflectionally invariant ($r$) restrictions for different order tensors $\mu_3, \mu_4$, combined with weighting matrices $W_n = I_d, \widehat \Sigma_n^{-1}$. The bottom panel reports comparison results for different independent component analysis methods: FastICA \citep {Hyvarinen.99}, JADE \cite{CardosoSouloumiac1993}, kernel ICA \citep{BachJordan.03}, ProDenICA \citep{HastieTibshirani.02}, efficient ICA \citep{chenbickel} and non-parametric ML ICA \citep{SamworthYuan_2012}. \end{minipage}  }
		\end{tabular}
	\end{center}
	\label{tab:estimatesbaselinese}
\end{table}

Next, we generate data from the nICA model with multiple scaled elliptical errors as in \eqref{eq:generalmultielliptical}. Specifically, $\eta$ is drawn from the uniform distribution on the $d$-sphere and $\tau$ is simulated as $\tau = K e$, where $K$ is a square matrix of ones and $e \in \R^d$ has independent components that are drawn from the distributions in Table \eqref{tab:ng_dist}. This implies that both the components of $\tau$ and $\eta$ are dependent and the diagonal tensor identification result no longer holds. That said, Corollary \ref{cor:mainreflec} show that the reflectionally invariant moment tensors can still be used for identification. Also, in this setting topographical ICA cannot be used as the components of $\eta$ are no longer independent. The other parts of the simulation design are similar as above and we perform the same comparisons. 

The results are shown in Table \ref{tab:estimatesbaselinese}. We find that the minimum distance methods based on the reflectionally invariant tensor restrictions now outperform all other methods. Using the efficient weighting matrix is not always preferable, as for most densities the weighting matrix is not estimated very accurately leading to more poor performance. The differences across the different densities for $e$ are often small and the errors made by the ICA methods are quite similar.

In the supplementary material Section~\ref{appsec:simsadditional} we provide a number of additional results. We show different error metrics, sample sizes and dimensions. The main conclusion --- incorrectly imposing independence is costly --- is found to hold across these variations

\section{Conclusion}\label{sec:conc}

In the ICA literature identifiability of \eqref{nica} is assured when $\varepsilon$ has independent components out of which at most one is Gaussian. Although in the classical ICA literature independence seems a natural assumption, in many other applications it is considered too strong. 

Our paper proposes a general framework to study weak conditions under which $A$ in $AY = \varepsilon$ is identified up to the set of signed permutations. We develop a novel approach to study this identifiability problem in the case of zero restrictions on fixed order moments or cumulants of $\varepsilon$. We obtain positive results for some useful zero patterns. These results can be used under strictly weaker conditions than independence and ensure the identifiability of several popular non-independent component models, e.g. common variance, scaled elliptical and mean independent components models. 

Our results are practically important as they provide researchers with alternative identification and inference tools when independence fails to hold. Indeed, while the literature has produced a variety of tests for the independence of the components of $\varepsilon$ in model \eqref{nica} \citep[e.g.][]{MT17}, it is unclear how one should proceed in practice when such tests reject. Our paper provides methods that can be applied in such situations.

While we have focused on relaxing the independence assumption in \eqref{nica}, it is easy to see that similar techniques can be used to relax independence assumptions in other linear models; e.g. measurement error models \citep{Schennach.21}, triangular systems \citep{Lewbel21}, and structural vector autoregressive models \citep{KilianLutkepohl.17}.

\section*{Acknowledgements}

We would like to thank Joe Kileel,  Mateusz Micha\l{}ek, Mikkel Plagborg-M{\o}ller and Anna Seigal for helpful remarks.

\bibliography{../../bib_mtp2}
\bibliographystyle{imsart-nameyear}

\end{document}

\begin{frontmatter}
\title{Supplementary material: Non-Independent Components Analysis}
%\title{A sample article title with some additional note\thanksref{t1}}
\runtitle{Supplementary material}
%\thankstext{T1}{A sample additional note to the title.}

\begin{aug}
%%%%%%%%%%%%%%%%%%%%%%%%%%%%%%%%%%%%%%%%%%%%%%%
%% Only one address is permitted per author. %%
%% Only division, organization and e-mail is %%
%% included in the address.                  %%
%% Additional information can be included in %%
\author[A]{\fnms{Geert}~\snm{Mesters}\ead[label=e1]{geert.mesters@upf.edu}\orcid{0000-0001-6996-9520}},
\author[B]{\fnms{Piotr}~\snm{Zwiernik}\ead[label=e2]{piotr.zwiernik@utoronto.ca}\orcid{0000-0003-3431-131X}}
%%%%%%%%%%%%%%%%%%%%%%%%%%%%%%%%%%%%%%%%%%%%%%
%% Addresses                                %%
%%%%%%%%%%%%%%%%%%%%%%%%%%%%%%%%%%%%%%%%%%%%%%
\address[A]{Department of Economics and Business,
Universitat Pompeu Fabra\printead[presep={,\ }]{e1}}

\address[B]{Department of Statistical Sciences,
University or Toronto\printead[presep={,\ }]{e2}}
\end{aug}

\begin{abstract}
	We provide the following additional results. 
	\begin{itemize} 
		\item[\ref{appsec:mainproofs}:] Omitted proofs -- main text 
		\item[\ref{appsec:addmotivation}:] Additional motivation 
		\item[\ref{appsec:localident}:] Local identification beyond signed-permutations
		\item[\ref{appsec:cumulant}:] Moments and Cumulants --- some useful properties 
		\item[\ref{appsec:inferencedetails}:] Additional inference tools
		\item[\ref{appsec:assvar}:] Computing the asymptotic variance 
		\item[\ref{appsec:simsadditional}:] Additional simulation results 
		\item[\ref{appsec:proofsomitted}:] Additional proofs.   
	\end{itemize}
\end{abstract}

\begin{keyword}[class=MSC]
\kwd[Primary ]{15A69}
\kwd{62H99}
\kwd[; secondary ]{}
\end{keyword}

\begin{keyword}
\kwd{independent component analysis}
\kwd{identifiability}
\kwd{cumulants}
\kwd{tensors}
\end{keyword}

\end{frontmatter}
%%%%%%%%%%%%%%%%%%%%%%%%%%%%%%%%%%%%%%%%%%%%%%
%% Please use \tableofcontents for articles %%
%% with 50 pages and more                   %%
%%%%%%%%%%%%%%%%%%%%%%%%%%%%%%%%%%%%%%%%%%%%%%
%\tableofcontents
\setcounter{section}{-1}

\section{Omitted proofs -- main text}\label{appsec:mainproofs} 

In this section we collect the omitted proofs from the main text. 

\subsection{Omitted proof from Section \ref{sec:mainres} }

\begin{proof}[Proof of Proposition~\ref{prop:binaryMI}]

The condition $Q\bullet T\in \cV$ translates into two equations $(Q\bullet T)_{12\cdots 2}=(Q\bullet T)_{1\cdots 12}=0$. In other words,
$$
Q_{11}\sum_{\bs j}Q_{2j_1}\cdots Q_{2 j_{r-1}} T_{1\bs j}+  Q_{12}\sum_{\bs j}Q_{2j_1}\cdots Q_{2 j_{r-1}} T_{2\bs j}\;=\;0
$$
and
$$
Q_{21}\sum_{\bs j}Q_{1j_1}\cdots Q_{1 j_{r-1}} T_{1\bs j}+  Q_{22}\sum_{\bs j}Q_{1j_1}\cdots Q_{1 j_{r-1}} T_{2\bs j}\;=\;0,
$$
where in both cases the sum goes over all $(r-1)$-tuples $\bs j=(j_1,\ldots,j_{r-1})\in \{1,2\}^{r-1}$. Note that, since $T$ is symmetric, the entry $T_{\bs i}$ depends only on how many times $1$ appears in $\bs i$. Write $t_k=T_{\bs i}$ if $\bs i$ has $k$ ones, $k=0,\ldots,r$. With this notation the two equations above simplify to
$$
\sum_{k=0}^{r-1}\binom{r-1}{k}Q_{11}Q_{21}^kQ_{22}^{r-1-k}t_{k+1}+\sum_{k=0}^{r-1}\binom{r-1}{k}Q_{12}Q_{21}^kQ_{22}^{r-1-k}t_{k}\;=\;0
$$
and $$
\sum_{k=0}^{r-1}\binom{r-1}{k}Q_{21}Q_{11}^kQ_{12}^{r-1-k}t_{k+1}+\sum_{k=0}^{r-1}\binom{r-1}{k}Q_{22}Q_{11}^kQ_{12}^{r-1-k}t_{k}\;=\;0.
$$
If one of the entries of $Q$ is zero then $Q$ is a permutation matrix. So assume that $Q$ has no zeros. Assume also without loss of generality that $Q$ is a rotation matrix, that is, $Q_{11}=Q_{22}$ and $Q_{12}=-Q_{21}$. Denote $z=Q_{21}/Q_{11}$, which corresponds to the tangent of the rotation angle and so it can take any non-zero value (zero is not possible as $Q_{21}\neq 0$). With this notation and after dividing by $Q_{11}^r$, the two equations become
\begin{equation}\label{eq:bin1}
	\sum_{k=0}^{r-1}\binom{r-1}{k}z^kt_{k+1}-\sum_{k=0}^{r-1}\binom{r-1}{k}z^{k+1}t_{k}\;=\;0
\end{equation}
and 
$$
\sum_{k=0}^{r-1}\binom{r-1}{k}(-1)^{r-1-k}z^{r-k}t_{k+1}+\sum_{k=0}^{r-1}\binom{r-1}{k}(-1)^{r-1-k}z^{r-1-k}t_{k}\;=\;0.
$$
It is convenient to rewrite the latter as
\begin{equation}\label{eq:bin2}
	\sum_{k=0}^{r-1}\binom{r-1}{k}(-1)^{k}z^{k+1}t_{r-k}+\sum_{k=0}^{r-1}\binom{r-1}{k}(-1)^{k}z^{k}t_{r-k-1}\;=\;0.
\end{equation}
Using the fact that $t_1=t_{r-1}=0$, \eqref{eq:bin1} can be written as 
$$
\sum_{k=1}^{r-1}\left(\binom{r-1}{k}t_{k+1}-\binom{r-1}{k-1}t_{k-1}\right)z^k=0.
$$
and \eqref{eq:bin2} can be written as
$$
\sum_{k=1}^{r-1}\left(\binom{r-1}{k}t_{r-k-1}-\binom{r-1}{k-1}t_{r-k+1}\right)(-z)^k=0.
$$
Since $z\neq 0$, we can divide by it and in both cases we obtain two polynomials of order $r-2$. The first polynomial has coefficients
$$
a_k=\binom{r-1}{k+1}t_{k+2}-\binom{r-1}{k}t_{k}\qquad\mbox{for }k=0,\ldots,r-2
$$
and the second has coefficients
$$
b_k=(-1)^{k-1}\left(\binom{r-1}{k+1}t_{r-k-2}-\binom{r-1}{k}t_{r-k}\right)=(-1)^ka_{r-k-2}.
$$
A common zero for these two polynomials exists if and only if the corresponding resultant is zero. Resultant is defined as the determinant of a certain matrix populated with the coefficients of both polynomials. After reordering the columns of this matrix, we obtain
$$
\begin{bmatrix}
	a_0 & a_{r-2} & 0 & 0 & \cdots & 0 & 0\\
	a_1 & -a_{r-3} & a_0 & a_{r-2} & \cdots & 0 & 0\\
	\vdots & \vdots & \vdots & \vdots & \cdots & \vdots & \vdots\\
	a_{r-2} & (-1)^ra_{0} & a_{r-3} & (-1)^{r-1}a_{1} & \cdots & a_0 & a_{r-2}\\
	0 & 0 & a_{r-2} & (-1)^ra_{0} & \cdots & a_1 & -a_{r-2}\\
	0 & 0 & 0 & 0 & \cdots & a_2 & a_{r-3}\\
	\vdots & \vdots & \vdots & \vdots & \cdots & \vdots& \vdots\\
	0 & 0 & 0 & 0 & \cdots & a_{r-2} & (-1)^ra_0\\
\end{bmatrix}.
$$
The first two columns are linearly independent of each other unless the second is a multiple of the first. Indeed, if $r$ is odd, this is only possible if $a_0=\cdots =a_{r-2}=0$ (which cannot hold under the genericity assumptions). If $r$ is even this is possible if and only if either $a_k=(-1)^ka_{r-2-k}$ for all $k$, or $a_k=(-1)^{k-1}a_{r-2-k}$ for all $k$ (which cannot hold under the genericity assumptions). By the same argument, the third and the fourth column are independent of each other and linearly independent of the previous two. Proceeding recursively like that, we conclude that all columns in this matrix are linearly independent proving that the two polynomials cannot have common roots. In other words, there is no rotation matrix apart from the $0^\circ$ and the $90^\circ$ rotation matrices that satisfy $Q\bullet T\in \cV$. 
\end{proof}

\subsection{Omitted proofs from Section \ref{sec:inference}}  

\begin{proof}[Proof of Lemma~\ref{lem:mainrephrase}]
	We have $L_W(A)=0$ if and only if $g(A)=0$, which is equivalent $A\bullet h_2(Y)=I_d$ and $ A\bullet h_r(Y)\in \cV$. Since \eqref{nica} holds, we also have
	$A_0\bullet h_2(Y)= I_2$ and $A_0\bullet h_r(Y)\in \cV$. It follows that $A_0^{-1}A\in {\rm O}(d)$, or in other words, $A=QA_0$ for some $Q\in {\rm O}(d)$. Further, 
	$$
	A\bullet h_r(Y)=QA_0\bullet h_r(Y)=Q\bullet h_r(\varepsilon)\in \cV,
	$$
	which implies that $Q\in \cG_T(\cV)$ and by Theorem \ref{th:diagonal} or \ref{th:reflinv} we have $\cG_T(\cV)={\rm SP}(d)$.	
\end{proof}

\begin{proof}[Proof of Proposition~\ref{prop:consist}] 

The proof follows from verifying the conditions for consistency of a general extremum estimator. Specifically, we will verify the conditions of Theorem 2.1 in \cite{NeweyMcFadden1994}. We restate the theorem for completeness. 
\begin{thm}\label{thm:consistnewey}
	Suppose that $\hat \theta$ minimizes $\hat L_n(\theta)$ over $\theta \in \Theta$. Assume that there exists a function $L_0(\theta)$ such that (a) $L_0(\theta)$ is uniquely minimized at $\theta_0$, (b) $L_0(\theta)$ is continuous, (c) $\Theta$ is compact and (d) $\sup_{\theta \in \Theta} | \hat L_n(\theta) -  L_0(\theta) | \stackrel{p}{\to} 0$, then $\hat \theta \stackrel{p}{\to} \theta_0$.     
\end{thm}	
Next, we verify assumptions (a)-(d) under assumptions (i)-(iv) stated in Proposition~\ref{prop:consist}. First, note that $\widehat A_{W_n}$ minimizes $\hat L_{W_n}(A)$ and we take $L_W(A)$ as $L_0(\theta)$ in Theorem~\ref{thm:consistnewey}. Second, in our case the minimizer of $L_W(A)$ is not unique but will correspond to any of the finite points $Q A_0$ for some $Q \in SP(d)$. It follows that our consistency result will only be up to permutation and sign changes of the true $A_0$ \cite[e.g.][]{chenbickel}. 
Formally, for (a): suppose that $A$ is such that $A  \neq Q A_0$ for any $Q \in {\rm SP}(d)$, then $g(A) \neq 0$ by assumption (i) and, since $W$ is positive definite by (ii), we have $L_W(A) > 0$. Hence it follows that $L_W(A)$ is only minimized at $Q A_0$ for some $Q \in {\rm SP}(d)$.
Condition (b) follows as $L_W(A)$ is a composition of two polynomial maps. Condition (c) follows from (ii). Condition (d) is assured by the following result. 
\begin{lem}\label{uniformcompact}
	Suppose that $\{ Y_s \}_{s=1}^n$ is i.i.d, $W_n \stackrel{p}{\to } W$, $\mathbb E \|Y_{s}\|^r  < \infty$, and $\cA\subset {\rm GL}(d)$ is a compact set. Then 
	\[
	\sup_{A \in \mathcal \cA} |\hat L_{W_n}(A) - L_{W}(A)  | \stackrel{p}{\to} 0 
	\]
\end{lem}
\begin{proof}
	First, note that given the i.i.d. assumption and the moment condition~(iv) we have that $\|\hat{\bs \mu}_p - \mu_p(Y) \|\stackrel{p}{\to} 0$ and $\| \mathsf k_p  - \kappa_r(Y) \| \stackrel{p}{\to} 0$ for any $p \leq r$ by Lemma~\ref{lem:convergence} part 1. Note that the norm $\| \cdot \|$ on the tensor is defined in the usual way as the sum of the squares of all elements. Using the general notation of Section~\ref{sec:inference} we have that $\| \hat h_p - h_p(Y) \| \stackrel{p}{\to} 0$ for $p \leq r$. Hence, 
	$$
	\sup_{A \in \cA}\|A^{\otimes p}  {\rm vec}( \hat h_p  -  h_p(Y))  \|^2 \leq \|\hat h_p - h_p(Y) \|^2 \sup_{A \in \cA}  \| A^{\otimes p} \|^2  \stackrel{p}{\to} 0. 
	$$
	Here we used the fact that $\cA$ is a compact and so, in particular, $\| A^{\otimes p} \|^2$ is bounded on $\cA$. 
	
	Using \eqref{eq:mdiff}, we get 
	\begin{align*}
		\sup_{A \in \cA} \|\hat m_n(A) - m(A)  \|^2 \leq& \sup_{A \in \cA} \|A^{\otimes 2} {\rm vec}( \hat h_2  -  h_2(Y))  \|^2 \\  &+\sup_{A \in \cA} \|A^{\otimes r} {\rm vec}( \hat h_r  -  h_r(Y) )  \|^2  \stackrel{p}{\to} 0~.
	\end{align*}
	As $g_{S,T}(A)$ is defined in \eqref{eq:gst} as a projection of $m_{S,T}(A)$ on certain coordinates, we conclude that 
	$$
	\sup_{A\in \cA}\|\hat g_n(A)-g(A)\| \stackrel{p}{\to} 0.
	$$
	By the triangle inequality 
	$$
	\left|\hat L_{W_n}(A) - L_W(A) \right|  \leq  \left|\hat L_{W_n}(A) - L_{W_n}(A)\right|+\left| L_{W_n}(A) - L_W(A)\right|.
	$$
	The second term is is readily bounded by $\|g(A)\|^2 \|W_n-W\|$ using the basic operator norm inequality. To bound the first term, note that, by the triangle inequality
	$$
	\left|\hat L_{W_n}(A) - L_{W_n}(A)\right|\;=\;\left| \|\hat g_n(A)\|^2_{W_n}-\|g(A)\|^2_{W_n}\right|\;\leq\;\|\hat g_n(A)-g(A)\|^2_{W_n},
	$$
	which can be bounded by $\|\hat g_n(A)-g(A)\|^2\|W_n\|$.	We conclude that 
	$$
	\left|\hat L_{W_n}(A) - L_W(A) \right|  \leq \|\hat g_n(A)-g(A)\|^2\|W_n\|+ \|g(A)\|^2 \|W_n-W\|.
	$$
	It follows that $\sup_{A \in \cA}|\hat L_{W_n}(A) - L_W(A) |  \stackrel{p}{\to} 0$ as required. 
\end{proof}
We may now apply Theorem~\ref{thm:consistnewey} to conclude that $\widehat A_{W_n} \stackrel{p}{\to} Q A_0$ for some $Q \in {\rm SP}(d)$. 	
\end{proof}

\subsection{Proof of Proposition~\ref{prop:assnormal}}

The proof follows from verifying the conditions for asymptotic normality of a generalized moment or distance estimator. Specifically, we will verify the conditions of Theorem 3.2 in \cite{NeweyMcFadden1994}. We restate the theorem for completeness. 
\begin{thm}\label{th:NMFasnorm}
	Suppose that $\hat \theta$ minimizes $\hat L_n(\theta)$ over $\theta \in \Theta$ with $\Theta$ compact, where $\hat L_n(\theta)$ is of the form $\hat g_n(\theta)' W_n \hat g_n(\theta)$ and $W_n \stackrel{p}{\to} W$ with $W$ positive semi-definite, $\hat \theta \stackrel{p}{\to} \theta_0$ and (a) $\theta_0 \in {\rm Int}(\Theta)$, (b) $\hat g_n(\theta)$ is continuously differentiable in a neighborhood $\mathcal N$ of $\theta_0$, (c) $\sqrt{n} \hat g_n(\theta_0) \stackrel{d}{\to} N(0,\Omega)$, (d) there is $G(\theta)$ that is continuous at $\theta_0$ and $\sup_{\theta \in \Theta} \| \nabla_{\theta} \hat g_n(\theta) - G(\theta)  \|\stackrel{p}{\to} 0$, (e) for $G = G(\theta_0)$, $G'WG$ is nonsingular. Then, $$\sqrt{n}(\hat \theta - \theta_0) \;\;\;\;\stackrel{d}{\to}\;\;\;\; N\big(0, (G'WG)^{-1} G'W \Omega W G'  (G'WG)^{-1}\big)~.$$		\end{thm}
The loss $\hat L_n(\theta)$ in Theorem~\ref{th:NMFasnorm} corresponds to our $\hat L_n(A)$. Our $\hat g_n(A)$ corresponds to their $\hat g_n(\theta)$. We have $\widehat A_{W_n} \stackrel{p}{\to} \tilde A_0=Q A_0$ for some $Q\in {\rm SP}(d)$ by Proposition~\ref{prop:consist}, and the conditions on the weighting matrix are satisfied by (ii). Condition (a) of Theorem~\ref{th:NMFasnorm} is satisfied by assumption (v). For (b) note that $\hat g_n(A)$ is a polynomial map  in $A$ and hence smooth. For (c), by Lemma~\ref{lem:convergence}, $\sqrt{n}\, {\rm vec}(\hat m_n(\tilde A_0)-m(\tilde A_0))$ weakly converges to $N(0,\Sigma_h^{2,r})$, where $h = \mu $ or $h= \mathsf k$ pending whether moments or $\mathsf k$-statistics are used to compute $\hat m_n(A)$. The variance matrices are defined in \eqref{sigma2rmu} or \eqref{eq:Sig2r}. However, $\hat g_n(\tilde A_0)$ is simply a projection of $(\hat m_n(\tilde A_0)-m(\tilde A_0))$ onto the coordinates of $\cV^\perp$. Therefore, it also  weakly converges to $N(0,\Sigma)$, where 
\begin{equation}\label{eq:thisissigma}
	\Sigma = D_{\mathcal I}^{2,r} \Sigma_h^{2,r} D_{\mathcal I}^{2,r'}~
\end{equation}
with
% \\
% 	\nonumber	                         &=  \sqrt{n}\left[ \begin{array}{c} {\rm vech}[ \tilde A_0 \bullet (\mathsf k_2 - \kappa_2( Y))]  \\ {\rm vec}_{\mathcal I}[ \tilde A_0  \bullet (\mathsf k_r - \kappa_r( Y))]  \end{array} \right] \\
% 	\nonumber	                         &= D_{\mathcal I}^{2,r} \sqrt{n} \left[ \begin{array}{c} {\rm vec}[ \tilde A_0 \bullet (\mathsf k_2 - \kappa_2( Y))]  \\ {\rm vec}[ \tilde A_0  \bullet (\mathsf k_r - \kappa_r( Y))]  \end{array} \right]\\
% 	\label{assmoments}	                         &\stackrel{d}{\to} N\left(0,D_{\mathcal I}^{2,r} \Sigma^{2,r} D_{\mathcal I}^{2,r'} \right) ~,   
% 	\end{align}   
$D_{\mathcal I}^{2,r}$ being a selection matrix that selects the corresponding to the unique entries in $S^r(\R^d)\oplus \cV^\perp$. Note that the specific form of $\Sigma$ depends on whether moment or cumulant restrictions are used, i.e. $h = \mu, \kappa$. Here we suppress this dependence in the notation, but in Appendix~\ref{appsec:assvar} where we discuss the estimation of $\Sigma$ we make it explicit.

We now show that  (d) holds. The derivative of the map $g_{S,T}(A)$ in \eqref{eq:gst} is a linear mapping from $\R^{d\times d}$ to $\R^{d_g}$. It is obtained as a composition of the derivative of $m_{S,T}(A)$ given by the vectorized version of  $(K_{S,A}(V),K_{T,A}(V))$, with each component defined in \eqref{eq:jacqt}, and the projection $\pi_\cV$. Thus, the derivative is given by mapping $V\in \R^{d\times d}$ to the vector
$$
{\rm vec}\Big((V,A)\bullet S+(A,V)\bullet S,\;\pi_{\cV}\big((V,A,\ldots,A)\bullet T+\cdots +(A,\ldots,A,V)\bullet T\big)\Big).
$$

The Jacobian matrix $G_{S,T}(A)$ representing this derivative has $d^2$ columns and the column corresponding to variable $A_{ij}$ is obtained simply by evaluating the derivative at the unit matrix $E_{ij}\in \R^{d\times d}$. In symbols, this column is given by stacking the vector  $(E_{ij}\otimes A+A\otimes E_{ij}){\rm vec}(S)$ over the vector
\begin{equation}\label{eq:jaccol}
\big((E_{ij}\otimes A\otimes \cdots\otimes A)+\cdots+(A\otimes \cdots\otimes A\otimes E_{ij})\big) \cdot {\rm vec}(T),
\end{equation}
and then selecting only the entries corresponding to the $2$-tuples $i\leq j$ and  $r$-tuples in $\cI$.

Denote the Jacobian $G_{S,T}$ by $G(A)$ if $S=\mu_2(Y)$, $T=\mu_r(Y)$ and by $\widehat G(A)$ if $S=\hat{\bs \mu}_2$, $T=\hat{\bs \mu}_r$ (or $S=\kappa_2(Y)$, $T=\kappa_r(Y)$ and $S = \mathsf{k}_2$, $T = \mathsf{k}_r$  ).
The columns of $\widehat G(A) - G(A)$ are like explained in \eqref{eq:jaccol} with
$S = \hat{\bs \mu}_2 - \mu_2(Y)$ and $T = \hat{\bs \mu}_r - \mu_r(Y)$ (or 
$S=\mathsf{k}_2-\kappa_2(Y)$ and $T=\mathsf{k}_r-\kappa_r(Y)$). Since $\|S\|\stackrel{p}{\to} 0 $ and $\|T\|\stackrel{p}{\to} 0 $ by Lemma~\ref{lem:convergence} part 1, and because $A$ is fixed, the norm of each column converges to zero. In consequence, for each $A$, $\| \widehat G(A) - G(A)  \| \stackrel{p}{\to} 0$. Since $\cA$ is compact and $\widehat G(A) - G(A)$ is smooth, we conclude
\begin{equation}\label{eq:uniformder}
\sup_{A \in \mathcal A} \| \widehat G(A) - G(A)  \|\stackrel{p}{\to} 0.
\end{equation}

This establishes part (d). To establish part~(e) note that $W$ is positive definite and the Jacobian $G(Q A_0)$ has full column rank by Lemma~\ref{lem:fullrank} below.

\begin{lem}\label{lem:fullrank}
If $\cV$ assures identifiability up to a sign permutation matrix, then the matrix $G(QA_0)$ has full column rank for each $Q\in {\rm SP}(d)$.  
\end{lem}
\begin{proof}
It is enough to show that the derivative of $g(A)$ at $QA_0$ has trivial kernel. We first analyze the $S^2(\R^d)$-part of the derivative noting that $\mu_2(Y) = \kappa_2(Y)$ as $\E Y = 0$. Suppose $(QA_0)\bullet \kappa_2(Y)=I_d$ and so the condition $(V,QA_0)\bullet \kappa_2(Y)+(QA_0,V)\bullet \kappa_2(Y)=0$ is equivalent to
$$
(A_0^{-1}Q'V,I_d)\bullet I_d+(I_d,A_0^{-1}Q'V)\bullet I_d\;=\;0.
$$
Using the derivative $K_{S,A}$ notation given in \eqref{eq:jacqt}, we write this last condition as $K_{I_d,I_d}(A_0^{-1}Q'V)=0$. Similarly, the $\cV^\perp$-part implies that $K_{T,I_d}(A_0^{-1}Q'V)=0$ with $T=\kappa_r(Y)$.  This implies that $A_0^{-1}Q'V=0$ by  Lemma~\ref{lem:isolder} and the fact that $I_d$ is an isolated point of $\cG_T$. We conclude that $V$ must be zero.
\end{proof}
Having verified all conditions of \ref{th:NMFasnorm} we can apply the theorem to prove the first display in Proposition \ref{prop:assnormal}. The second display follows as a special case when taking $W_n = \widehat \Sigma_n^{-1}$, noting that $\widehat \Sigma_n^{-1} \to \Sigma^{-1}$, and replacing $W$ by $\Sigma^{-1}$ in the first display.

\section{Additional motivation}\label{appsec:addmotivation} 

In this section we discuss some additional relations that aim to further highlight the usefulness of the identification results presented in the main text. 

\subsection{Scaled Elliptical LiNGAM} 

For the model $AY = \varepsilon$, where the elements of $\varepsilon$ are independent and non-Gaussian \cite{Shimizu.06}, showed that one can uniquely recover $A$ if there exists an (unknown) permutation of the rows of $A$ that is lower triangular, i.e. the model corresponds to a directed acyclic graph. The proposed LiNGAM discovery algorithm uses ICA and a search over permutations to find the best fitting lower triangular model. 

Now reconsider the multiple scaled elliptical components model
$$
	AY \;=\; \varepsilon~, \quad \text{with} \quad \varepsilon\;=\;\tau \odot U \quad \text{and} \quad U \sim \cU_d~, 
$$
with $\tau \in \R^d$ and $U$ independent.

With elliptical errors the LiNGAM algorithm can no longer be used. However, the results of this paper suggests a natural modification where the ICA algorithm is replaced by the moment or cumulant based estimation methods that we introduce in Section~\ref{sec:inference}. These methods are build on the new identification results for the multiple scaled elliptical components model.  

Specifically, in Algorithm A of \cite{Shimizu.06} one can replace the ICA method that is used in step 1 by the minimum distance moment/cumulant estimation method of Section~\ref{sec:inference}. The other steps of the algorithm do not require adjustment.

\subsection{Invariance} 

In Section~\ref{sec:motivation} we motivated non-independent component models using specific examples (e.g. common variance model) as well as by relaxing independence (e.g. mean independence). In both cases the resulting model still implied sufficient zero restrictions on the higher order moments/cumulants of $\varepsilon$ to ensure the identifiability of $A$ (cf Corollaries \ref{cor:maindiag} and \ref{cor:mainreflec}). Here we briefly show that such zero restrictions can also arise from invariance properties of the distribution of $\varepsilon$. 

Suppose that the distribution of $\varepsilon$ is the same as the distribution of $D\varepsilon$ for \emph{every} diagonal matrix $D$ with $D_{ii}=\pm 1$ for all $i=1,\ldots,d$ (e.g. when $\varepsilon$ has spherical distribution). In this case, by multilinearity of cumulants,
$$
[h_r(D\varepsilon)]_{i_1\cdots i_r}\;=\;D_{i_1 i_1}\cdots D_{i_r i_r} [h_r(\varepsilon)]_{i_1\cdots i_r}~.
$$
Since $D$ is arbitrary, $ [\kappa_r(\varepsilon)]_{i_1\cdots i_r}$ must be zero unless all indices appear even number of times. In particular, if $r$ must be even and for example, if $r=4$, the only potentially non-zero cumulants are $\kappa_{iiii}$ and $\kappa_{iijj}$. These zero patterns correspond exactly with the reflectionally invariant restrictions introduced in Section~\ref{sec:reflectionally} and as such Corollary \ref{cor:mainreflec} also ensure the identifiability of $A$ in $AY = \varepsilon$ when the distribution of $\varepsilon$ is the same as the distribution of $D\varepsilon$. 

%We treat this zero pattern more in detail in Section~\ref{sec:reflectionally} and we show how it appears in (trans)elliptical distributions.

\subsection{Alternative estimation methods} 

In the main text we outlined some minimum distance estimation methods for estimating $A$ in $AY = \varepsilon$ based on the identifying moment/cumulant restrictions. We adopted this approach as it can be implemented naturally based on our identification results. That said, for specific non-independent components models it is obviously feasible to develop alternative estimators based on the identification results. To illustrate, we discuss some approaches for the mean independent components model:  
$$
	a'_iY \;=\; \varepsilon_i~, \quad \text{with} \quad \mathbb E(\varepsilon_i | \varepsilon_{-i} ) = 0~, \quad \text{for} \ \  i=1, \ldots,d~. 
$$
\cite{ShaoZhang.14} introduce \emph{martingale difference correlations} to measure the departure of conditional mean independence between a scalar response variable (i.e. $\varepsilon_i$) and a vector predictor variable (i.e. $\varepsilon_{-i}$). This metric is a natural extension of distance correlation proposed by \cite{SzekelyRizzoBakirov.07}, which was adopted in \cite{MT17} for independent components analysis. These observations immediately suggest that jointly minimizing the martingale difference correlations between $\varepsilon_i$ and $\varepsilon_{-i}$ for all $i$ with respect to $A$ provides an attractive approach for estimating mean independent components models. 

Alternatively, recall that the efficient ICA method of \cite{chenbickel} is based on setting the efficient score function of the semi-parametric ICA model (with independent errors) to zero. The analytical form of these efficient scores relies on the independence assumption. When relaxing towards mean independence it is straightforward to derive a new analytical expression for the efficient scores and apply the algorithm of \cite{chenbickel} to set these scores to zero.

\section{Local identification beyond signed-permutations}\label{appsec:localident}

The results in Section~\ref{sec:mainres} stipulate conditions on moment tensors $ T = \mu_r(\varepsilon)$ or cumulant tensors $T = \kappa_r(\varepsilon)$ for which $A$ can be recovered up to sign and permutation. This section gives minimal conditions on $\mathcal V$ that ensure that $\mathcal G_T$ is finite. We subsequently use this result to highlight the gap that exists between restrictions that lead to finite sets and restrictions that lead to signed permutation sets. This finding has the important implication that it is in general not sufficient to prove that the Jacobian of the moment or cumulant restrictions is full rank in order to establish that the identified set is equal to the set of signed permutations.

Let $\cV\subset S^r(\R^d)$ be a set given as a set of zeros of a system of polynomials in the coordinated of $S^r(\R^d)$ (such set is called an algebraic variety). A subset $\cU\subseteq \cV$ is Zariski open in $\cV$ if the complement $\cV\setminus \cU$ is also an algebraic variety. In particular, a Zariski open set is also open in the classical topology. For example, the set of diagonal tensors in $S^r(\R^d)$ with at most one zero on the diagonal forms a Zariski open subset of the set of diagonal tensors. Similarly, the set of reflectionally invariant tensors satisfying the genericity condition \eqref{eq:genrefl} is Zariski open in the set of reflectionally invariant tensors.  Note that, in both cases, the constraints defining $\cV$ and $\cV\setminus \cU$ were linear.

Recall from \eqref{eq:GT} that for $T=h_r(\varepsilon)\in \cU$ we define $\cG_T(\cU)=\{Q\in {\rm O}(d):\;Q\bullet T\in \cU\}$ to be the set of all orthogonal matrices for which $h_r(Q\varepsilon)$ also lies in $\cU$. \begin{defn}\label{def:locid}
The problem of recovering $A$ in \eqref{nica} is locally identifiable under moment/cumulant constraints $\cU\subseteq \cV\subset S^r(\R^d)$ with $\cU$ open in $\cV$ if every point of $\cG_T(\cU)$ is an isolated point of $\cG_T(\cU)$. \end{defn}

Note that, at least in principle, $\cG_T(\cU)$ could contain infinitely many isolated points. The following result establishes link between local identification and finiteness of $\cG_T$. 
\begin{prop}\label{prop:finiteisol}
Let $\cU$ be a Zariski open subset of $\cV$. For $T^*\in \cU$ we have $|\cG_{T^*}(\cU)|<\infty$ if and only if each point of $\cG_{T^*}(\cU)$ is an isolated point of $\cG_{T^*}(\cU)$.
\end{prop}
\begin{proof}
The right implication is clear. For the left implication first note that $\cG_{T^*}(\cU)$ is a Zariski open subset of the real algebraic variety $\cG_{T^*}(\cV)$. Indeed, if $f_1(T)=\cdots=f_k(T)=0$ are the polynomials, in $T$, describing $\cV$ then the polynomials, in $Q$,  describing $\cG_{T^*}(\cV)$ within ${\rm O}(d)$ are $f_1(Q\bullet {T^*})=\cdots =f_k(Q\bullet {T^*})=0$. Similarly, if $\cV\setminus \cU$ is described within $\cV$ by $g_1(T)=\cdots=g_l(T)=0$. Then $\cG_{T^*}(\cV)\setminus \cG_{T^*}(\cU)$ is described by  $g_1(Q\bullet {T^*})=\cdots=g_l(Q \bullet {T^*})=0$.

Since $\cG_{T^*}(\cV)$ is a real algebraic variety,  the set of its isolated points is equal to its zero-dimensional components and so it must be finite; see for example Theorem~4.6.2 in \cite{cox2013ideals}. It is then enough to show that if $Q^\circ$ is isolated in $\cG_{T^*}(\cU)$ then it must be isolated in $\cG_{T^*}(\cV)$. Suppose that $Q^\circ\in \cG_{T^*}(\cU)$ is not isolated in $\cG_{T^*}(\cV)$. Then it must lie on an irreducible component of the variety $\cG_{T^*}(\cV)$ of a positive dimension. By assumption, for this $Q^\circ$,  $g_1(Q^\circ\bullet {T^*})\neq 0$, \ldots, $g_l(Q^\circ \bullet {T^*})\neq 0$. Thus, in any sufficiently small neighbourhood of $Q^\circ$ there will be a point that lies in $\cG_{T^*}(\cV)$ and $g_1,\ldots,g_l$ evaluate to something non-zero. In other words,  in any sufficiently small neighbourhood of $Q^\circ$ there is a point in $\cG_{T^*}(\cU)$ proving that $Q^\circ$ cannot be isolated in $\cG_{T^*}(\cU)$, which leads to contradiction. 
\end{proof}

\begin{rem}
The proof of Proposition~\ref{prop:finiteisol} also shows that if $\cU$ is a Zariski open subset of $\cV$ and $T\in \cU$ then $\cG_{T}(\cU)$ is a Zariski open subset of $\cG_{T}(\cV)$. Moreover, $Q\in \cG_{T}(\cU)$ is isolated if and only if it is isolated in $\cG_{T}(\cV)$.
\end{rem}

%\piotr{Perhaps we do not want to discuss these annoying technicalities and simply focus on the case when each solution is isolated. }
%\begin{prop}
%	Local identifiability under $\cU$ is equivalent to finiteness of $\cG_T$.\piotr{Work on this. }
%\end{prop}
%\begin{proof}
%	For every $\cI$ and $T\in \cV(\cI)$, the set of $Q$ satisfying $Q\bullet T\in \cV(\cI)$ is a real algebraic geometry. It then has a finite number of zero-dimensional components (isolated points). This remains true if we restrict it further to a Zariski open subset.\piotr{Work on this. }
%\end{proof}

% To verify whether Definition~\ref{def:locid} holds for a specific choice for $\cV$ we use the notation $A\bullet_k T$ to denote the linear action of $A$ on the $k$-th dimension of $T$, namely, for all $(i_1,\ldots,i_r)\in [d]^r$
% \begin{equation}\label{eq:bulletk}
% 	(A\bullet_k T)_{i_1\cdots i_r}\;:=\;\sum_{j_k=1}^d A_{i_k j_k}T_{i_1\cdots j_k\cdots i_d}.	
% \end{equation}

 {By the above remark, to show local identifiability it is enough to show that every element of $\cG_T(\cU)$ is isolated in $\cG_T(\cV)$. To show this, we take any point in $\cG_T(\cU)$ and try to perturb it infinitesimally staying in the orthogonal group. We need that every such infinitesimal perturbation sends the point outside of $\cG_T(\cV)$.}

\begin{lem}
For a fixed $T\in S^r(\R^d)$, consider the map from $\R^{d\times d}$ to $S^r(\R^d)$ given by $A\mapsto A\bullet T$. Its derivative at $A$ is a linear mapping on $\R^{d\times d}$ defined by
\begin{equation}\label{eq:jacqt}
K_{T,A}(V)\;=\; (V,A,\ldots,A)\bullet T+\cdots +(A,\ldots,A,V)\bullet T.
\end{equation}
Moreover, if $A$ is invertible, then
\begin{equation}\label{eq:KTAsimpl}
K_{T,A}(V)\;=\;K_{A\bullet T,I_d}(V A^{-1}).
\end{equation}
\end{lem}
\begin{proof}
For any direction $V\in \R^{d\times d}$,  we have 
\begin{eqnarray*}
(A+t V)\bullet T&-&A\bullet T\\
&=&t(V,A,\ldots,A)\bullet T+\cdots +t(A,\ldots,A,V)\bullet T + o(t). 
\end{eqnarray*}
So the proof of the first claim follows by the definition of a derivative. The second claim follows by direct calculation.
% Thus, by definition, the derivative of $A\mapsto A\bullet T$ is a linear mapping from $\R^{d\times d}$ to $S^r(\R^d)$ given by 
% \begin{equation}\label{eq:jacqt}
% U\;\;\mapsto\;\; (U,A,\ldots,A)\bullet T+\cdots +(A,\ldots,A,U)\bullet T.
% \end{equation}
\end{proof}

% \begin{rem}
% The $\binom{d+r-1}{r}\times d^2$ Jacobian matrix representing the vectorized version of this map is obtained by simply taking as the column corresponding to the variable $A_{ij}$ the vectorized $K_{T,A}(E_{ij})$, where $E_{ij}$ is the canonical unit matrix.
% \end{rem}

For a given linear subspace $\cV\subseteq S^r(\R^d)$, let $\pi_{\cV}:S^r(\R^d)\to \cV^\perp$ denote the orthogonal projection on $\cV^\perp$. Of course, $T\in \cV$ if and only if $\pi_{\cV}(T)=0$. Moreover, if $\cV=\cV(\cI)$ is given by zero constraints, then $\pi_\cV(T)$ simply gives the coordinates $T_{\bs i}$ for $\bs i\in \cI$.

In the next result, $K_{I_d,A}(V)=(V,A)\bullet I_d+(A,V)\bullet I_d$, which is a special instance of \eqref{eq:jacqt} for $r=2$. 

\begin{lem}\label{lem:isolder}
Let $\cU$ be a Zariski open subset of $\cV$. A point $Q$ is an isolated point of $\cG_T(\cU)$ if and only if  
\begin{equation}\label{eq:isol}
K_{I_d,Q}(V)=0\mbox{ and }\pi_{\cV}( K_{T,Q}(V))=0\qquad\mbox{implies } V=0.
\end{equation}
% $\pi_{\cV}\circ {\rm D} g_T(Q)[U'Q]=0$ with $U+U'=0$ ($U$ antisymmetric) implies that $U=0$.
\end{lem}
\begin{proof}
Since,
$$
(Q+tV)(Q+tV)'=I_d+t (VQ'+QV')+o(t),
$$
$V$ is a direction in the tangent space to ${\rm O}(d)$ at $Q$ if and only if $VQ'+QV'=0$. Equivalently,
$$
VQ'+QV'=(V,Q)\bullet I_d+(Q,V)\bullet I_d=K_{I_d,Q}(V)=0.
$$
Thus, the first condition $K_{I_d,Q}(V)=0$ simply restates that $V$ lies in the tangent space of ${\rm O}(d)$ at $Q$.

The proof of Proposition~\ref{prop:finiteisol} showed that, $\cU\subseteq \cV$ is Zariski open, then $\cG_T(\cU)$ is Zariski open (and so also open in the classical topology) in $\cG_T(\cU)$. Thus, if $Q$ is not isolated, every neighborhood of $Q$ must contain an element in $\cG_T(\cU)$ different than $Q$. In other words, the point $Q\in \cG_T(\cU)$ is not isolated if and only if there exists a tangent direction $V\neq 0$ such that 
$$
\pi_{\cV}((Q+tV)\bullet T)-\pi_{\cV}(Q\bullet T)\;=\;\pi_{\cV}((Q+tV)\bullet T)\;=\; o(t).
$$
Taking the limit $t\to 0$, we get that equivalently $\pi_\cV(K_{T,Q}(V))=0$. This shows that $Q$ is isolated if and only if no such non-trivial tangent direction exists. 
\end{proof}

\begin{rem}
In the examples of Section~\ref{sec:mainres}, for $T\in \cU\subseteq \cV$, we always had $\cG_T(\cU)=\cG_T(\cV)={\rm SP}(d)$. The proof of Proposition~\ref{prop:finiteisol} suggests that, at least in principle $\cG_T(\cU)$ could be finite but $\cG_T(\cV)$ could have components of positive dimension. In the proof of the next result, we crucially rely on the fact that we compute $\cG_T(\cU)$ rather than $\cG_T(\cV)$.
\end{rem}

Note that the dimension of the orthogonal group ${\rm O}(d)$ is $\binom{d}{2}$, which is then also the minimal number of constraints that need to be imposed in order to hope for identifiability. The main result of this section studies local identifiability with a model defined by the minimal number of $\binom{d}{2}$ constraints with
$$
\cI\;=\;\{(i,j,\ldots,j):\;\;1\leq i<j\leq d\}.
$$
We write $\cV^\circ=\cV(\cI)$. Denote 
\begin{equation}\label{eq:Bjkl}
B^{(j)}\;=\;[T_{klj\cdots j}]_{k,l<j}\;\in\; S^2(\R^{j-1})
\end{equation} and define $\cU^\circ\subset S^r(\R^d)$ as the set of tensors $T\in \cV^\circ$ such that, 
\begin{equation}\label{eq:genericirc}
\det\left(T_{j\cdots j}I_{j-1}-(r-1)B^{(j)}\right)\neq 0\quad\mbox{for all }j=2,\ldots,d.
\end{equation}
\begin{thm}\label{th:localijj}
If $T\in \cU^\circ$ then $|\cG_T(\cU^\circ)|<\infty$.
\end{thm}
\begin{proof}By Proposition~\ref{prop:finiteisol} it is enough to show that each point of $\cG_T(\cU^\circ)$ is isolated. By Lemma~\ref{lem:isolder}, equivalently for every $Q\in \cG_T(\cU^\circ)$, if $K_{I_d,Q}(V)=0$ and $\pi_\cV(K_{T,Q}(V))=0$ then $V=0$. By \eqref{eq:KTAsimpl},
$K_{I_d,Q}(V)\;=\;K_{I_d,I_d}(VQ')$. Thus, denoting $U=VQ'$, this condition is equivalent to saying that $U$ antisymmetric ($U+U'=0$). We will show that the conditions above imply that $U$ must be zero. By assumption, we have $U_{ii}=0$ and $U_{ij}=-U_{ji}$ for all $i\neq j$. Again using \eqref{eq:KTAsimpl}, we get $\pi_\cV(K_{T,Q}(V))=\pi_{\cV}(K_{Q\bullet T,I_d}(U))$. Denote $S:=Q\bullet T$. Since $Q\in \cG_T(\cU^\circ)$, in particular, $S\in \cU^\circ$. The condition $\pi_\cV(K_{S,I_d}(U))=0$ means that for every $\bs i=(i,j,\ldots,j)$ with $i<j$, $(K_{S,I_d}(U))_{ij\cdots j}=0$. More explicitly,
\begin{eqnarray*}
0&=&\sum_{l=1}^d U_{il}S_{lj\cdots j}+\sum_{l=1}^d U_{jl}S_{ilj\cdots j}+\cdots +\sum_{l=1}^d U_{jl}S_{ij\cdots jl}\\
&=&U_{ij}S_{j\cdots j}+(r-1)\sum_{l=1}^d U_{jl}S_{ilj\cdots j}\\
&=&-U_{ji}S_{j\cdots j}+(r-1)\sum_{l=1}^d U_{jl}S_{ilj\cdots j}
%	\\
%	&=&((d-1)A^{(j)}_{ii}-A^{(j)}_{jj})U_{ij}+(d-1)\sum_{l\neq i} A^{(j)}_{il}U_{lj}
\end{eqnarray*}
Let $u_j=(U_{j\,1},\ldots,U_{j \,j-1})$ for $j=2,\ldots,d$. Let first $j=d$. Using the matrix $B^{(d)}$ defined in \eqref{eq:Bjkl}  the equation above gives
$$
\left(S_{d\cdots d}I_{d-1}-(r-1)B^{(d)}\right)u_d\;=\;0.
$$
This has a unique solution $u_d=0$ if and only if 
$\det(S_{d\cdots d}I_{d-1}-(r-1)B^{(d)})\neq 0$, which holds by \eqref{eq:genericirc}. We have shown that  the last row of $U$ is zero. Now suppose that we have established that the rows $j+1,\ldots,d$ of $U$ are zero. If $j=1$, we are done by the fact that $U$ is antisymmetric. So assume $j\geq 2$. We will use the fact that $U_{jl}=0$ if $l\geq j$. For every $i<j$
$$
0=-U_{ji}S_{j\cdots j}+(r-1)\sum_{l\neq j} U_{jl}B^{(j)}_{il}=-U_{ji}S_{j\cdots j}+(r-1)\sum_{l< j} B^{(j)}_{il}U_{lj}.
$$
This again has a unique solution if and only if 
$	\det(S_{j\cdots j}I_{j-1}-(r-1)B^{(j)})\neq0$, which holds by \eqref{eq:genericirc}. Using a recursive argument, we conclude that $U=0$. 
\end{proof}

% \begin{ex}
% 	Consider $\cV\subseteq S^3(\R^2)$ given by $T_{112}=T_{121}=T_{211}=0$ and $T_{122}=T_{212}=T_{221}=0$. In this $\cV$, $T_{111}$ and $T_{222}$ can be the only nonzero entries of $T$. By direct computations, if $\cU\subset \cV$ is such that $T_{111}T_{222}\neq 0$ then $\cG_T(\cU)$ has eight elements that corresponds to all the $2\times 2$ sign permutation matrices. 
% \end{ex}

% This is an example, where $\cG_T$ is both finite and it has nice algebraic structure (it forms a group). But more complicated situations are possible as illustrated by the following cautionary example.

\begin{ex}\label{ex:simple3}
Consider $\cV^\circ\subseteq S^3(\R^2)$ given by $T_{122}=0$. Direct calculations show that, for any given generic $T$, there are $12$ orthogonal matrices such that $Q\bullet T\in \cV$. There are four elements given by the diagonal matrices together with $8$ additional elements that depend on $T$. So, for example, if $T_{111}=1$, $T_{222}=2$, and $T_{112}=3$ then the twelve elements are the four matrices $D$ and eight matrices of the form
$$
\frac{1}{5}D\begin{bmatrix}
3 & 4\\
4 & -3
\end{bmatrix}\qquad\mbox{ and }\qquad \frac{1}{\sqrt{2}}D\begin{bmatrix}
1 & 1\\
1 & -1
\end{bmatrix}.
$$ 
Going back to our original motivation, suppose $\varepsilon$ is a two-dimensional mean-zero random vector with ${\rm var}(\varepsilon)=I_2$. If we impose in addition that $\E \varepsilon_1\varepsilon^2_2=0$, then, even if we impose some genericity conditions, the matrix $A$ in \eqref{nica} is identified only up to the set of $12$ elements. Moreover, as illustrated above, these elements may look nothing like $A$ in the sense that they are not obtained by simple row permutation and sign swapping.  
%Here we took a generic $T$ because for specific $T$ some of these extra eight points may end up at infinity. This happens for example when $T_{111}=T_{222}=2$, $T_{112}=T_{121}=T_{211}=1$.
\end{ex}

% Example~\ref{ex:simple3} illustrates a situation, when $\cG_T$ is not equal to the set of signed permutations. In fact, the non-trivial permutation matrix does not lie in $\cG_T$ simply because the support of $T$ is not preserved under this permutation. 

% \begin{defn}\label{def:symmetry}
% 	We say that the set of zero-indices $\cI$ is symmetric if for every permutation matrix $P$ and for every $T\in \cV(\cI)$ it holds that $P\bullet T\in \cV(\cI)$ .
% \end{defn}
% If $r=2$, that is when $T\in S^2(\R^d)$ is a symmetric matrix, the only non-trivial symmetric index sets is $\cI=\{\{i,j\}:i\neq j\}$ (diagonal matrices) and $\cI=\{\{i,i\}:i=1.\ldots,d\}$ (symmetric matrices with zero diagonal). For higher order tensors we have more interesting choices. 
% \begin{rem}
% 	We now have two forms of symmetry which should not be confused. For example, if $T\in S^3(\R^3)$ then $(1,2,2)\in \cI$ implies that $(2,2,1)\in \cI$ by the fact that $T$ is a symmetric tensor. If $\cI$ is symmetric, then $(1,2,2)\in \cI$ implies also that  $(1,3,3)\in \cI$ for example.
% \end{rem}

% Note that $\dim {\rm O}(d)=\binom{d}{2}$ and so at least $\binom{d}{2}$  additional restrictions are required on $S^r(\R^d)$ in order to make sure that the associated $\cG_T$ is finite. For instance, since $\dim {\rm O}(2)=1$, if $d=2$ a single generic constraint should be enough. In Example~\ref{ex:simple3}, a single constraint $T_{122}=T_{212}=T_{221}=0$ on $S^3(\R^2)$, indeed resulted in a finite $\cG_T$. 

\begin{rem}
The set $\cG_T(\cU^\circ)$ is finite but, as illustrated by Example~\ref{ex:simple3}, it typically contains matrices that do not have an easy interpretation. In particular, if $d=2$ then $\cV^\circ$ is given by a single constraint $T_{12\cdots 2}=0$. In this case we can show that there are generically $4r$ \emph{complex} solutions (which generalized the number $12$ in the above example). There are $4$ solutions given by the elements of $\Z^2_2$ and $4(r-1)$ extra solutions, which do not have any particular structure. 
\end{rem}

We conclude the following result.
\begin{thm}\label{thm:localmomcum}
Consider the model \eqref{nica} with $\E \varepsilon =0$, ${\rm var}(\varepsilon)=I_d$ and suppose that either $\mu_r(\varepsilon)\in \cU^\circ$ or $\kappa_r(\varepsilon)\in \cU^\circ$. Then $A$ is locally identifiable. 
\end{thm}

\section{Moments and Cumulants  --- some useful properties}\label{appsec:cumulant}

We collect some results on moments and cumulants and their sample estimates that are used below for some of the proofs. 

\subsection{Combinatorial relationship between moments and cumulants}\label{app:cumulcomb}

Let $\bs\Pi_r$ be the poset of all set partitions of $\{1,\ldots,r\}$ ordered by refinement. For $\pi\in \bs\Pi_r$ we write $B\in \pi$ for a block in $\pi$. The number of blocks of $\pi$ is denoted by $|\pi|$. For example, if $r=3$ then $\Pi_3$ has 5 elements: $123$, $1/23$, $2/13$, $3/12$, $1/2/3$. They have 1, 2, 2, 2, and 3 blocks respectively. If $\bs i=(i_1,\ldots,i_r)$ then  $\bs i_B$ is a subvector of $\bs i$ with indices corresponding to the block $B\subseteq \{1,\ldots,r\}$. For any multiset $\{i_1, \ldots, i_r\}$ of the indices $\{1,\ldots,d\}$ we can relate the moments $\mu_r(Y)$ to the cumulants \citep[e.g.][]{speed1983cumulants}. 
\begin{equation}\label{eq:momink}
	[\mu_r(Y)]_{i_1, \ldots, i_r } 	= \sum_{\pi \in \Pi_r} \prod_{B \in \pi} [\kappa_{|B|}(Y)]_{\bs i_B} ~, 
\end{equation}
where $B$ loops over each block in a given partition $\pi$. For instance, for $r=3$ we have 
\[
[\mu_r(Y)]_{i_1, i_2, i_3 } = \kappa_{i_1 i_2 i_3} + \kappa_{i_1 i_2} \kappa_{i_3} + \kappa_{i_1 i_3} \kappa_{i_2} + \kappa_{i_2 i_3} \kappa_{i_1}  + \kappa_{i_1} \kappa_{i_2} \kappa_{i_3}~,
\] 
where we use the more convenient notation $\kappa_{i_1 \ldots i_l} = [\kappa_l(Y)]_{i_1 \ldots i_l}$. Similarly, from \cite{speed1983cumulants} we have 
\begin{equation}\label{eq:kinmom}
	[\kappa_r(Y)]_{i_1, \ldots, i_r} = \sum_{\pi \in \Pi_r} (-1)^{|\pi|-1} (|\pi| - 1)! \prod_{B \in \pi} [\mu_{|B|}(Y)]_{\bs i_B}~.	
\end{equation}
For example, 
$$
[\kappa_r(Y)]_{i_1, i_2, i_3}\;=\;\mu_{i_1 i_2 i_3}-\mu_{i_1}\mu_{i_2 i_3}-\mu_{i_2}\mu_{i_1 i_3}-\mu_{i_3}\mu_{i_1 i_2}+2\mu_{i_1}\mu_{i_2}\mu_{i_3}~,
$$
using $\mu_{i_1 \ldots  i_l} = [\mu_l(Y)]_{i_1 \ldots  i_l}$. 

%and note that for $(i_1,i_2,i_3)=(i,i,j)$ this formula gives
%$$
%[\kappa_r(Y)]_{i,i,j}\;=\;\mu_{i i j}-2\mu_{i}\mu_{ij}-\mu_{j}\mu_{ii}+2\mu_{i}^2\mu_{j}.
%$$
%Simple computational ways to get these formulae are discussed in Section~4.2.1 in \cite{zwiernik2016semialgebraic}.
The coefficients $(-1)^{|\pi|-1}(|\pi|-1)!$ in \eqref{eq:kinmom} have an important combinatorial interpretation, which we now briefly explain. If $\bs P$ is a finite partially ordered set (poset) with ordering $\leq$ we define the zeta function on $\bs P\times \bs P$ as $\zeta(x,y)=1$ if $x\leq y$ and $\zeta(x,y)=0$ otherwise. The M\"{o}bius function is then defined by setting $\mathfrak{m}(x,y)=0$ if $x\not\leq y$ and
$$
\sum_{x\leq z\leq y}\mathfrak{m}(x,z)\zeta(z,y)\;=\;\begin{cases}
	1 &\mbox{if }x=y,\\
	0 & \mbox{otherwise}.
\end{cases}
$$
Fixing a total ordering on $\bs P$, we can represent the zeta function by a matrix $Z$ and then the matrix $M$ representing the M\"{o}bius function is simply the inverse of $Z$. If this total ordering is consistent with the partial ordering of $\bs P$ then both $Z$ and $M$ are upper-triangular and have ones on the diagonal; see Section~4.1 in \cite{zwiernik2016semialgebraic} for more details. 

For the poset $\bs \Pi_r$ the M\"{o}bius function  satisfies for any $\rho\leq \pi$ ($\rho$ is a refinement of $\pi$)
\begin{equation}\label{eq:mobiuspart}
	\mathfrak{m}(\rho,\pi)\;=\;(-1)^{|\rho|-|\pi|}\prod_{B\in \pi}(|\rho_B|-1)!,
\end{equation}
where $|\rho_B|$ is the number of blocks in which $\rho$ subdivides the block $B$ of $\pi$. In particular, denoting by $\bs 1\in \bs\Pi_r$ the one-block partition, for every $\pi\in \bs\Pi_r$
$$
\mathfrak{m}(\pi,\bs 1)\;=\;(-1)^{|\pi|-1}(|\pi|-1)!.
$$

To explain how $\mathfrak{m}(\pi,\bs 1)$ appears in \eqref{eq:kinmom}, we recall the M\"{o}bius inversion formula, which becomes clear given the matrix formulation using $Z$ and $M=Z^{-1}$.
\begin{lem}[M\"{o}bius inversion theorem]\label{lem:mInvers}
	Let $\bs P$ be a poset. For two functions $c,d$ on $\bs P$, we have $d(x)=\sum_{y\leq x} c(y)$ for all $x\in \bs P$ if and only if $c(x)=\sum_{y\leq x}\mathfrak{m}(x,y)d(y)$.
\end{lem}
For example, this result gives the simple formula \eqref{eq:momink} that defines moments in terms of cumulants. 

\subsection{Laws of total expectation and cumulance}\label{app:ltc}

The law of total expectation is well known; for two random variables $X,H$ defined on the same probability space we have $\E X = \E[ \E(X | H) ]$. \cite{brillinger1969calculation} derives an analog result for cumulants.
\begin{prop}[Multivariate law of total cumulants]\label{prop:ltc}Let $\kappa_s(X|H)$ be the conditional $s$-th cumulant tensor of $X$ given a variable $H$. We have
	$$
	\kappa_r(X)\;=\;\sum_{\pi\in \bs\Pi_r}{\rm cum}\left((\kappa_{|B|}(X|H))_{B\in \pi}\right),
	$$
	where for $\bs i=(i_1,\ldots,i_r)$
	$$
	\left[{\rm cum}((\kappa_{|B|}(X|H))_{B\in \pi})\right]_{\bs i}={\rm cum}\left(({\rm cum}(X_{\bs i_B}|H))_{B\in \pi}\right).
	$$
\end{prop}
It is certainly hard to parse this formula at first so we offer a short discussion. The expression ${\rm cum}(({\rm cum}(X_{\bs i_B}|H))_{B\in \pi})$ on the right denotes the cumulant of order $|\pi|$ of the conditional variances ${\rm cum}(X_{\bs i_B}|H)$ for $B\in \pi$. A special case of this result is the law of total covariance. 
$$
[\kappa_2(X)]_{ij}\;=\;{\rm cov}(X_i,X_j)\;=\;\E({\rm cov}(X_i,X_j|H))+{\rm cov}(\E(X_i|H),\E(X_j|H)),
$$
where the first summand on the right corresponds to the partition $12$ and the second corresponds to the split $1/2$. Since there are five possible partitions of $\{1,2,3\}$ the third order cumulant can be given in conditional cumulants as
\begin{eqnarray*}
	[\kappa_3(X)]_{ijk}&=& \E({\rm cum}(X_i,X_j,X_k|H))+{\rm cov}(\E(X_i|H),{\rm cov}(X_j,X_k|H))\\
	&+& {\rm cov}(\E(X_j|H),{\rm cov}(X_i,X_k|H))+{\rm cov}(\E(X_k|H),{\rm cov}(X_i,X_j|H))\\
	&+&{\rm cum}(\E(X_i|H),\E(X_j|H),\E(X_k|H)).
\end{eqnarray*}
Proposition~\ref{prop:ltc} is useful for example if the components of $X$ are conditionally independent given $H$ in which case all mixed conditional cumulants vanish. Another scenario is when $X$ conditionally on $H$ is Gaussian, in which case all higher order conditional tensors vanish. 

\subsection{Estimating moments and cumulants}\label{app:kstats}

Given a sample $\{Y_s \}_{s=1}^n$, unbiased estimates for the $r$th order moment tensor $\mu_r(Y)$ are obtained by computing the sample moments 
\begin{equation}\label{eq:samplemom}
	[\hat{\bs \mu}_r]_{i_1\ldots i_r} = \frac{1}{n} \sum_{s=1}^n Y_{s,i_1}Y_{s,i_2} \ldots Y_{s,i_r}~. 
\end{equation}
Using our multilinear notation we can more compactly write 
\begin{equation}\label{eq:mominY}
	\hat{\bs \mu}_r=\frac{1}{n}\bs Y'\bullet I_r\;\in \;S^r(\R^d).
\end{equation}
where $I_r\in S^r(\R^n)$ is the identity tensor, that is, the diagonal tensor satisfying $(I_r)_{t\cdots t}=1$ for all $1\leq t\leq n$.

Unbiased estimates for the cumulants are computed using multivariate \textsf{k}-statistics \cite{speed1983cumulants}, which generalize classical \textsf{k}-statistics introduced by \cite{fisher1930moments}. For a collection of useful results on \textsf{k}-statistics see also \cite[Chapter 4]{mccullagh2018tensor}. 

Specifically, the entries of the $r$th order $\mathsf k$-statistic used to estimate the cumulant $[\kappa_r(Y)]_{i_1\ldots i_r}$ are given by (see \cite[(4.5)-(4.7)]{mccullagh2018tensor})
\begin{equation}\label{kstat} 
	[\mathsf k_r]_{i_1, \ldots, i_r} =  \frac1n \sum_{t_1=1}^n\cdots \sum_{t_r=1}^n\Phi_{t_1,\ldots,t_r}  Y_{t_1,i_1}\cdots Y_{t_r,i_r}
\end{equation}
with $\Phi\in S^r(\R^n)$ satisfying
$$
\Phi_{t_1\cdots t_r}\;=\;(-1)^{\nu-1}\frac{1}{\binom{n-1}{\nu-1}},
$$
where $\nu\leq n$ is the number of distinct indices in $(t_1,\ldots,t_n)$. Let $\bs Y\in \R^{n\times d}$ be the data matrix. More compactly, we have 
\begin{equation}\label{eq:kstatinY}
	\mathsf k_r\;=\;\frac1n \bs Y'\bullet \Phi\;\in \;S^r(\R^d).
\end{equation}
We note the following important result; see Proposition~4.3 in \cite{speed1986cumulants}.
\begin{prop}\label{prop:ustats}
	The $\mathsf k$-statistic in \eqref{kstat} forms a U-statistic. In particular, it is unbiased and it has the minimal variance among all unbiased estimators.
\end{prop}

Besides being unbiased and efficient, an additional benefit of working with $\mathsf k_r$ statistics is that there are several statistical packages available that compute them, e.g. \texttt{kStatistics} for \texttt{R} and \texttt{PyMoments} for \texttt{Python}. The first package uses the powerful machinery of umbral calculus to make the symbolic computations efficient \cite{di2009new}. 

\subsection{\textsf k-statistics and sample cumulants}\label{app:kstatisample}

For later considerations we need to understand better the relation between $\mathsf k_r$ and the natural plug-in estimator $\hat{\bs \kappa}_r$, which is obtained by first estimating the raw moments and then plugging them into \eqref{eq:kinmom}. The relevant sample moments that allow to compute $\hat{\bs \kappa}_r$ from \eqref{eq:kinmom} are summarized in $\hat{\bs \mu}_p$ for $p \leq r$. 

If $B\subseteq [n]$ then write $I_B$ for the identity tensor in $S^{|B|}(\R^n)$. For any partition $\pi\in \bs\Pi_r$ the tensor product $\bigotimes_{B\in \pi} I_B\in S^r(\R^n)$ satisfies
$$
\left[\bigotimes_{B\in \pi} I_B\right]_{t_1\cdots t_r}\;=\;\prod_{B\in \pi} \left[I_B\right]_{\bs t_B}\;=\;\begin{cases}
	1 & t_i=t_j\mbox{ whenever } i,j\in B\in \pi,\\
	0 & \mbox{otherwise}.
\end{cases}
$$
For every $\pi\in \bs\Pi_r$, define coefficients \begin{equation}\label{eq:cpi}
	c(\pi)\;=\;\sum_{\rho\leq \pi}\mathfrak{m}(\rho,\pi)(-1)^{|\rho|-1}\frac{1}{\binom{n-1}{|\rho|-1}}\;=\;n\sum_{\rho\leq \pi}\mathfrak{m}(\rho,\pi)\mathfrak{m}(\rho,\bs 1)\frac{1}{(n)_{|\rho|}},
\end{equation}
where $\mathfrak{m}$ is the M\"{o}bius function on $\bs\Pi_r$ given in \eqref{eq:mobiuspart} and $(n)_{k}=n(n-1)\cdots (n-k+1)$ is the corresponding falling factor.
\begin{lem}\label{lem:kstatinmom}
	We have
	$$
	\Phi\;=\;\sum_{\pi\in \bs\Pi_r}c(\pi)\bigotimes_{B\in \pi} I_B,
	$$
	which gives an alternative formula for \textsf{k}-statistics
	$$
	[\mathsf k_r]_{i_1, \ldots, i_r}\;=\; \sum_{\pi\in \bs\Pi_r}n^{|\pi|-1}c(\pi)\prod_{B\in \pi}\hat{\bs\mu}_{\bs i_B}.
	$$\end{lem}
\begin{proof}
	For any $t_1,\ldots,t_r$ let $\nu$ be the number of distinct elements in this sequence and let $\pi^*$ be the partition $[r]$ with $\nu$ blocks corresponding to indices that are equal. We have
	$$
	\left(\sum_{\pi\in \bs\Pi_r}c(\pi)\bigotimes_{B\in \pi} I_B\right)_{t_1\cdots t_r}\;=\;\sum_{\rho\leq \pi^*}c(\rho)\;=\;(-1)^{\nu-1}\frac{1}{\binom{n-1}{\nu-1}}\;=\;\Phi_{t_1\cdots t_r},
	$$
	where the first equality follows by the definition of $\pi^*$ and $\bigotimes_B I_B$, and the second equality follows directly by the M\"{o}bius inversion formula on $\bs\Pi_r$ as given in Lemma~\ref{lem:mInvers}.
	
	The second claim follows from the fact that 
	$$
	\mathsf k_r\;\overset{\eqref{eq:kstatinY}}{=}\;\frac1n \bs Y'\bullet \Phi\;=\;\frac1n \sum_{\pi\in \bs\Pi_r}c(\pi)\bigotimes_{B\in \pi}(\bs Y'\bullet I_B)\;\overset{\eqref{eq:mominY}}{=}\; \sum_{\pi\in \bs\Pi_r}n^{|\pi|-1}c(\pi)\bigotimes_{B\in \pi}\hat{\bs \mu}_B,
	$$
	where $\hat{\bs \mu}_B$ is the symmetric tensor containing all $|B|$ order sample moments among the variables in $B$. 
\end{proof}

In the analysis of the asymptotic difference between $\mathsf k_r$ and the plug-in estimator $\hat{\bs \kappa}_r$ we will use the following lemma.
\begin{lem}\label{lem:diffmc}
	For every $\pi\in \bs\Pi_r$ we have  
	$$
	n^{|\pi|-1}c(\pi)-\mathfrak{m}(\pi,\bs 1)\;=\;O(n^{-1}).
	$$
\end{lem}
\begin{proof}
	As we noted in the proof of Lemma~\ref{lem:kstatinmom}, the M\"{o}bius inversion formula in Lemma~\ref{lem:mInvers} gives that 
	\begin{equation}\label{eq:sumc}
		\sum_{\rho\leq \pi}c(\rho)=(-1)^{|\pi|-1}\frac{1}{\binom{n-1}{n-|\pi|}}.        
	\end{equation}
	Let $\bs 0\in \bs\Pi_r$ be the minimal partition into $r$ singleton blocks. By \eqref{eq:sumc}, applied to $\pi=\bs 0$,
	$$
	n^{r-1}c(\bs 0)=(-1)^{r-1}\frac{n^{r-1}}{\binom{n-1}{n-r}}\;=\;\mathfrak{m}(\bs 0,\bs 1)\frac{n^r}{(n)_r},
	$$
	where $(n)_r=n\cdots (n-r+1)$ is the corresponding falling factorial. In particular, $n^{r-1}c(\bs 0)=\mathfrak{m}(\bs 0,\bs 1)+O(n^{-1})$. Now suppose the claim is proven for all partitions with more than $l$ blocks. Let $\pi$ be a partition with exactly $l$ blocks. If $\rho<\pi$ then $|\rho|>l$ and $n^{|\rho|-1}c(\rho)=\mathfrak{m}(\rho,\bs 1)+O(n^{-1})$ so
	$$
	n^{l-1}c(\rho)=n^{l-|\rho|}n^{|\rho|-1}c(\rho)=n^{l-|\rho|}\mathfrak{m}(\rho,\bs 1)+O(n^{l-|\rho|-1})=O(n^{l-|\rho|}).
	$$
	This assures that
	$$
	n^{l-1}\sum_{\rho\leq \pi} c(\rho)\;=\;n^{l-1}c(\pi)+O(n^{-1}).
	$$
	Using \eqref{eq:sumc} in the same way as above, we get that $n^{|\pi|-1}c(\pi)=\mathfrak{m}(\pi,\bs 1)+O(n^{-1})$ and now the result follows by recursion.
\end{proof}

\subsection{Vectorizations of tensors}\label{app:vec}

The dimension of the space of symmetric tensors $S^r(\R^d)$ is $\binom{d+r-1}{r}$. Like for symmetric matrices, it is often convenient to view $T\in S^r(\R^d)$ as a general tensor in $\R^{d\times \cdots\times d}$. In this case ${\rm vec}(T)\in \R^{d^r}$ is a vector obtained from all the entries of $T$. 

Throughout the paper we largely avoided vectorization. This operation is however hard to circumvent in the asymptotic considerations. If we make a specific claim about the joint Gaussianity of the entries of a random tensor $T$, we could use a more invariant approach of  \cite{eaton}. However, using vectorizations, makes the calculations more direct without referring to abstract linear algebra.  

In this context we also often rely on the matrix-vector version of the tensor equation $S=A\bullet T$
\begin{equation}\label{eq:kronecker}
	{\rm vec}(S)\;=\;A^{\otimes r}\cdot {\rm vec}(T),	
\end{equation}
where $A^{\otimes r}=A\otimes \cdots\otimes A$ if the $r$-th Kronecker power of $A$.

\subsection{Asymptotic distribution of sample statistics}\label{app:kstatsas}

To derive the asymptotic distribution of the minimum distance estimators in Section~\ref{sec:inference} we require the asymptotic distribution of the sample moments or the $\mathsf k$-statistics. 

Specifically, we need the joint distribution of the sample moments/cumulants that are restricted to zero. To derive these in a convenient way we define $m_{S,T}:\R^{d\times d}\to S^2(\R^d)\oplus S^r(\R^d)$ to be 
\begin{equation}\label{eq:geng}
	m_{S,T}(A)\;=\; (A\bullet S-I_d,\;A\bullet T)~. 
\end{equation}
The cases that we consider are $S=h_2(Y)$, $T=h_r(Y)$, in which case we write simply $m(A)$, and $S=\hat{h}_2$, $T=\hat{h}_r$, in which case we write $\hat m_n(A)$. Here, $\hat{h}_r$ denotes either the sample moments, denoted by $\hat{\bs \mu}_r$, or the $r$th order $\mathsf k$-statistic, denoted by $\mathsf k_r$, which are computed from a given sample $\{Y_s \}_{s=1}^n$ as discussed above. It is worth pointing out that these results generalize existing results \citep[e.g.][]{Jammalamadaka2021} for the asymptotic analysis of cumulant estimates to higher order tensors.  \\ 

\noindent \textbf{Sample moments}

The sample moments of $Y$ are defined as in \eqref{eq:samplemom}. When using moments the distance measure $\hat m_n(A)$ (see \eqref{eq:geng}) depends on the tensors $\hat{\bs \mu}_2$ and $\hat{\bs \mu}_r$. As formalized in the lemma below, we have that under suitable moment assumptions that 
\begin{equation}\label{consistencymoments}
	\hat{\bs \mu}_p \stackrel{p}{\to} \mu_p(Y) \qquad \forall \ p \leq r ~,
\end{equation}
and 
\begin{equation}\label{eq:assnormalmoments}
	\sqrt{n}{\rm vec}(\hat{\bs \mu}_2 - \mu_2(Y) , \hat{\bs \mu}_r - \mu_r(Y) ) \stackrel{d}{\to} N(0, V  )~,
\end{equation}
where $V$ is the asymptotic variance matrix with entries 
\[
V_{\bs i, \bs j} = {\rm cov}(Y_{i_1}\cdots Y_{i_k} , Y_{i_1}\cdots Y_{i_l}  ) \qquad k,l \in \{2,r\}~.
\]
We note that $V$ is not positive definite as vectorizing the tensors does not imply that the entries are unique. We will correct for this when required below. Further $V$ can be consistently estimated by its sample version. 

Given \eqref{eq:assnormalmoments} we can use \eqref{eq:kronecker} to derive the limiting distribution of $\hat m_n(A)$ for moments. We have 
\begin{align}
	\nonumber	\sqrt{n}{\rm vec}(\hat m_n(A) - m(A)) &= [A^{\otimes 2}, A^{\otimes r}] \cdot \sqrt{n} {\rm vec}(\hat{\bs \mu}_2 - \mu_2(Y) , \hat{\bs \mu}_r - \mu_r(Y) ) \\
	&\stackrel{d}{\to} N(0, A^{2,r} V A^{2,r'} ) 
\end{align}
where $ A^{2,r} = [A^{\otimes 2}, A^{\otimes r}]$. Let the asymptotic variance matrix be denoted by 
\begin{equation}\label{sigma2rmu}
	\Sigma^{2,r}_\mu = A^{2,r} V A^{2,r'}~. 
\end{equation}
\smallskip 

\noindent \textbf{$\mathsf k$-statistics}

Next, we provide analog steps for the $\mathsf k$-statistics. First, let $\bs\mu_{\leq r}$ be the vector containing all moments of a random vector $Y$ of order up to $r$ (it has dimension $\binom{d+r}{r}$). 
%The fact that $\mathsf k$-statistics are U-statistics (cf. Proposition~\ref{prop:ustats}) implies a couple of useful results for our analysis. The most important one is that, under minor assumptions, $\sqrt{n}{\rm vec}(\mathsf k_r-\kappa_r(Y))$ is asymptotically normal with mean zero. %In this section, the main result provides the form of the asymptotic covariance matrix (see Lemma~\ref{replacek} below).
Formula \eqref{eq:kinmom} gives an explicit function for $\kappa_r(Y)$ in terms of $\bs \mu_{\leq r}$. For the vectorized tensor $\kappa_r(Y)$ we define the Jacobian $F = \nabla_{\bs \mu'_{\leq r}} {\rm vec}(\kappa_r(Y))$, which is a $d^r\times \binom{d+r}{r}$ matrix. This matrix is not a full rank but only because $\kappa_r(Y)$ is a symmetric tensor which has many repeated entries. The submatrix obtained from $F$ by taking the rows corresponding to the unique entries of $\kappa_r(Y)$ has full row rank. This follows because for any two $r$-tuples $1\leq i_1\leq \cdots\leq i_r\leq d$ and $1\leq j_1\leq \cdots\leq j_r\leq d$ we have that 
$$    \tfrac{\partial \kappa_{i_1\cdots i_r}}{\partial \mu_{j_1\cdots j_r}}=\begin{cases}
	1 & \mbox{if } (i_1,\ldots,i_r)=(j_1,\ldots,j_r),\\
	0 & \mbox{otherwise},
\end{cases}
$$
and so, this submatrix contains the identity matrix. 

%We collected together sample moments \eqref{eq:samplemom} of different orders to form $\hat{\bs \mu}_{\leq r}$. The covariance between two sample moments $\hat{\bs\mu}_{\bs i}$ and $\hat{\bs \mu}_{\bs j}$, for $\bs i=(i_1,\ldots,i_k)$, $\bs j=(j_1,\ldots,j_l)$, $k,l\leq r$, is
%\begin{eqnarray}\label{eq:covmuij}
%{\rm cov}(\hat{\bs\mu}_{\bs i},\hat{\bs \mu}_{\bs j})&=&\frac{1}{n^2}\sum_{s,t=1}^n {\rm cov}(Y_{s,i_1}\cdots Y_{s,i_k},Y_{t,j_1}\cdots Y_{t,j_l})\\
%\nonumber &=&\frac1n {\rm cov}(Y_{i_1}\cdots Y_{i_k},Y_{j_1}\cdots Y_{j_l}).
%\end{eqnarray}
Under suitable moment conditions we have 
\[
\hat{\bm \mu}_{\leq r} \stackrel{p}{\to} \bm \mu_{\leq r} \qquad \text{and} \qquad \sqrt{n}\,(\hat{\bm \mu}_{\leq r} - \bm \mu_{\leq r}) \stackrel{d}{\to} N(0, H )   
\]
and since $\bm \mu_{\leq r}$ only includes unique moments we may conclude that $H$ is positive definite.
% and, by \eqref{eq:covmuij}, its entries are
%$$
%H_{\bs i,\bs j}\;=\;{\rm cov}(Y_{i_1}\cdots Y_{i_k},Y_{j_1}\cdots Y_{j_l})
%$$ 
%for every $k,l\leq r$ and $\bs i=(i_1,\ldots,i_k)$, $\bs j=(j_1,\ldots,j_l)$. 
%$\widehat H $ which has entries of the form 
%\[
%\frac{1}{n} \sum_{s=1}^n Y_{s,i_1} \ldots Y_{s,i_r} Y_{s,j_1}  \ldots Y_{s,j_r}~. 
%\]
%Collecting the preliminaries: (i) the map $\bm \mu_{\leq r}$ is continuously differentiable, (ii) $\Sigma_{\bm \mu_{\leq r}}$ is positive definite, (iii) $\sqrt{n}(\hat{\bm \mu}_r - \bm \mu_{\leq r}) \stackrel{d}{\to} N(0, H  )$; it follows from the delta method that

As in Appendix~\ref{app:kstatisample}, denote $\hat{\bs \kappa}_r$ to be the image of $\hat{\bs \mu}_{\leq r}$ under the map \eqref{eq:kinmom}. It then follows from the delta method that 
\begin{equation}\label{eq:sampled}
	\sqrt{n}\,{\rm vec}(\hat{\bs \kappa}_r -  \kappa_r(Y)) \;\;\stackrel{d}{\to}\;\; N(0 ,   F H  F'  )~. 
\end{equation}
We emphasize that this particular estimator of cumulants will not be of direct interest.  What we need is the form of the covariance matrix in \eqref{eq:sampled}. We will show that \textsf{k}-statistics $\mathsf k_r$ have the same asymptotic distribution.
\begin{lem}\label{replacek}
	If $\mathbb E \| Y_s \|^{2r} < \infty$ we have that 
	\[
	\sqrt{n} \,{\rm vec}(\mathsf k_r - \kappa_r(Y) ) \;\;\stackrel{d}{\to}\;\; N(0 ,   F  H  F '  )~.
	\]
\end{lem}
\begin{proof}
	By \eqref{eq:sampled} and Slutsky lemma, it is enough to show that $\sqrt{n} \,(\mathsf k_r - \hat{\bs \kappa}_r ) \stackrel{p}{\to} 0$. By Lemma~\ref{lem:kstatinmom},
	$$
	[\mathsf k_r-\hat{\bs \kappa}_r]_{i_1\cdots i_r}\;=\;\sum_{\pi\in \bs\Pi_r} (n^{|\pi|-1}c(\pi)-\mathfrak{m}(\pi,\bs 1))\prod_{B\in \pi}\hat{\bs \mu}_{\bs i_B},
	$$
	where the coefficients $c(\pi)$ are defined in \eqref{eq:cpi}. By Lemma~\ref{lem:diffmc}, $n^{|\pi|-1}c(\pi)-\mathfrak{m}(\pi,\bs 1)\;=\;O(n^{-1})$ for all $\pi\in \bs \Pi_r$ and so in particular
	$$
	\sqrt{n}(n^{|\pi|-1}c(\pi)-\mathfrak{m}(\pi,\bs 1))\;=\;o(1).
	$$

	% For this proof, we consider a more involved version of \eqref{kstat} that explicitly uses the M\"{o}bius function on the partition lattice (CITATION)
	% \begin{equation}\label{kstat2} 
		% [\mathsf k_r]_{i_1, \ldots, i_r} =  \sum_{\pi \in \Pi_r} (-1)^{|\pi|-1}  c_\pi \prod_{B \in \pi} \hat \mu_{[{\bf{i}}_j | j \in B ]} 
		% \end{equation}
	% with 
	% \[
	% c_\pi = n^{|\pi|} \sum_{b_1 = 1 }^{|B_1|} \sum_{b_2 = 1 }^{|B_2|} \ldots \sum_{b_{|\pi|} = 1 }^{|B_{|\pi|}|} \frac{(\sum_{j=1}^{|\pi|} b_j -1 )! }{(n)_{\sum_{j=1}^{|\pi|} b_j }} \left( \prod_{j=1}^{|\pi|} \left\{ _{\ b_j}^{|B_j|}  \right\} (b_j-1)! \right) ~,
	% \]
	% where $|\pi|$ denotes the number of blocks in partition $\pi$, $(n)_k = n(n-1)\ldots (n-k + 1)$ is the falling factorial and for integers $a,b$, $\left\{ _{a}^{b}  \right\}$ denotes the Stirling number of the second kind. Consider any entry $i_1, \ldots, i_r$, we have 
	% 	\[
	% 	\sqrt{n}(\hat{\bs \kappa}_{i_1, \ldots, i_r} - [\mathsf k_r]_{i_1, \ldots, i_r}) = \sum_{\pi \in \Pi_r} (-1)^{|\pi|-1} \sqrt{n}[ (|\pi| - 1)! - c_\pi] \prod_{B \in \pi} \hat \mu_{[i_j | j \in B ]}
	% 	\]
	% 	We show that $ \sqrt{n}[ (|\pi| - 1)! - c_\pi] \to 0$, which is sufficient as 
	Under the stated moment assumption $\hat{\bs \mu}_{\bs i_B} = O_p(1)$ and so $[\mathsf k_r-\hat{\bs \kappa}_r]_{i_1\cdots i_r}=o_P(1)$,
	which completes the proof.
	% 	We can separate 
	% 	\[
	% 	\sqrt{n}[ (|\pi| - 1)! - c_\pi]  = T_1 + T_2
	% 	\]
	% 	where $T_1$ takes the first term and the sum corresponding to $b_j=1$ for all $j$ from $c_\pi$. We have 
	% 	\begin{align*}
		% 		T_1 &= \sqrt{n}[(|\pi| - 1)! - (|\pi| - 1)! n^{|\pi|} / (n)_{|\pi|} ] \\
		% 		&=  (|\pi| - 1)! \sqrt{n}(1 - n^{|\pi|} / (n)_{|\pi|} ) \\
		% 		&= O\left( n^{|\pi| - 1/2} \right) /  (n)_{|\pi|} = o(1) ~,
		% 	\end{align*}
	% 	as $(n)_{|\pi|} = O(n^{|\pi|} )$. The second term is given by 
	% 	\begin{align*}
		% 		T_2 =&  n^{|\pi|} \sum_{b_1 = 1 }^{|B_1|} \sum_{b_2 = 1 }^{|B_2|} \ldots \sum_{b_{|\pi|} = 1 }^{|B_{|\pi|}|} \bm 1(b_j \neq 1 \ \forall \ j ) \frac{(\sum_{j=1}^{|\pi|} b_j -1 )! }{(n)_{\sum_{k=1}^{|\pi|} b_j }} \left( \prod_{j=1}^{|\pi|} \left\{ \begin{array}{c} |B_j| \\ b_j \end{array} \right\} (b_j-1)! \right)
		% 	\end{align*}
	% 	We note that for each term of the sum there is an integer $\delta > 0$ such that    
	% 	\[
	% 	\sqrt{n} n^{|\pi|} / (n)_{|\pi| + \delta} \times O(1) = o(1)~.
	% 	\]
	% 	We have $T_1 \to 0$ and $T_2 \to 0$ which completes the proof. 
\end{proof}

By Lemma~\ref{replacek}, every linear transformation of $\sqrt{n}{\rm vec}(\mathsf{k}_r-\kappa_r(Y))$ will be also Gaussian. We will be in particular interested in transformations $A^{\otimes r} {\rm vec}(\mathsf{k}_r-\kappa_r(Y))$ as motivated by the multilinear action of $A$ on $S^r(\R^d)$ (cf. \eqref{eq:kronecker}). We have
% or equivalently
% $$
% {\rm vec}(\mathsf k_r)\;=\;\frac1n (\bs Y^{\otimes r})' {\rm vec}(\Phi)
% Lemma~\ref{replacek} shows 
% \[
% \sqrt{n} (f(\hat{\bm \mu}_r) - {\rm vec}(\mathsf k_r) ) \stackrel{p}{\to} 0~,
% \]
% and by asymptotic equivalence \citep[e.g.][page 123]{Rao1973}
% \[
% \sqrt{n}{\rm vec}\left[ \mathsf k_r - \kappa_r(Y) \right] \stackrel{d}{\to} N(0 ,   F  H  F '  ) ~.
% \]
% Next, for $A \bullet \mathsf k_r $ it immediately follows that  
\begin{align*} 
	\sqrt{n} A^{\otimes r} {\rm vec}(\mathsf k_r  -  \kappa_r(Y))  &\;\;\stackrel{d}{\to}\;\; N(0 , A ^{\otimes r}  F  H (A^{\otimes r}F) '  )~. 
\end{align*}
%This asymptotic covariance matrix will get a name
%\begin{equation}\label{eq:Sigr}
%    \Sigma_{\mathsf k}^r\;=\;A ^{\otimes r}  F  H (A^{\otimes r}F) '.
%\end{equation}
A similar analysis can be given if $\kappa_r(Y)$ is complemented with some other lower order cumulants. We will use one version of that. Let $F^{2,r}$ be the Jacobian matrix of the transformation from $\bs\mu_{\leq r}$ to cumulants ${\rm vec}(\kappa_2(Y), \kappa_r(Y))\in \R^{d^2+d^r}$. By exactly the same arguments as above we get 
\begin{equation}\label{eq:kappa2r}
	\sqrt{n}\, {\rm vec}(\mathsf{k}_2- \kappa_2(Y),\mathsf{k}_r-\kappa_r(Y))\;\;\stackrel{d}{\to}\;\; N(0 ,   F^{2,r}  H  (F^{2,r}) '  )~.
\end{equation}
Recall from \eqref{eq:geng} that $m_{S,T}(A)=(A\bullet S-I_d,\;A\bullet T)$ and consider $m(A)$ and $\hat m_n(A)$ as defined by cumulants and $\mathsf k$-statistics in Section~\ref{sec:inference}. 
\begin{equation}\label{eq:mdiff}
	{\rm vec}(\hat m_n(A)-m(A))\;=\;[A^{\otimes 2},A^{\otimes r}]\cdot {\rm vec} (\mathsf{k}_2-\kappa_2(Y),\mathsf{k}_r-\kappa_r(Y)).
\end{equation}
We will write $A^{2,r}=[A^{\otimes 2},A^{\otimes r}]$ and, using \eqref{eq:kappa2r}, we immediately conclude
$$
\sqrt{n}\,{\rm vec}(\hat m_n(A)-m(A))\;\;\stackrel{d}{\to}\;\; N(0 ,   A^{2,r}F^{2,r}  H  (A^{2,r} F^{2,r}) '  )~.
$$

Let this asymptotic covariance matrix be denoted by 
\begin{equation}\label{eq:Sig2r}
	\Sigma^{2,r}_{\mathsf k} = A^{2,r}F^{2,r}  H  (A^{2,r} F^{2,r}) '.
\end{equation}

We summarize these general results in the following lemma adopting the notation required for the main text. 
\begin{lem}\label{lem:convergence}
	Suppose $\{Y_s \}_{s=1}^n$ is i.i.d. 
	\begin{enumerate} 
		\item if $\mathbb E \| Y_s \|^{r} < \infty$, then $\hat{\bs \mu}_p - \mu_p(Y) \stackrel{p}{\to} 0$ and  $\mathsf k_p - \kappa_p(Y) \stackrel{p}{\to} 0$ for all $p \leq r$.   
		
		%	\item if $\mathbb E \| Y_s \|^{2r} < \infty$, then 
		%	\[
		%		\sqrt{n} A^{\otimes r} {\rm vec}(\mathsf k_r  - \kappa_r(Y)) \stackrel{d}{\to} N(0 , \Sigma^r ) 
		% 	\]
		% 	with $\Sigma^r$ defined in \eqref{eq:Sigr}. 
		\item if $\mathbb E \| Y_s \|^{2r} < \infty$, then 
		\[
		\sqrt{n} {\rm vec}(\hat m_n(A)-m(A)) \stackrel{d}{\to} N(0, \Sigma_h^{2,r})  \qquad h=\mu, \mathsf k ~, 
		\] 
		where $h=\mu$ or $h = \mathsf k$ depends on whether $\hat m_n(A)$ and $m(A)$ are based on moments or cumulants, respectively. We have that the moment based variance $\Sigma_\mu^{2,r}$ is defined in \eqref{sigma2rmu} and the cumulant based variance  $\Sigma_{\mathsf k}^{2,r}$ in \eqref{eq:Sig2r}.
		% 	 $A _{(2r)} = [ (A )^{\otimes 2} , (A )^{\otimes r}]$. 
	\end{enumerate}
\end{lem}

\section{Additional inference tools}\label{appsec:inferencedetails}

In this section we complement the inference Section~\ref{sec:inference} with some additional tools that can be used to select the appropriate moment/cumulant zero restrictions in a data driven way.

\subsection{Testing over-identifying restrictions}

While zero restrictions on higher order moments or cumulants can be motivated from several angles (cf. the discussion in Section~\ref{sec:motivation}), it is useful to test ex-post whether the restrictions indeed appear to hold in a given application. In the setting where $d_g$ is strictly greater then $d^2$, i.e. the total number of restrictions is larger when compared to the number of parameters in $A$, we can conduct a general specification test following the approach outlined in \cite{Hansen1982}. 

\begin{prop}\label{prop:hansenJ}
If the conditions of Proposition~\ref{prop:assnormal} hold we have that as $n\to \infty$
\[
\Lambda_n\;:=\; n\hat L_{\widehat \Sigma_n^{-1}}(\widehat A_{\widehat \Sigma_n^{-1}}) \;\stackrel{d}{\to}\; \chi^2(d_g - d^2)~. 
\]
\end{prop}
The proposition implies that $\Lambda_n$ can be viewed as a test statistic for verifying the identifying restrictions. Specifically, when $g(QA_0) \neq 0$ the statistic $\Lambda_n$ diverges under most alternatives. That said, if any of the other assumptions fails, e.g. the moment condition, the statistic will also fail to converge to a $\chi^2(d_g - d^2)$ random variable. This implies that we should view Proposition~\ref{prop:hansenJ} as a general test for model misspecification. 

A more refined test can be formulated when sufficient confidence exists in a subset of the identifying restrictions. To set this up let $g(A) = (g_1(A) , g_2(A))$ be a partition of the identifying moment/cumulant restrictions such that $g_1(A)$ has dimension $d_{g_1} \geq d^2$. We propose a test for whether the additional identifying restrictions $g_2(A)$ are valid. 

Denote as earlier $\Lambda_n= n\hat L_{\widehat \Sigma_n^{-1}}(\widehat A_{\widehat \Sigma_n^{-1}})$ and let $\Lambda_n^0$ be similarly defined by for a smaller set of identifying restrictions. 

\begin{prop}\label{prop:subset}
If the conditions of Proposition~\ref{prop:assnormal} hold we have that as $n\to \infty$
\[
C_n \;:=\; \Lambda_n-\Lambda_n^0 \;\stackrel{d}{\to} \;\chi^2(d_{g} - d_{g_1} )~.
\]
%  	where $\hat L_1(\widehat A_{\widehat \Sigma_{11}^{-1}})$ denotes the objective function using only $g_1(A)$ identifying restrictions and $\widehat A_{1,\widehat \Sigma_{11}^{-1}}$ is the corresponding minimizer.  
\end{prop}    
The test statistic $C_n$ allows to verify whether adding the additional identifying restrictions $g_2(A)$ is valid. The test rejects when $g_2(Q A_0 ) \neq 0$, that is, when the additional restrictions do not hold.

\section{Computing the asymptotic variance}\label{appsec:assvar} 

In this section we give computational details for estimating the asymptotic variance matrices $\Sigma$ and $S$ as defined in Proposition \ref{prop:assnormal}. Starting with $\Sigma$ (see equation \eqref{eq:Sigma}) we first recall that $\Sigma$ is really $\Sigma_h$ and the expression depends on whether moment or cumulant restrictions are used. For moments we obtained 
\[
\Sigma_\mu = D_{\mathcal I}^{2,r} \Sigma_\mu^{2,r} D_{\mathcal I}^{2,r'}~ \qquad \text{with} \qquad \Sigma_\mu^{2,r} = A^{2,r} V A^{2,r'} ~,
\]
and for cumulants 
\[
\Sigma_\kappa = D_{\mathcal I}^{2,r} \Sigma_\kappa^{2,r} D_{\mathcal I}^{2,r'}~ \qquad \text{with} \qquad \Sigma_\kappa^{2,r} = A^{2,r}F^{2,r}  H  (A^{2,r} F^{2,r}) ' ~,  
\]
where $D_{\mathcal I}^{2,r}$ is a selection matrix that selects the corresponding to the unique entries in $S^r(\R^d)\oplus \cV^\perp$, $V$ and $H$ contain the covariances of ${\rm vec}( \hat{\bs \mu}_2,\hat{\bs \mu}_r)$ and $\hat{\bs \mu}_{\leq r}$, respectively,  $A^{2,r}=[A^{\otimes 2},A^{\otimes r}]$ and $F^{2,r}$ is the Jacobian matrix of the transformation from $\bs{ \mu}_{\leq r}$ to cumulants $(\kappa_2, \kappa_r)$, see Section~\ref{app:kstatsas} for explicit definitions. 

The moment matrices $V$ and $H$ and the Jacobian matrix $F^{2,r}$ can be estimated by replacing the population moments of $\mu_r(Y)$ by the sample moments $\hat{\bs \mu}_{r}$. Further, $A^{2,r}=[A^{\otimes 2},A^{\otimes r}]$ can replaced by its estimate $\widehat A_{W_n}^{2,r}=[\widehat A_{W_n}^{\otimes 2},\widehat A_{W_n}^{\otimes r}]$. Combining we obtain the estimates 
\[
\widehat \Sigma_{\mu,n} = D_{\mathcal I}^{2,r} \widehat V D_{\mathcal I}^{2,r'}  \quad \text{and} \quad \widehat \Sigma_{\mathsf k,n} = D_{\mathcal I}^{2,r}  \widehat A_{W_n}^{2,r} \widehat F^{2,r}  \widehat H  (\widehat A_{W_n}^{2,r} \widehat F^{2,r}) ' D_{\mathcal I}^{2,r'}~.
\]
While these plug-in estimators are conceptually straightforward, for cumulants it does require determining the Jacobian $F^{2,r}$, which can be a tedious task. 

Therefore, for cumulant restriction we recommend estimating $\Sigma_h$ using a simple bootstrap. Let $\hat{\bm  \varepsilon}_n = \widehat A_{W_n} {\bf Y}_n $ denote the $n \times 1$ vector of residuals. We can resample these residuals (with replacement) to get $\hat{\bm  \varepsilon}_n^*$ and construct bootstrap draws of $\hat g_n(\widehat A_{W_n})$, say $g^*_n$. Repeating this $B$ times allows to compute the bootstrap variance estimate 
\[
\widehat \Sigma_n /n  = \frac{1}{B} \sum_{b=1}^B (g^{*,b}_n - \bar g^{*}_n) (g^{*,b'}_n- \bar g^{*}_n)' \qquad \text{with} \quad \bar g^{*}_n = \frac{1}{B}  \sum_{b=1}^B g^{*,b}_n~.
\]
The $1/n$ comes from the definition $\Sigma = \lim_{n\to \infty} {\rm var}(\sqrt{n} \hat g_n(QA_0) )$. Using the bootstrap has the benefit that no additional analytical calculations are needed and evaluating $g^{*,b}_n$ only requires computing the sample statistics $\mu_p(Y)$ or $\mathsf k_p$, for $p=2,r$, for each bootstrap draw $\hat{\bm  \varepsilon}_n^*$. 
The validity of the bootstrap follows as we have a random sample $\{Y_i \}$, $\widehat A_{W_n}$ is $\sqrt{n}$-consistent for $A$ and asymptotically normal (cf Propositions \ref{prop:consist} and \ref{prop:assnormal}) and $\hat g_n(A)$ is a polynomial map in $A$ and hence smooth.    

While the bootstrap is conceptually attractive, it is worth nothing that, at least in principle, the covariance between two $\mathsf{k}$-statistics  $\mathsf{k}_{i_1\cdots i_r}$ and $\mathsf{k}_{j_1\cdots j_r}$ can be computed exactly for any given sample size using the general formula for cumulants of $\mathsf{k}$-statistics as given in Section~4.2.3 in \cite{mccullagh2018tensor}. Although the covariance is arguably the simplest cumulant, the formula still involves combinatorial quantities that are hard to obtain. Given the moments of $Y$, we could also use the explicit formula \eqref{eq:kstatinY} to obtain the covariance in any given case by noting that  
$$
\E {\rm vec}(\mathsf{k}_r){\rm vec}(\mathsf{k}_r)'\;=\;\frac{1}{n^{2}}\E\left[(\bs Y')^{\otimes r}{\rm vec}(\Phi){\rm vec}(\Phi)' \bs Y^{\otimes r}\right].
$$
Note however that ${\rm vec}(\Phi)$ has $n^r$ entries with many of them repeated, so the naive approach is very inefficient. An efficient, perhaps umbral, approach to these symbolic computations could help to obtain better estimates of $A$.

Next, we compute the asymptotic variance $S = (G' \Sigma^{-1} G )^{-1}$, where $G = G(QA_0)$ is the Jacobian matrix corresponding to $g(A)$. Combining the estimator $\widehat A_{W_n}$ and the map \eqref{eq:jaccol} provides the estimate for $G$. Combining this an estimate for $\Sigma$ as defined above allows to estimate $S$.

\section{Additional numerical results}\label{appsec:simsadditional} 

In this section we provide additional simulation results that complement Section~\ref{sec:sims}. We compare the performance of the minimum distance estimators across different measures, dimensions and sample sizes. 

\subsection{Alternative performance measures} 

We start by providing the same results as in the main text but now measuring the accuracy of the different procedures in terms of the Frobenius distance $d_F$ \citep[e.g.][]{chenbickel} which is often referred to as the minimum distance index \citep[e.g.][]{IlmonenNordhausenOjaOllila.10} and can be defined as  
\[
d_F(\widehat A_{W_n} , A_0 ) = \min_{Q \in SP(d)} \frac{1}{d^2} \| \widehat  A_{W_n}^{-1}QA_0  - I_d \|_F~, 
\]
where the scaling by $d^2$ is an arbitrary choice. 

For this distance measure Tables \ref{tab:estimatesbaselinecv_frobi} and \ref{tab:estimatesbaselinese_frobi} replicate Tables \ref{tab:estimatesbaselinecv} and \ref{tab:estimatesbaselinese} from the main text. We find that the minimum distance estimator based on the reflectionally invariant restrictions remains to perform well across all specifications. Also for skewed densities the minimum distance estimator based on the diagonal third order tensor restrictions performs well. 

For the common variance model, i.e. Table \ref{tab:estimatesbaselinecv_frobi}, there are a few differences with respect to the Amari errors that are worth pointing out. First, TICA performs relatively less well. Further inspection showed that is largely due to the small sample size and the performance of TICA improves considerably when $n$ increases.  Second, some of the ICA methods (e.g. FastICA and JADE) perform well for $t(5)$, SKU and KU densities.  

For the multiple scaled elliptical model, i.e. Table \ref{tab:estimatesbaselinese_frobi}, the results are very similar when compared to the Amari errors, and the minimum distance estimator based on the reflectionally invariant restrictions is always preferred.

\subsection{Larger experiments common variance model} 

In the main text in Figure \ref{fig:commonvariance} we showed the results for the common variance model with $d=5$ and $n=200,1000$ corresponding to two specific distributions for the errors $\eta_i$: the $t(5)$ distribution as well as the Bi-Modal distribution $BM$. Here we show the same results but also include the other densities from Table \ref{tab:ng_dist}. 

Specifically, Figures \ref{fig:commonvariance1}-\ref{fig:commonvariance3} show all experiments that we conducted for the common variance model. The following additional results are worth mentioning. First, when the true errors correspond to the normal distribution the variances of all estimators are large and do not shrink noticeably when $n$ increases. The reason under normal errors for $\eta_i$ the deviations from the Gaussian distribution of $\varepsilon_i = \tau \eta_i$, with independent $\tau \sim {\rm gamma}(1,1)$, is very close to the Gaussian distribution and hence the parameters are poorly identified. 

Second, for $n = 1000$ we find that the diagonal tensor restrictions based on the third moments work well for the Skewed Unimodal (SKU) density. For $n=200$ the evidence is not convincing, but for larger sample sizes these restrictions in combination with the efficient weighting matrix yield good performance. Only TICA, which assumes that $K$ is known, leads to better performance. 

Third, in general TICA works well for Student's $t$ type densities like, $t(5)$ and Kurtotic Unimodal (KU). The reason is that, in addition to exploiting knowledge of $K$, the objective function of TICA is close in shape to the Student's $t$ density \cite[see][equation 3.10]{hyvarinen2001topographic}. As such TICA behaves like the MLE estimator for these densities. 

Fourth, for all other densities which impose larger deviations from the Student $t$ shape the estimators that were based on the reflectional invariant restrictions always perform better. The benefits are most clearly shown for bi modal densities.

Fifth, using the efficient weighting matrix shows most advantages for large sample sizes. The reason is that estimating the efficient weighting matrix accurately requires a large sample size. This improvement in weighting matrix accuracy is directly reflected in the Amari errors. 

\subsection{Larger experiments multiple scaled elliptical} 

Next, we revisit the nICA model with multiple scaled elliptical errors as presented in  \eqref{eq:generalmultielliptical}. For this model comparative simulation results were shown in Section \ref{ssec:multscaledellip} for $d=2$ and $n=200$. Here we consider the specifications where $d=5$ and $n=200,1000$. Figures  \ref{fig:scaledelliptical1}-\ref{fig:scaledelliptical3} show the results. Overall, the results for the scaled elliptical components model are quite similar across the densities for $\eta_i$. As in Table \ref{tab:estimatesbaselinese} the means of the Amari errors are roughly equal, but there exist some variations in the variances. First, except for the Gaussian density (which does not yield an identified model) when $n$ increases the variances generally decrease. Second, the evidence in favor of the efficient weighting matrix is mixed often the identity weighting matrix is preferred. This is most likely due to the fact that the multiple scaled elliptical model has quite heavy tails which may invalidate the moment assumptions needed for the consistent estimation of the weighting matrix, or at least reduce the accuracy of the weighting matrix estimate.

\begin{table}
	\caption{\sc Minimum Distance: Common Variance Model   }
	\begin{center}
		\begin{tabular}{lcccccccccc} 
			\hline\hline 
			& \multicolumn{10}{c}{Non-Independent Components Analysis} \\[0.5ex]
			Method       & $\mathcal N$ & $t(5)$ & SKU & KU & BM & SBM & SKB & TRI & CL & ACL \\   \hline 
			$\mu_3^{d,I} $   & 0.16 &   0.13 &   0.13  &  0.13  &  0.18 &   0.20  &  0.16  &  0.18  &  0.16  &  0.15 \\
			$\mu_3^{d, \widehat \Sigma_n^{-1}} $ &0.14  &  0.12 &   0.11  &  0.12 &   0.15  &  0.17 &   0.13 &   0.15  &  0.13 &   0.13 \\
			$\mu_4^{r,I}$ & 0.11   & 0.12  &  0.12  &  0.11  &  0.10 &   0.08 &   0.11  &  0.10  &  0.12  &  0.11 \\
			$\mu_4^{r,\widehat \Sigma_n^{-1}}$ & 0.11  &  0.11  &  0.11  &  0.10 &   0.09  &  0.05  &  0.10 &   0.08  &  0.10  &  0.11 \\ 
			& & & & & & & & & & \\
			TICA & 0.19  &  0.18  &  0.18  &  0.18  &  0.23 &   0.25 &   0.22  &  0.24 &  0.20  &  0.20 \\ 
			%& & & & & & & & & & \\ 
			\hline\hline 
			& \multicolumn{10}{c}{Independent Components Analysis} \\ [0.5ex]
			Method  	& $\mathcal N$ & $t(5)$ & SKU & KU & BM & SBM & SKB & TRI & CL & ACL \\   \hline  
			Fast  & 0.14 &   0.09 &   0.11  &  0.08 &   0.20 &   0.25 &   0.18   & 0.21  &  0.15  &  0.15 \\
			JADE & 0.15  &  0.10  &  0.11   & 0.07  &  0.23  &  0.26 &   0.21   & 0.23  &  0.18  &  0.17 \\
			Kernel  & 0.15   & 0.11 &   0.12&   0.10 &   0.20  &  0.25&    0.18   & 0.21  &  0.16  &  0.16 \\
			ProDen   &0.15   & 0.79  &  0.30 & 0.64 &   0.23  & 0.60 &   0.21   & 0.26  &  0.17  &  2.78 \\	
			Efficient  & 0.14&    0.17 &   0.16 &   0.17 &   0.14 &    0.15 &   0.14  &  0.14  &  0.14&    0.14 \\
			NPML	 & 0.14  &  0.14  &  0.14   & 0.14   & 0.16  &  0.15   & 0.15  &  0.15  &  0.15  &  0.15 \\	
			\hline\hline 
			& & & & & & & & & & \\
			\multicolumn{11}{l}{ \begin{minipage}{12.5cm} \footnotesize \emph{Notes:} The table reports the average Minimum Distance Index (across $S=1000$ simulations) for data sampled from the common variance model \eqref{eq:generalcommonvariance} with $d=2$ and $n=200$. The columns correspond to the different errors considered for the components of $\eta$, see Table \ref{tab:ng_dist}. The top panel reports the errors for the minimum distance methods and Topographical ICA (TICA). For the minimum distance methods we consider diagonal ($d$) and reflectionally invariant ($r$) restrictions for different order tensors $\mu_3, \mu_4$, combined with weighting matrices $W_n = I_d, \widehat \Sigma_n^{-1}$. The bottom panel reports comparison results for different independent component analysis methods: FastICA \citep {Hyvarinen.99}, JADE \cite{CardosoSouloumiac1993}, kernel ICA \citep{BachJordan.03}, ProDenICA \citep{HastieTibshirani.02}, efficient ICA \citep{chenbickel} and non-parametric ML ICA \citep{SamworthYuan_2012}. \end{minipage}  }
		\end{tabular} 
	\end{center}
	\label{tab:estimatesbaselinecv_frobi}
\end{table}

\begin{table}
	\caption{\sc Minimum Distance: Scaled Elliptical \ \ \ \ \  \ \ \ \ \ \ \ \ }
	\begin{center}
		\begin{tabular}{lcccccccccc} 
			\hline\hline 
			& \multicolumn{10}{c}{Non-Independent Components Analysis} \\[0.5ex]
			Method       & $\mathcal N$ & $t(5)$ & SKU & KU & BM & SBM & SKB & TRI & CL & ACL \\   \hline 
			$\mu_3^{d, I} $   &0.14  &  0.14  &  0.14 &   0.14  &  0.14 &   0.14  &  0.14  &  0.14  &  0.14 &   0.14 \\
			$\mu_3^{d, \widehat \Sigma_n^{-1}} $   &0.14  &  0.13  &  0.13  &  0.13  &  0.14  &  0.14 &   0.14 &   0.14 &   0.13 &   0.13 \\
			$\mu_4^{r, I} $   & 0.08   & 0.09  &  0.09   & 0.08 &   0.08 &   0.08  &  0.09  &  0.09 &   0.09 &   0.09 \\
			$\mu_4^{r, \widehat \Sigma_n^{-1}} $   &0.09  &  0.09  &  0.09   & 0.09  &  0.09  &  0.09 &   0.08 &   0.08 &   0.08  &  0.08 \\
			& & & & & & & & & & \\
			TICA & 0.15   & 0.16 &   0.15 &   0.16  &  0.14  &  0.15  &  0.15 &  0.15 &   0.15   & 0.15 \\
			& & & & & & & & & & \\\hline\hline 
			& \multicolumn{10}{c}{Independent Components Analysis} \\ [0.5ex]
			Method  	& $\mathcal N$ & $t(5)$ & SKU & KU & BM & SBM & SKB & TRI & CL & ACL \\   \hline  
			Fast  &0.14   & 0.13   & 0.14 &   0.14   & 0.13 &   0.14 &   0.14  &  0.14 &    0.13 &   0.14 \\
			JADE & 0.14 &   0.14 &   0.14  &  0.14   & 0.14  &  0.14 &   0.15   & 0.14 &    0.14 &   0.14 \\
			Kernel & 0.14   & 0.14 &   0.14 &   0.14 &   0.14 &   0.14  & 0.14  &  0.14 &    0.14 &   0.14 \\
			ProDen   &0.14  &  0.14  &  0.14 &   0.14  &  0.14 &   0.14   & 0.14   & 0.14 &    0.14 &   0.15 \\	
			Efficient &0.14 &   0.14 &   0.14 &   0.14 &   0.14 &   0.14 &   0.14   & 0.14&    0.14 &   0.14 \\
			NPML   & 0.14   & 0.14  &  0.14   & 0.14  &  0.14  &  0.14  &  0.14   & 0.14&    0.14 &   0.14 \\	
			\hline\hline 
			& & & & & & & & & & \\
			\multicolumn{11}{l}{ \begin{minipage}{12.5cm} \footnotesize \emph{Notes:} The table reports the average Minimum Distance errors (across $S=1000$ simulations) for data sampled from the multiple scaled elliptical model \eqref{eq:generalmultielliptical} with $d=2$ and $n=200$. The columns correspond to the different errors considered for the components of $\eta$, see Table \ref{tab:ng_dist}. The top panel reports the errors for the minimum distance methods and Topographical ICA (TICA). For the minimum distance methods we consider diagonal ($d$) and reflectionally invariant ($r$) restrictions for different order tensors $\mu_3, \mu_4$, combined with weighting matrices $W_n = I_d, \widehat \Sigma_n^{-1}$. The bottom panel reports comparison results for different independent component analysis methods: FastICA \citep {Hyvarinen.99}, JADE \cite{CardosoSouloumiac1993}, kernel ICA \citep{BachJordan.03}, ProDenICA \citep{HastieTibshirani.02}, efficient ICA \citep{chenbickel} and non-parametric ML ICA \citep{SamworthYuan_2012}. \end{minipage}  }
		\end{tabular}
	\end{center}
	\label{tab:estimatesbaselinese_frobi}
\end{table}

\clearpage

\begin{figure}
	\caption{\sc  Common Variance Experiments}\label{fig:commonvariance1}
	\begin{center}
		\begin{subfigure}{	
				\includegraphics[scale=0.45]{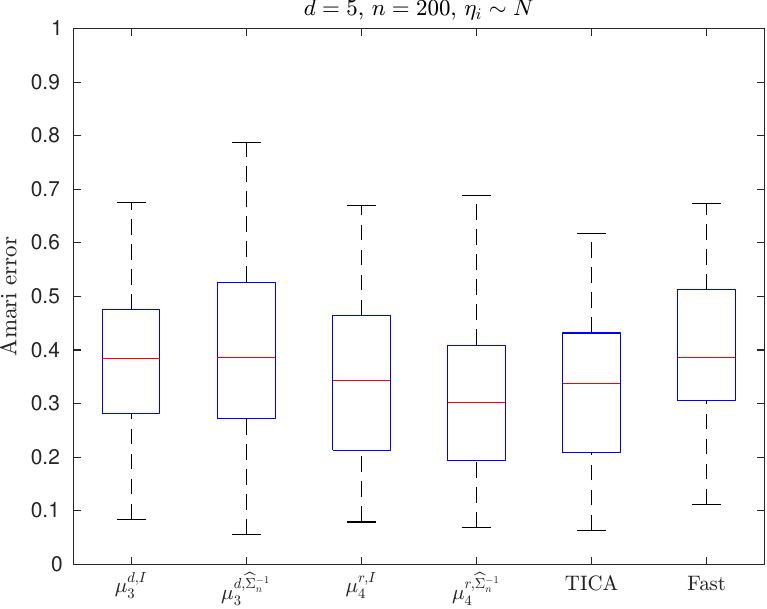}\hfill} 
		\end{subfigure}
		\begin{subfigure}{	
				\includegraphics[scale=0.45]{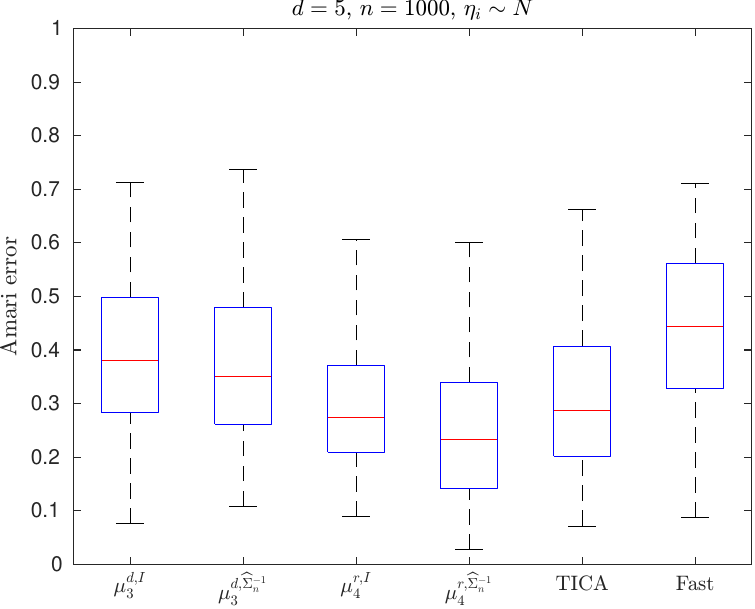}} 
		\end{subfigure}
	\begin{subfigure}{	
			\includegraphics[scale=0.45]{./Figures/Figure_commonvariance_d=5_n=200_eta=t5_amari}\hfill} 
	\end{subfigure}
	\begin{subfigure}{	
			\includegraphics[scale=0.45]{./Figures/Figure_commonvariance_d=5_n=1000_eta=t5_amari}} 
	\end{subfigure}
	\begin{subfigure}{	
			\includegraphics[scale=0.45]{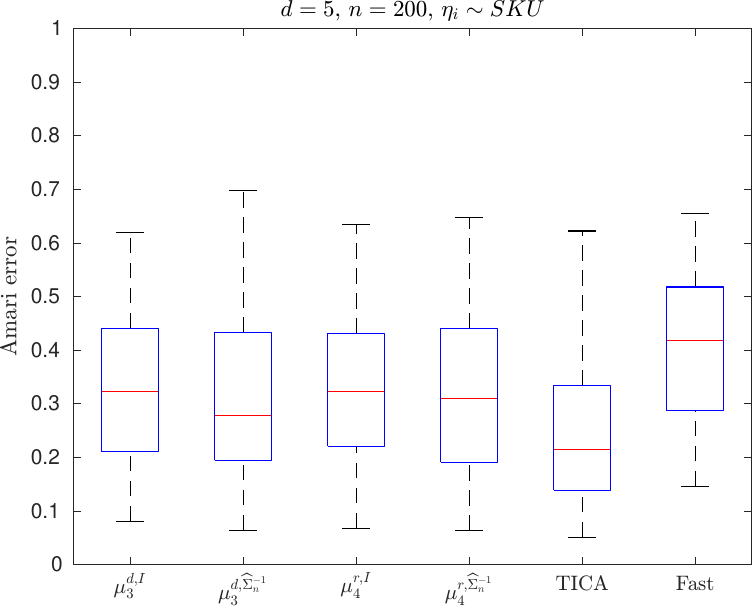}\hfill} 
	\end{subfigure}
	\begin{subfigure}{	
			\includegraphics[scale=0.45]{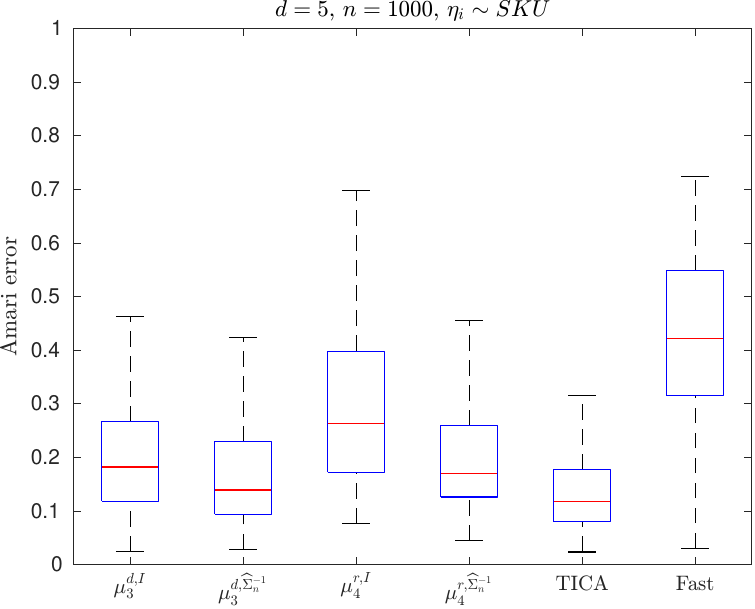}} 
	\end{subfigure}
	\begin{subfigure}{	
			\includegraphics[scale=0.45]{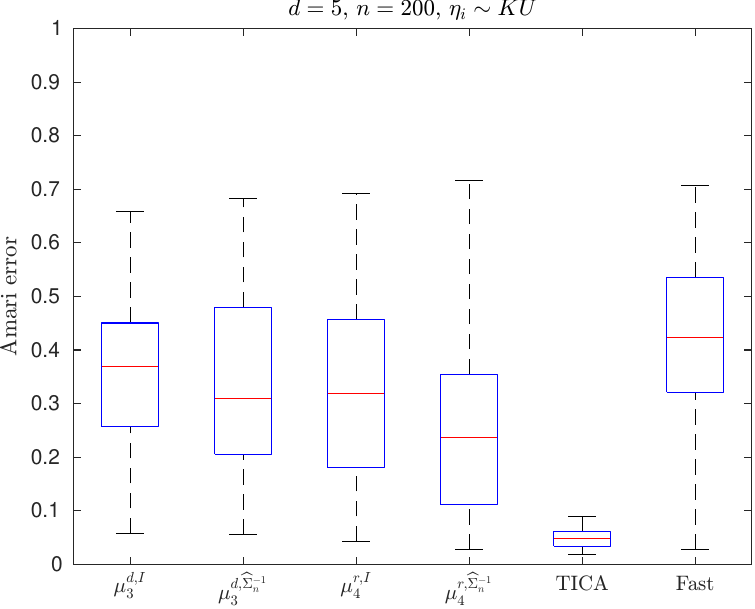}\hfill} 
	\end{subfigure}
	\begin{subfigure}{	
			\includegraphics[scale=0.45]{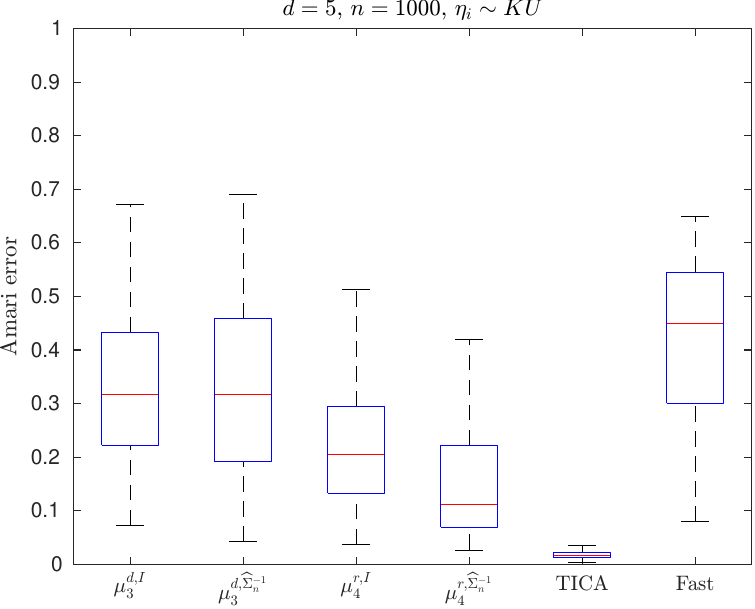}} 
	\end{subfigure}
	\end{center}
	{\begin{minipage}{12.5cm} \footnotesize \emph{Notes:} The figure shows the boxplots for the Amari errors (across $S=100$ simulations) for data sampled from the common variance model \eqref{eq:generalcommonvariance}. The different settings for the simulations designs are described in the titles and the $x$-labels indicate the different estimation methods used. \end{minipage} }
\end{figure}

\begin{figure}
	\caption{\sc  Common Variance Experiments}\label{fig:commonvariance2}
	\begin{center}
		\begin{subfigure}{	
				\includegraphics[scale=0.45]{./Figures/Figure_commonvariance_d=5_n=200_eta=BM_amari}\hfill} 
		\end{subfigure}
		\begin{subfigure}{	
				\includegraphics[scale=0.45]{./Figures/Figure_commonvariance_d=5_n=1000_eta=BM_amari}} 
		\end{subfigure}
		\begin{subfigure}{	
				\includegraphics[scale=0.45]{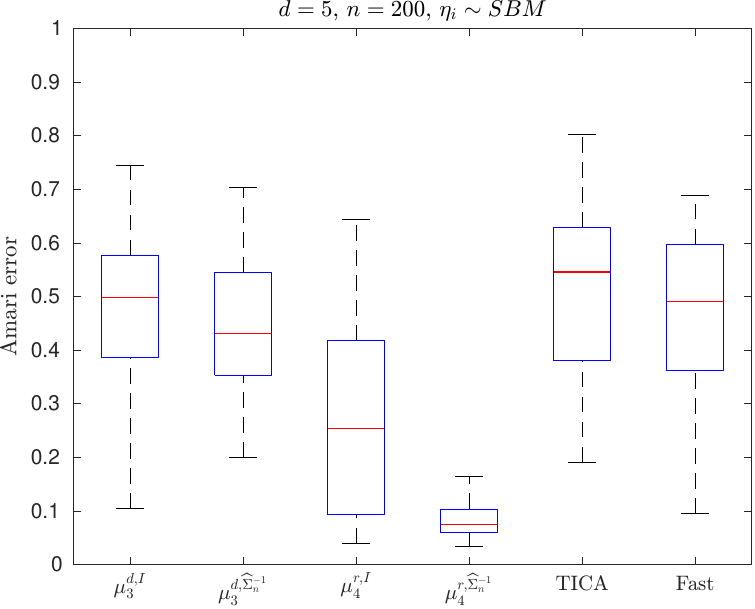}\hfill} 
		\end{subfigure}
		\begin{subfigure}{	
				\includegraphics[scale=0.45]{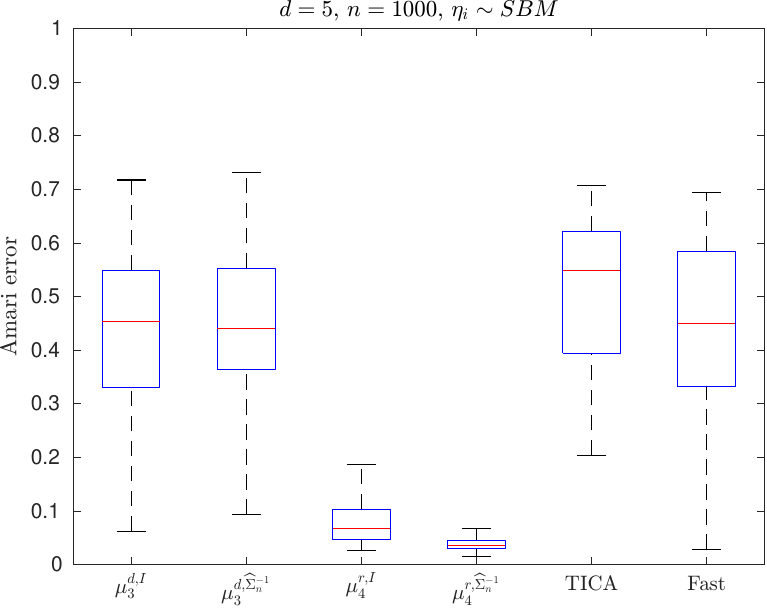}} 
		\end{subfigure}
		\begin{subfigure}{	
				\includegraphics[scale=0.45]{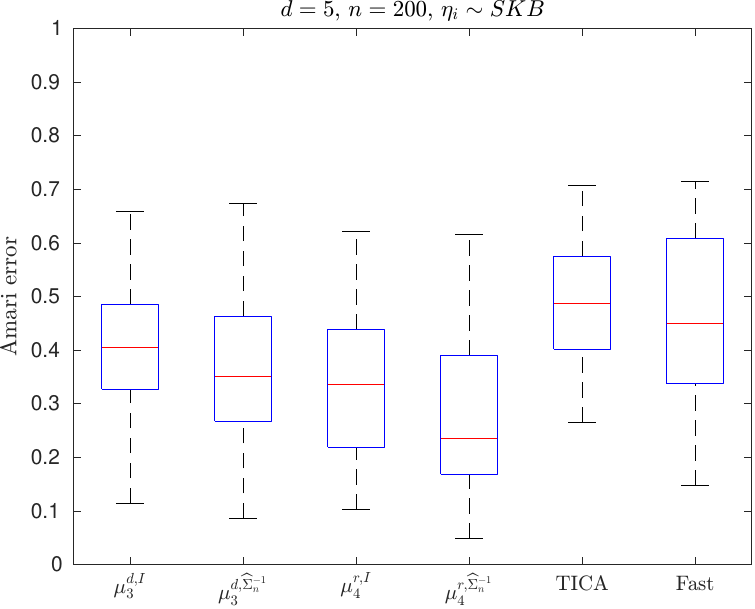}\hfill} 
		\end{subfigure}
		\begin{subfigure}{	
				\includegraphics[scale=0.45]{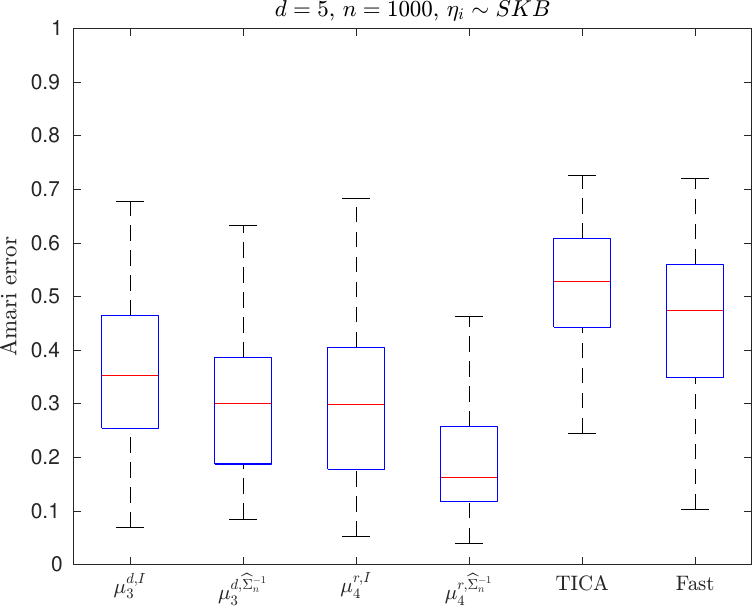}} 
		\end{subfigure}
		\begin{subfigure}{	
				\includegraphics[scale=0.45]{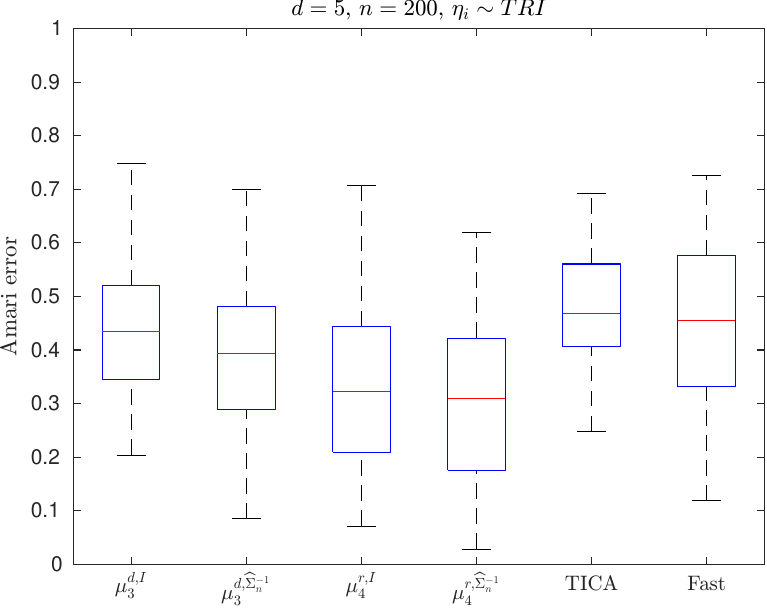}\hfill} 
		\end{subfigure}
		\begin{subfigure}{	
				\includegraphics[scale=0.45]{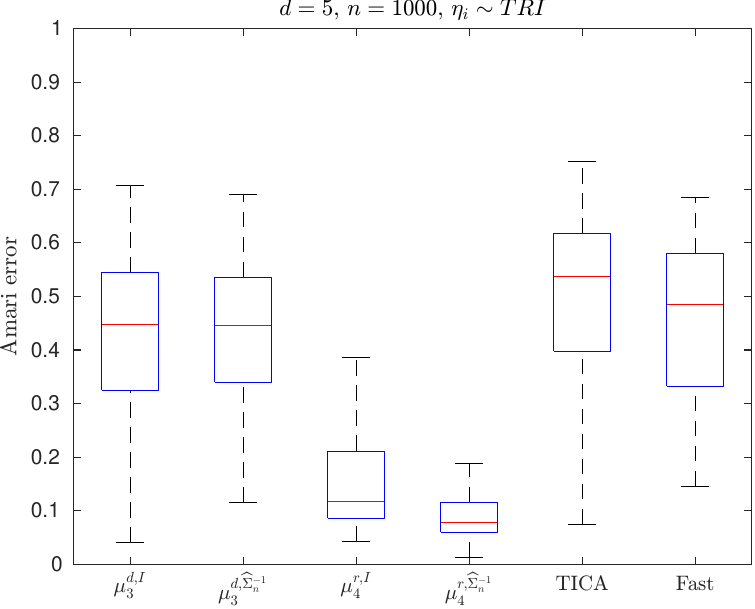}} 
		\end{subfigure}
	\end{center}
	{\begin{minipage}{12.5cm} \footnotesize \emph{Notes:} The figure shows the boxplots for the Amari errors (across $S=100$ simulations) for data sampled from the common variance model \eqref{eq:generalcommonvariance}. The different settings for the simulations designs are described in the titles and the $x$-labels indicate the different estimation methods used. \end{minipage} }
\end{figure}

\begin{figure}
	\caption{\sc  Common Variance Experiments}\label{fig:commonvariance3}
	\begin{center}
		\begin{subfigure}{	
				\includegraphics[scale=0.45]{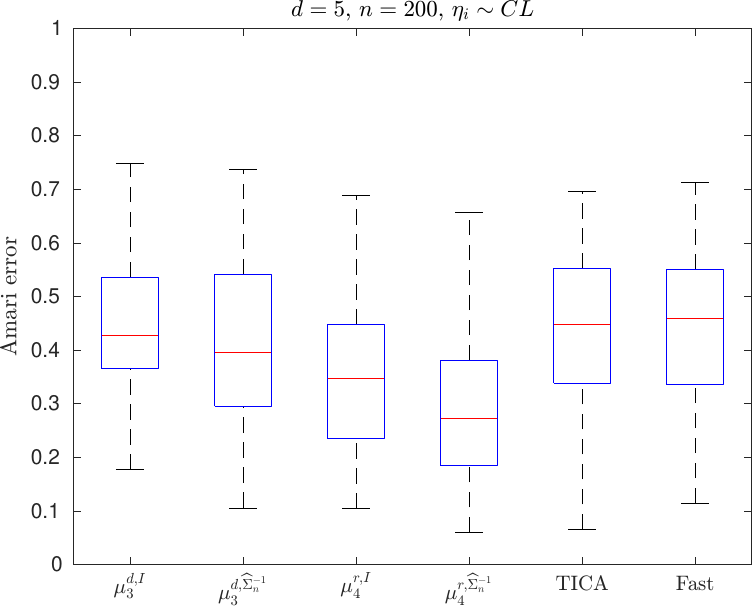}\hfill} 
		\end{subfigure}
		\begin{subfigure}{	
				\includegraphics[scale=0.45]{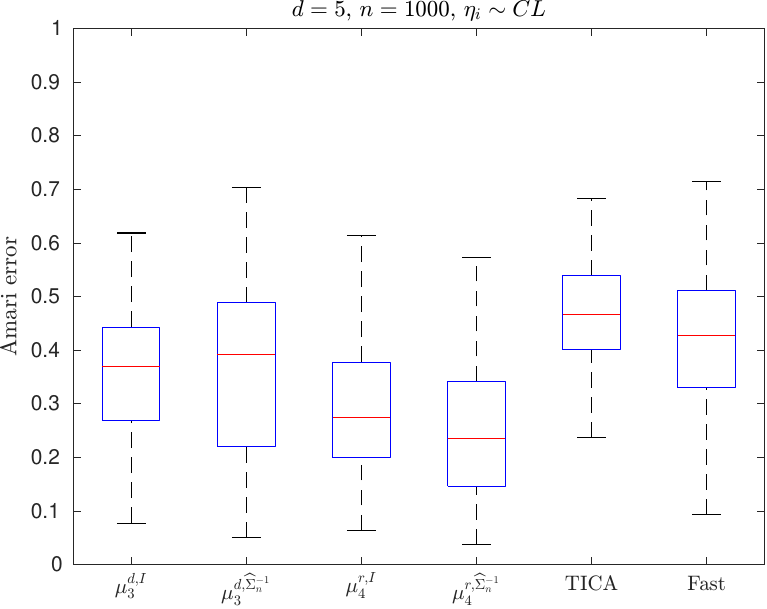}} 
		\end{subfigure}
		\begin{subfigure}{	
				\includegraphics[scale=0.45]{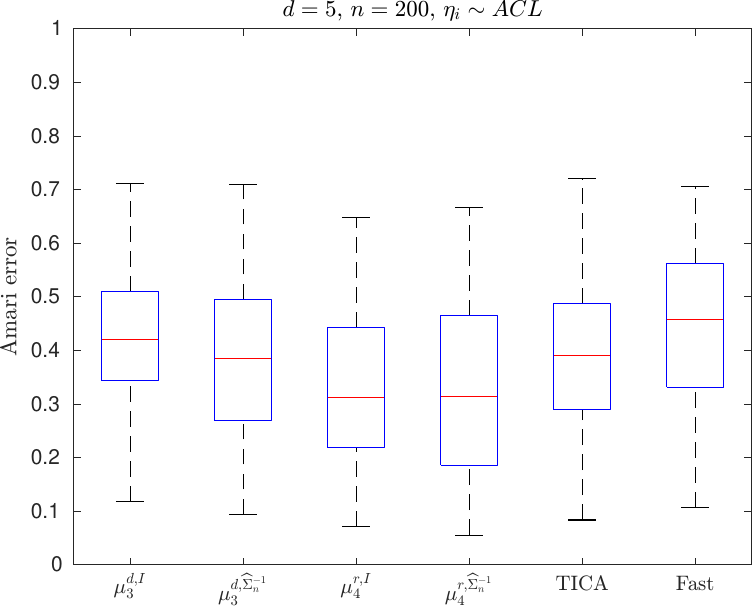}\hfill} 
		\end{subfigure}
		\begin{subfigure}{	
				\includegraphics[scale=0.45]{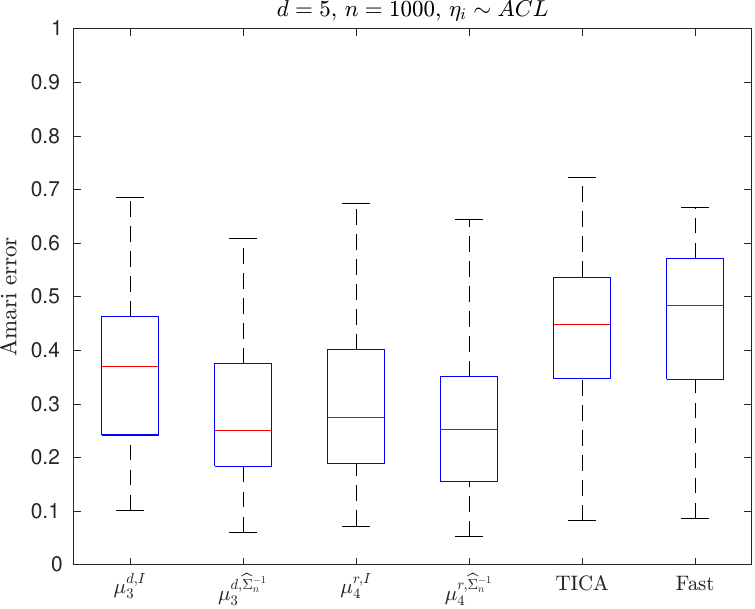}} 
		\end{subfigure}
	\end{center}
	{\begin{minipage}{12.5cm} \footnotesize \emph{Notes:} The figure shows the boxplots for the Amari errors (across $S=100$ simulations) for data sampled from the common variance model \eqref{eq:generalcommonvariance}. The different settings for the simulations designs are described in the titles and the $x$-labels indicate the different estimation methods used. \end{minipage} }
\end{figure}

\begin{figure}
	\caption{\sc  Scaled Elliptical Experiments}\label{fig:scaledelliptical1}
	\begin{center}
		\begin{subfigure}{	
				\includegraphics[scale=0.45]{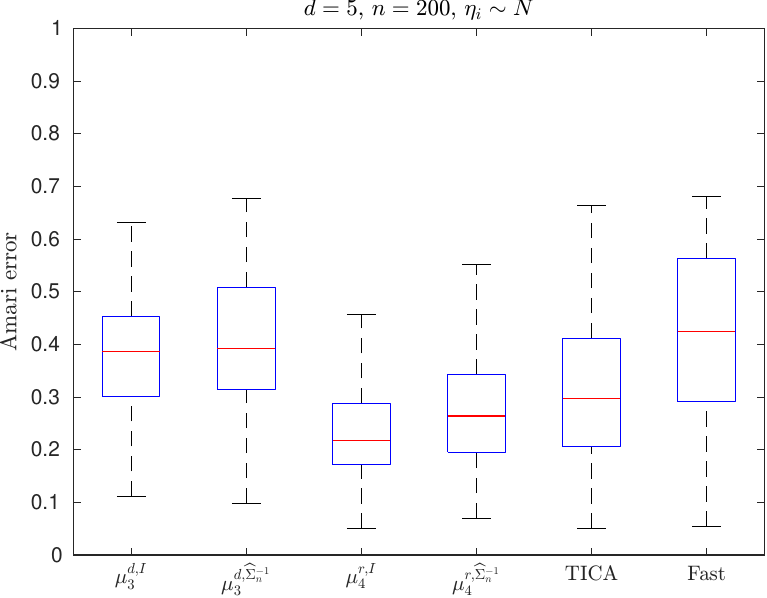}\hfill} 
		\end{subfigure}
		\begin{subfigure}{	
				\includegraphics[scale=0.45]{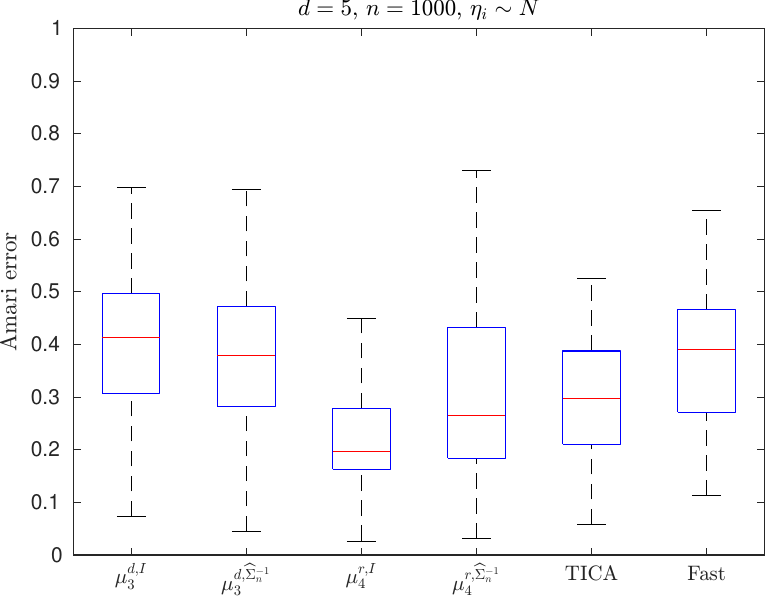}} 
		\end{subfigure}
		\begin{subfigure}{	
				\includegraphics[scale=0.45]{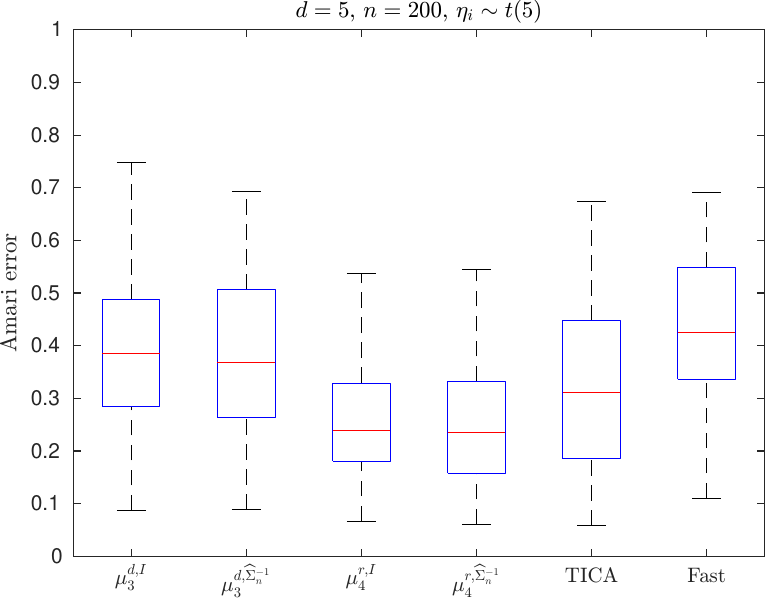}\hfill} 
		\end{subfigure}
		\begin{subfigure}{	
				\includegraphics[scale=0.45]{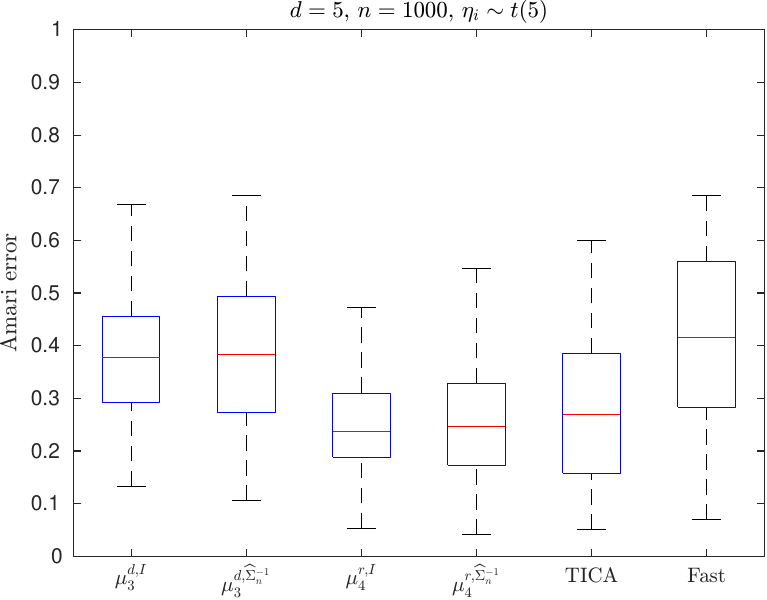}} 
		\end{subfigure}
		\begin{subfigure}{	
				\includegraphics[scale=0.45]{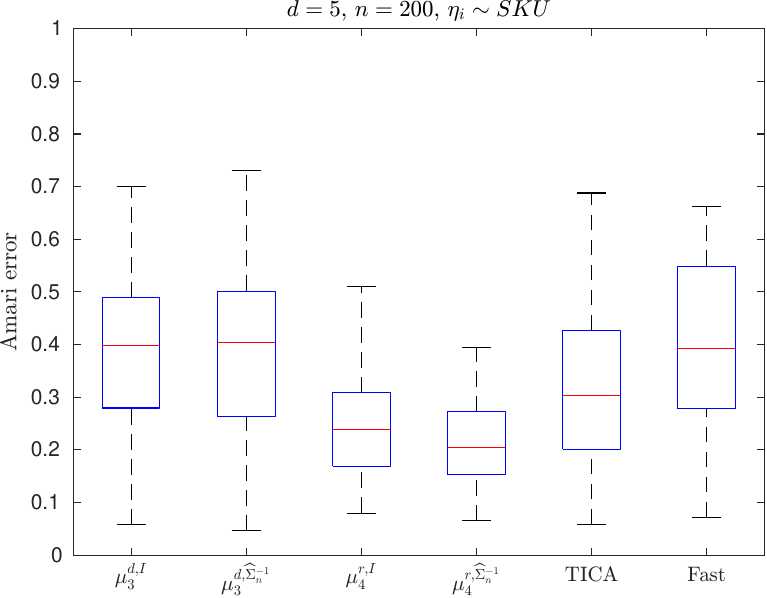}\hfill} 
		\end{subfigure}
		\begin{subfigure}{	
				\includegraphics[scale=0.45]{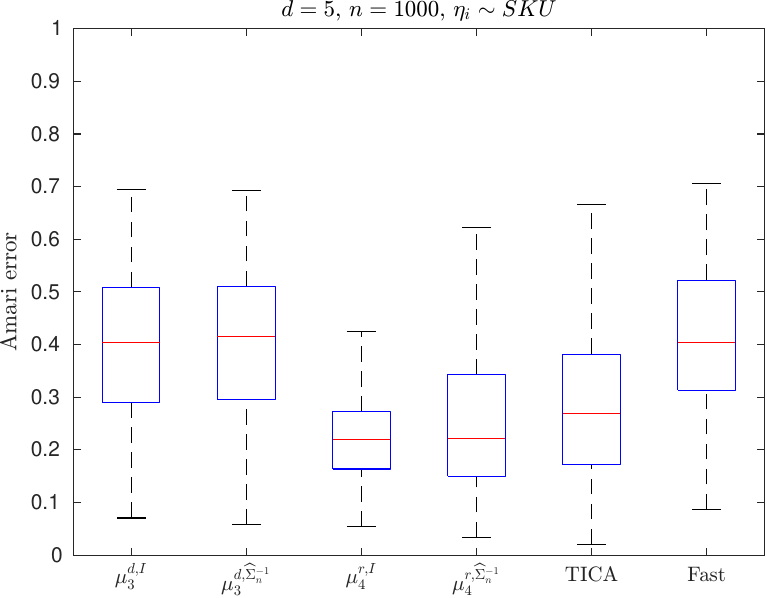}} 
		\end{subfigure}
		\begin{subfigure}{	
				\includegraphics[scale=0.45]{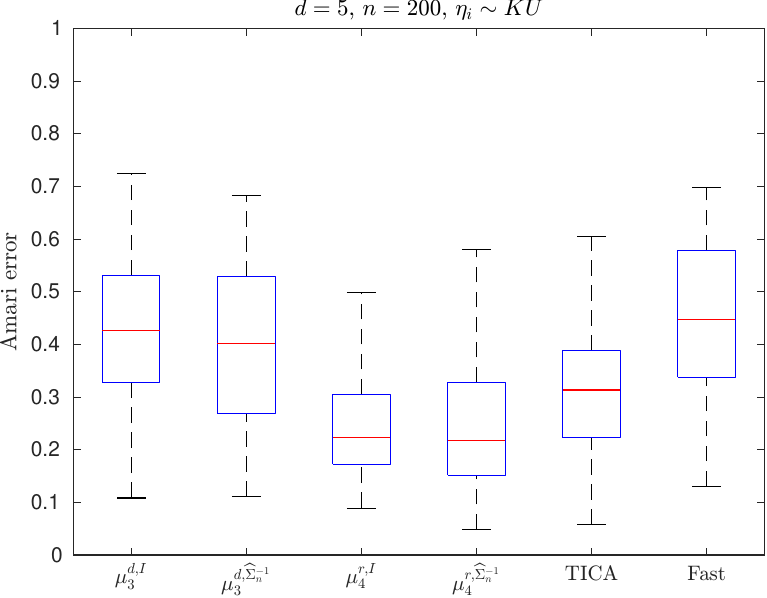}\hfill} 
		\end{subfigure}
		\begin{subfigure}{	
				\includegraphics[scale=0.45]{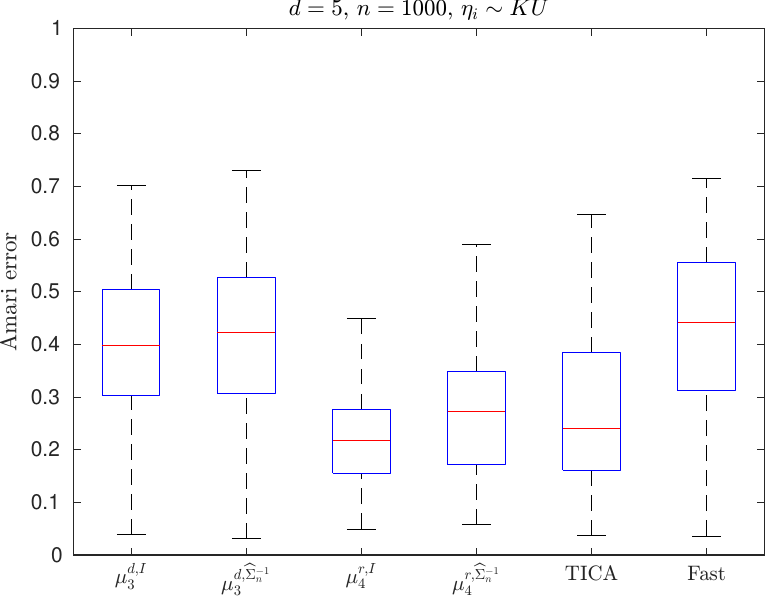}} 
		\end{subfigure}
	\end{center}
	{\begin{minipage}{12.5cm} \footnotesize \emph{Notes:} The figure shows the boxplots for the Amari errors (across $S=100$ simulations) for data sampled from the multiple scaled elliptical components model \eqref{eq:generalmultielliptical}. The different settings for the simulations designs are described in the titles and the $x$-labels indicate the different estimation methods used. \end{minipage} }
\end{figure}

\begin{figure}
	\caption{\sc  Scaled Elliptical Experiments}\label{fig:scaledelliptical2}
	\begin{center}
		\begin{subfigure}{	
				\includegraphics[scale=0.45]{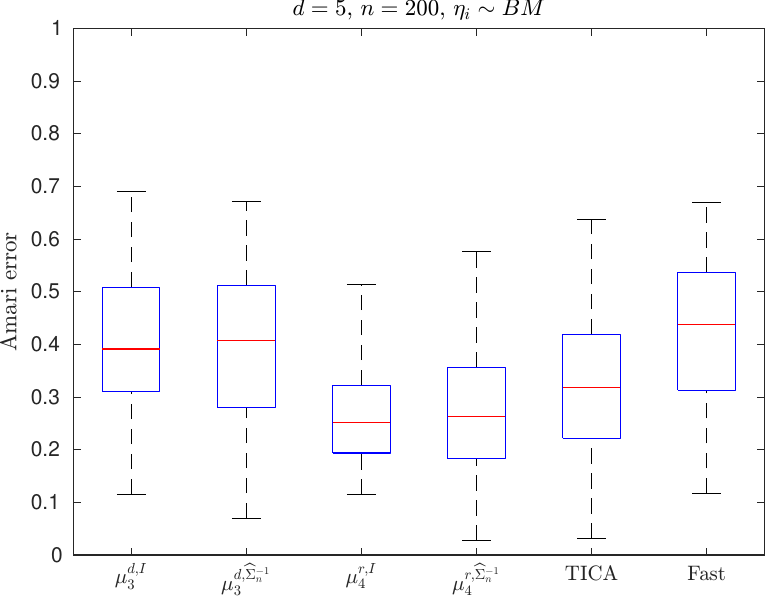}\hfill} 
		\end{subfigure}
		\begin{subfigure}{	
				\includegraphics[scale=0.45]{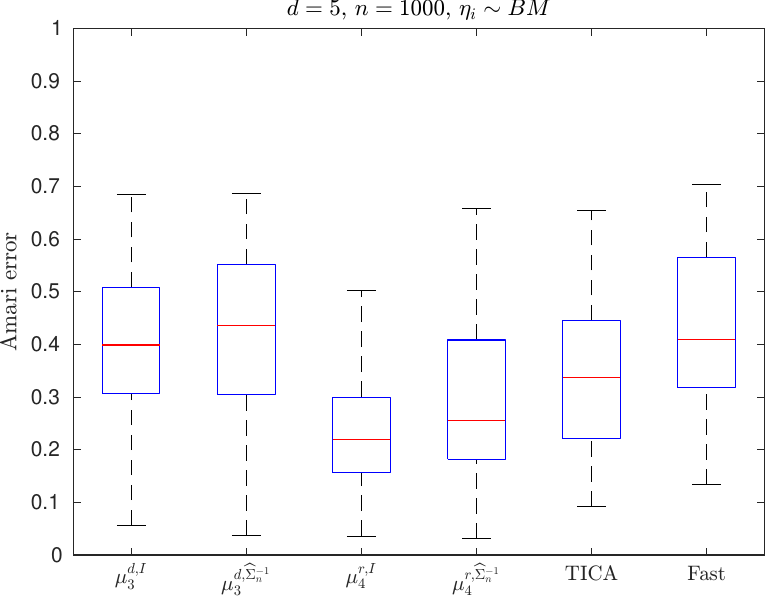}} 
		\end{subfigure}
		\begin{subfigure}{	
				\includegraphics[scale=0.45]{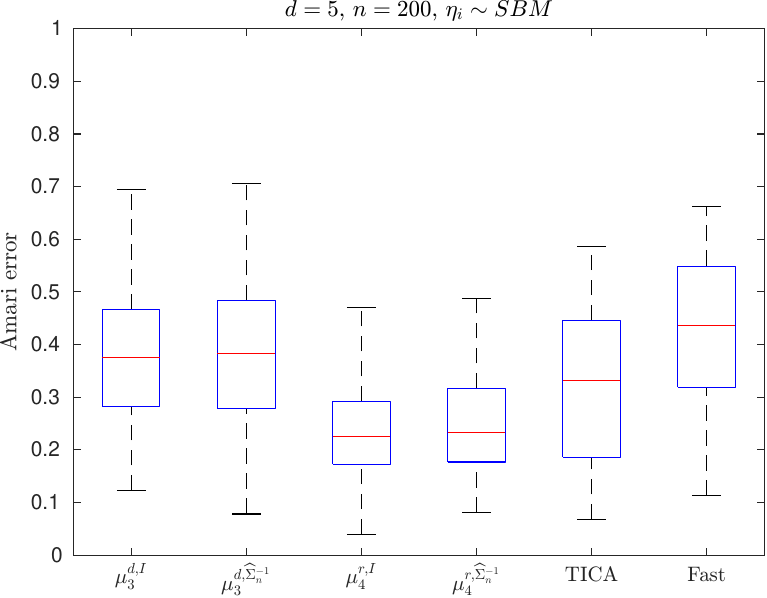}\hfill} 
		\end{subfigure}
		\begin{subfigure}{	
				\includegraphics[scale=0.45]{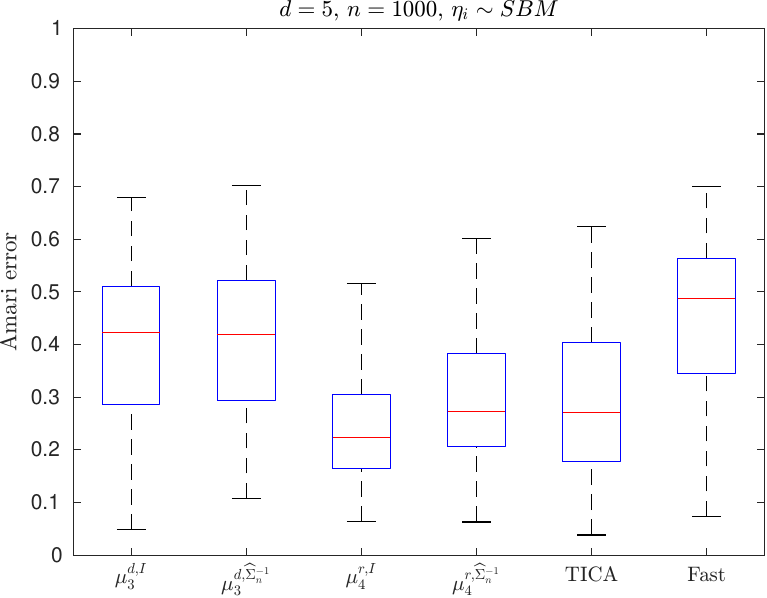}} 
		\end{subfigure}
		\begin{subfigure}{	
				\includegraphics[scale=0.45]{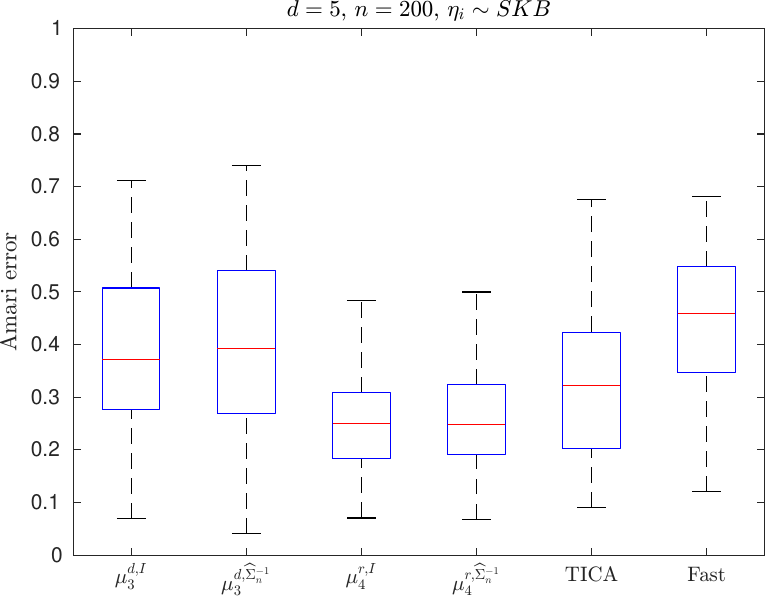}\hfill} 
		\end{subfigure}
		\begin{subfigure}{	
				\includegraphics[scale=0.45]{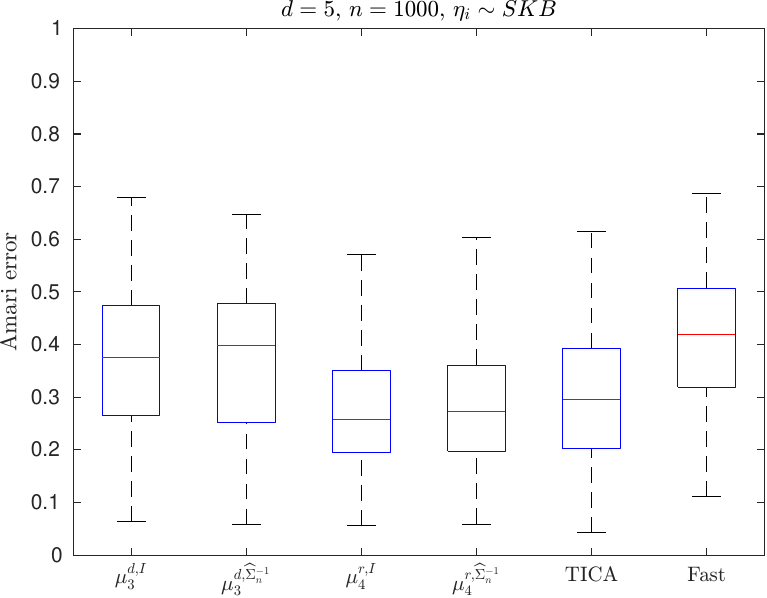}} 
		\end{subfigure}
		\begin{subfigure}{	
				\includegraphics[scale=0.45]{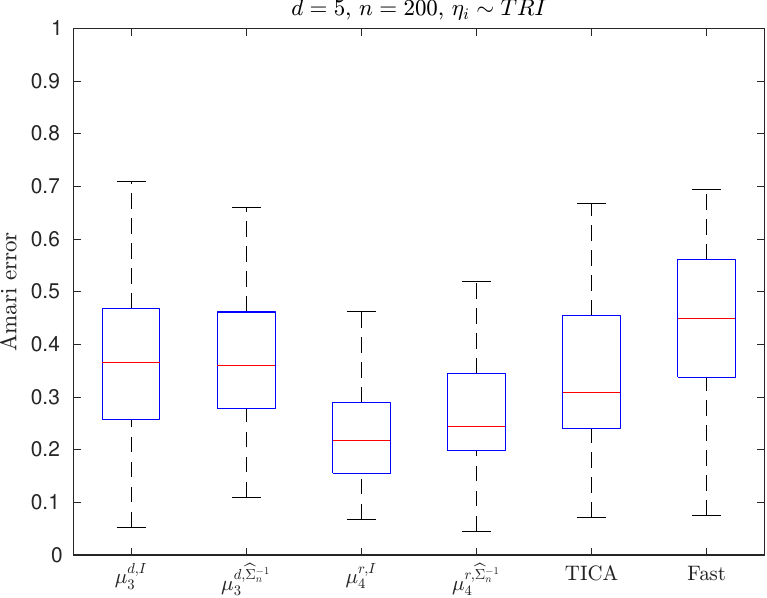}\hfill} 
		\end{subfigure}
		\begin{subfigure}{	
				\includegraphics[scale=0.45]{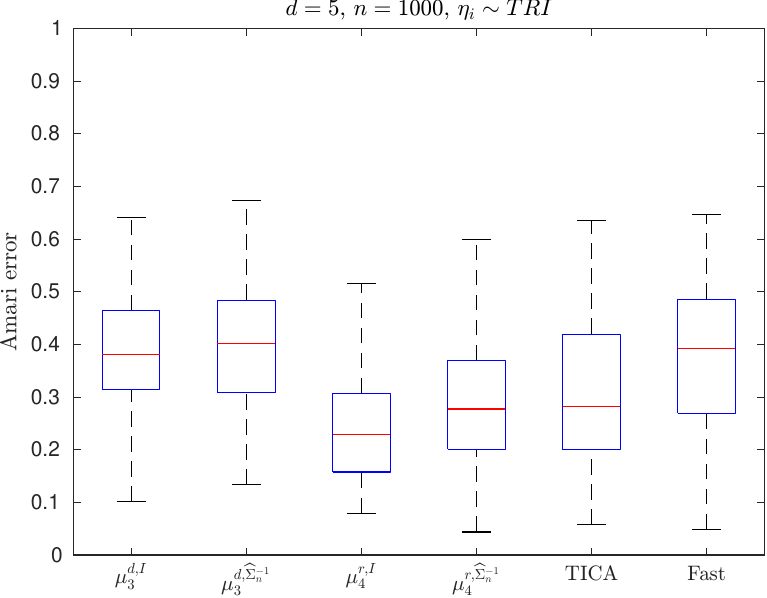}} 
		\end{subfigure}
	\end{center}
	{\begin{minipage}{12.5cm} \footnotesize \emph{Notes:} The figure shows the boxplots for the Amari errors (across $S=100$ simulations) for data sampled from the multiple scaled elliptical components model \eqref{eq:generalmultielliptical}. The different settings for the simulations designs are described in the titles and the $x$-labels indicate the different estimation methods used. \end{minipage} }
\end{figure}

\begin{figure}
	\caption{\sc  Scaled Elliptical Experiments}\label{fig:scaledelliptical3}
	\begin{center}
		\begin{subfigure}{	
				\includegraphics[scale=0.45]{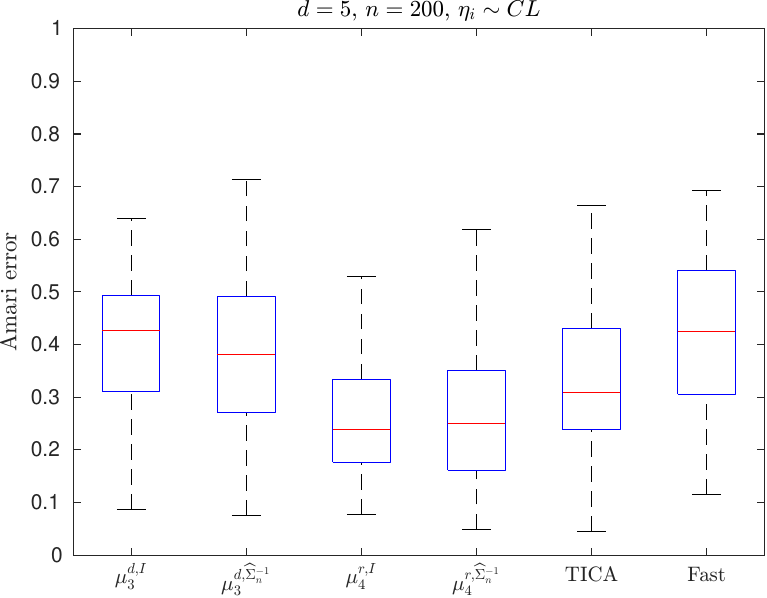}\hfill} 
		\end{subfigure}
		\begin{subfigure}{	
				\includegraphics[scale=0.45]{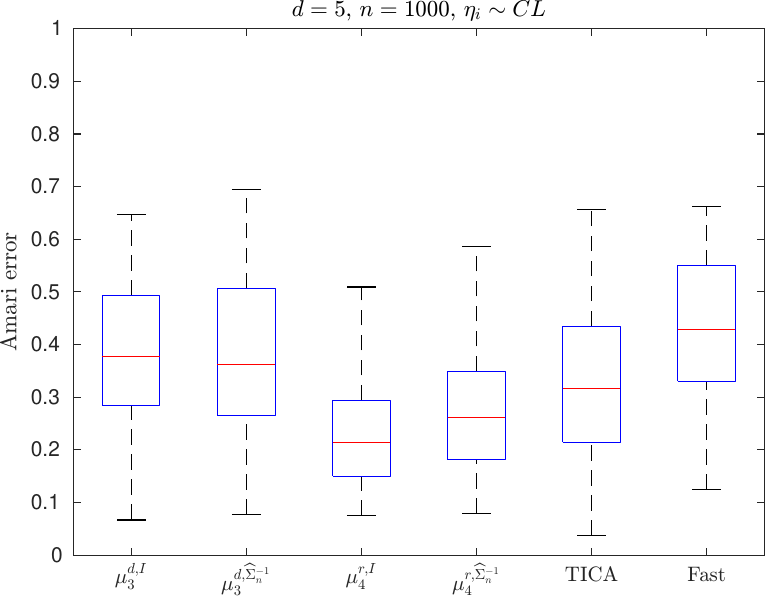}} 
		\end{subfigure}
		\begin{subfigure}{	
				\includegraphics[scale=0.45]{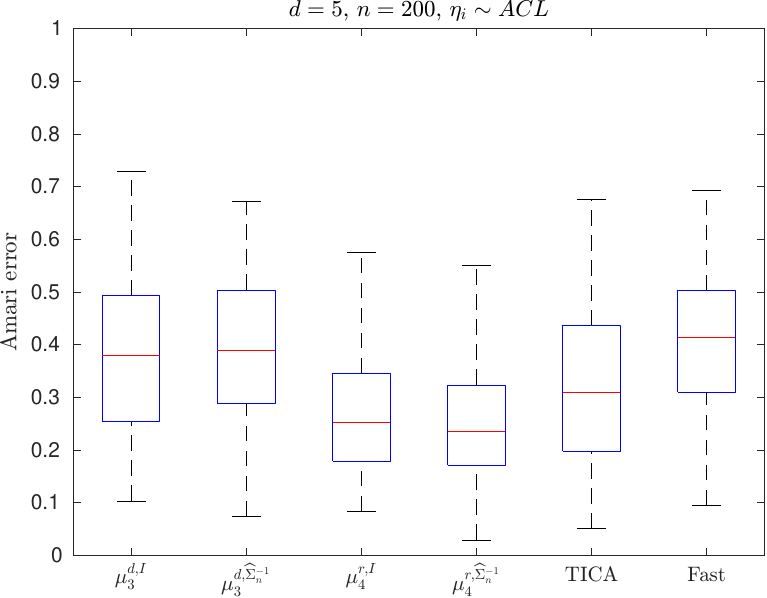}\hfill} 
		\end{subfigure}
		\begin{subfigure}{	
				\includegraphics[scale=0.45]{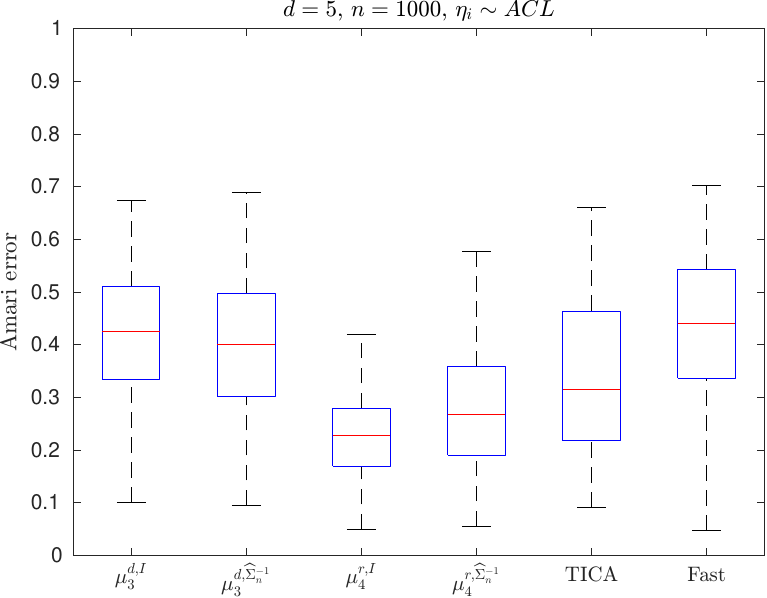}} 
		\end{subfigure}
	\end{center}
	{\begin{minipage}{12.5cm} \footnotesize \emph{Notes:} The figure shows the boxplots for the Amari errors (across $S=100$ simulations) for data sampled from the multiple scaled elliptical components model \eqref{eq:generalmultielliptical}. The different settings for the simulations designs are described in the titles and the $x$-labels indicate the different estimation methods used. \end{minipage} }
\end{figure}

\clearpage

\section{Omitted proofs}\label{appsec:proofsomitted}

\subsection{Proof of Proposition~\ref{prop:hansenJ}}

Let $\tilde A_0 = Q A_0$. Noting that $\hat g_n(\widehat A_{\widehat \Sigma_n^{-1}})$ minimizes $\|\cdot\|^2_{W_n}$ when taking $W_n = \widehat \Sigma_n^{-1}$, we get that $\hat g_n(\widehat A_{\widehat \Sigma_n^{-1}})=0$. Using Taylor's theorem we get that 
\[
0=\widehat \Sigma_n^{-1/2}\sqrt{n}\hat g_n( \widehat A_{\widehat \Sigma_n^{-1}}) = \widehat \Sigma_n^{-1/2}\sqrt{n}\hat g_n(\tilde A_0) + \widehat \Sigma_n^{-1/2} \widehat G(\bar A) \sqrt{n} {\rm vec}(\widehat A_{\widehat \Sigma_n^{-1}} -  \tilde A_0),
\]
where $\bar A$ lies on the segment between $\tilde A_0$ and $\widehat A_{\widehat \Sigma_n^{-1}}$. Pre-multiplying by $\widehat G(\bar A)'\widehat \Sigma_n^{-1/2}$ and rearranging gives 
\[
\sqrt{n} {\rm vec}(\widehat A_{\widehat \Sigma_n^{-1}} -  \tilde A_0) = -	[\widehat G(\bar A)'\widehat \Sigma_n^{-1} \widehat G(\bar A) ]^{-1} \widehat G(\bar A)'\widehat \Sigma_n^{-1} \sqrt{n}\hat g_n(\tilde A_0)~.
\]
Substituting $\sqrt{n} {\rm vec}(\widehat A_{\widehat \Sigma_n^{-1}} -  \tilde A_0)$ back into the expansion above gives 
\[
\widehat \Sigma_n^{-1/2}\sqrt{n}\hat g_n( \widehat A_{\widehat \Sigma_n^{-1}}) = \widehat N \widehat \Sigma_n^{-1/2} \sqrt{n}\hat g_n(\tilde A_0) 
\]
where 
\[
\widehat N = I_{d_g} - \widehat \Sigma_n^{-1/2} \widehat G(\bar A) [\widehat G(\bar A)'\widehat \Sigma_n^{-1} \widehat G(\bar A) ]^{-1} \widehat G(\bar A)'\widehat \Sigma_n^{-1/2}  ~.
\]
By the discussion preceding \eqref{eq:thisissigma}, we have $\Sigma^{-1/2} \sqrt{n}\hat g_n(\tilde A_0) \stackrel{d}{\to} Z \sim N(0,I_{d_g})$. Note that this random variable differs from $\widehat\Sigma^{-1/2}_n \sqrt{n}\hat g_n(\tilde A_0) \stackrel{d}{\to} Z \sim N(0,I_{d_g})$ only by something that converges to zero in probability, as $\widehat \Sigma_n \stackrel{p}{\to} \Sigma$. By Slutsky's lemma we have $\widehat\Sigma^{-1/2} \sqrt{n}\hat g_n(\tilde A_0) \stackrel{d}{\to} Z \sim N(0,I_{d_g})$, and from Proposition~\ref{prop:consist}, equation \eqref{eq:uniformder} and $\widehat \Sigma_n \stackrel{p}{\to} \Sigma$ and the continuous mapping theorem, we get 
\begin{equation}\label{eq:N}
\widehat N \stackrel{p}{\to} N = I_{d_g} -  \Sigma^{-1/2}  G(\tilde A_0)  [G(\tilde A_0)' \Sigma^{-1} G(\tilde A_0) ]^{-1} G(\tilde A_0)' \Sigma^{-1/2} ~. 
\end{equation}
We note that $N$ is a projection matrix of rank $d_g - d^2$. Combining we get 
\begin{align*}
\hat L_{\widehat \Sigma_n^{-1}}(\widehat A_{\widehat \Sigma_n^{-1}})&= \left(\widehat \Sigma_n^{-1/2}\sqrt{n}\hat g_n( \widehat A_{\widehat \Sigma_n^{-1}})\right)' \left(\widehat \Sigma_n^{-1/2}\sqrt{n}\hat g_n( \widehat A_{\widehat \Sigma_n^{-1}})\right) \\
&\stackrel{d}{\to} Z' N Z \sim \chi^2(d_g -d^2)~,  
\end{align*}
where the last step follows from \citet[][page 186]{Rao1973}. 

\subsection{Proof of Proposition~\ref{prop:subset}}

From the proof of Proposition~\ref{prop:hansenJ} we have 
\[
\widehat \Sigma_n^{-1/2}\sqrt{n}\hat g_n( \widehat A_{\widehat \Sigma^{-1}}) =  N \Sigma^{-1/2} \sqrt{n}\hat g_n(\tilde A_0) + o_p(1),
\]
where are $N$ is the projection matrix defined in \eqref{eq:N}. Let $\hat g_{1,n}$, $G_1$, $N_1$ be the equivalent quantities to  $\hat g_{n}$, $G$, $N$ just computed for the smaller set of identifying restrictions. Using similar arguments we get 
\begin{align*}
\widehat \Sigma_{11}^{-1/2}\sqrt{n}\hat g_{1,n}( \widehat A_{\widehat \Sigma_{11}^{-1}}) =&   N_1   \Sigma_{11}^{-1/2} \sqrt{n}\hat g_{1,n}(\tilde A_0) + o_p(1) \\
=&   N_1  \Sigma_{11}^{-1/2} [ I_{d_{g_1}} : 0_{d_{g_1} \times d_g} ]     \Sigma^{1/2}    \Sigma^{-1/2} \sqrt{n}\hat g_n(\tilde A_0) \\
&+ o_p(1) ~.
\end{align*}	
% 	where $N_1 = I_{d_{g_1}} - P_1$ with $P_1 = \Sigma_{11}^{-1/2}  G_1(\tilde A_0) [G_1(\tilde A_0)' \Sigma_{11}^{-1} G_1(\tilde A_0) ]^{-1}  G_1(\tilde A_0)' \Sigma_{11}^{-1/2}$ and the partial Jacobian estimate is given by $ G_1(A) = \nabla_{{\rm vec}(A)'} m_{1}(A)$. 

Define $\Xi =  \Sigma_{11}^{-1/2} [ I_{d_{m_1}} : 0_{d_{g_1} \times d_g} ] \Sigma^{1/2}$ and $J = N_1 \Xi$. Note that $N$ is idempotent and set $B \equiv J'J = \Xi' N_1 \Xi$. We show that (i) $N - B$ is idempotent and (ii) $N -B$ has rank $d_g - d_{g_1}$. First, letting $N = I_{d_g} - P$ with $P = \Sigma^{-1/2}  G(\tilde A_0)  [G(\tilde A_0)' \Sigma^{-1} G(\tilde A_0) ]^{-1} G(\tilde A_0)' \Sigma^{-1/2}$, we have 
\begin{align*}
BN &= B - B P (P'P)^{-1} P' \\  
&= B - \Xi' N_1 \Xi P (P'P)^{-1} P'~, 
\end{align*}
and $N_1 \Xi P = N_1 P_1 = 0$, such that $BN = B$. Using similar step we find that $NB = B$. Finally, consider $BB$ for which we have 
\begin{align*}
BB &= \Xi' N_1 \Xi \Xi' N_1 \Xi \\  
&=  \Xi' N_1 \Sigma_{11}^{-1/2} \Sigma_{11} \Sigma_{11}^{-1/2} N_1 \Xi \\ 
&=  \Xi' N_1 \Xi = B 
\end{align*}
Combining we get that $(N - B)(N - B) = N - B$. For (ii) note that since $N - B$ is idempotent we have ${\rm rank}(N - B) = {\rm Tr}(N - B) = d_g - d_{g_1}$. 
To complete the proof note that 
\begin{align*}
C_n &= \sqrt{n}\hat g_n(\tilde A_0)' \widehat \Sigma_n^{-1/2'} [N - B]\widehat \Sigma_n^{-1/2}  \sqrt{n}\hat g_n(\tilde A_0) + o_p(1)\\
&\stackrel{d}{\to} Z' [N - B] Z \sim \chi^2(d_g - d_{g_1})~.
\end{align*}

\newpage 

\bibliographystyle{imsart-nameyear}
\bibliography{../../bib_mtp2}